\documentclass[12pt,a4paper,reqno]{amsart}
\usepackage{cmap}
\usepackage[T1]{fontenc}
\usepackage[section]{placeins}
\usepackage{graphicx}
\usepackage{algorithm} 
\usepackage{algpseudocode} 
\usepackage{subcaption}
\usepackage{multirow}

\usepackage{booktabs}
\usepackage{setspace,mathrsfs,amsthm,amsmath,amssymb,amsfonts,lipsum,appendix,mathtools,cite,tabularx,fontenc,bbm,caption}

\usepackage[shortlabels]{enumitem}
\makeatletter
\newtheorem*{rep@theorem}{\rep@title}
\newcommand{\newreptheorem}[2]{%
	\newenvironment{rep#1}[1]{%
		\def\rep@title{#2 \ref{##1}}%
		\begin{rep@theorem}}%
		{\end{rep@theorem}}}
\makeatother

\newtheorem{theorem}{Theorem}[section]

\usepackage[nointegrals]{wasysym}
\usepackage[misc]{ifsym}
\usepackage[dvipsnames]{xcolor}
\definecolor{ao}{rgb}{0.0, 0.5, 0.0}
\definecolor{lasallegreen}{rgb}{0.0, 0.3, 0.0}
\usepackage{hyperref,tikz}
\hypersetup{
	colorlinks=true,
	linkcolor=blue,
	filecolor=magenta,      
	urlcolor=violet,
	citecolor=OrangeRed,
}

\usepackage[
  hmarginratio={1:1},     
  vmarginratio={2:1},     
  textwidth=400pt,        
  heightrounded,          
]{geometry}
\usetikzlibrary{calc,patterns,angles,quotes,intersections}

\usepackage{wrapfig}

\allowdisplaybreaks

\newcommand{\x}[0]{\boldsymbol{x}}
\newcommand\blfootnote[1]{%
	\begingroup
	\renewcommand\thefootnote{}\footnote{#1}%
	\addtocounter{footnote}{-1}%
	\endgroup
}

\newtheorem{lemma}{Lemma}[section]

\newtheorem{ass}{Assumption}

\newtheorem{remark}{Remark}[section]
\newcommand{\mb}[1]{\mathbf{#1}}
\newcommand{\tv}{X}
\newcommand{\tw}{Y}

\newcommand{\norm}[1]{\vert\vert #1\vert\vert}

\usepackage[hyperpageref]{backref}

\usepackage{orcidlink}

\usepackage{a4wide}
\usepackage{csquotes}

\usepackage{amsfonts}
\usepackage{amssymb,amsthm}
\usepackage{amsmath}
\allowdisplaybreaks
\usepackage{mathtools}
\mathtoolsset{showonlyrefs}

\numberwithin{equation}{section}

\usepackage {setspace}

\usepackage{enumitem}

\setlist{nosep}

\usepackage{todonotes}

\begin{document}
\singlespacing

\title[Space-time IgA for biharmonic wave equation]{An unconditionally stable space-time isogeometric method for a biharmonic wave equation}

\author[Shreya Chauhan]{Shreya Chauhan$^{1,*}$}
\author[Sudhakar Chaudhary]{Sudhakar Chaudhary$^1$}
\blfootnote{$^{*}$Corresponding author.}
\blfootnote{$^1$Department of Basic Sciences (Mathematics), Institute of Infrastructure, Technology, Research and Management, Gujarat, India.}

\blfootnote{%
E-mail addresses:
\href{mailto: shreyachauhan898@gmail.com}{ shreyachauhan898@gmail.com},
\href{mailto:sudhakarchaudhary@iitram.ac.in}{dr.sudhakarchaudhary@gmail.com}
}


\begin{abstract}
	This work presents a space-time isogeometric analysis of biharmonic wave problem, in contrast to the more common application of space-time methods to second-order wave equations. We first establish the unique solvability of the continuous space-time variational formulation. In order to obtain $H^2$- conforming discretization of the biharmonic wave equation, we consider globally smooth B-spline functions having continuity higher than $C^0$. We prove that the resulting space-time discrete formulation is stable under a Courant–Friedrichs–Lewy (CFL) condition. Furthermore, we propose a stabilized formulation, achieved by adding a non-consistent penalty term, which yields unconditional stability. Exploiting the tensor product structure, an efficient direct solver is also provided for solving the linear system arising from the discrete formulation. A few numerical experiments are presented to demonstrate the stability and convergence properties of the proposed scheme as well as the efficiency of the proposed solver.
	\end{abstract}

\maketitle
\definecolor{lblack}{gray}{0.3}
\definecolor{mygray}{gray}{0.9}
\definecolor{vlgray}{gray}{0.96}
\definecolor{medgray}{gray}{0.8}
\definecolor{dgray}{gray}{0.7}
\begin{quote}	
	\setcounter{tocdepth}{1}
	\tableofcontents
	\addtocontents{toc}{\vspace*{0ex}}
\end{quote}
 \section{Introduction}In this work, we focus on the following biharmonic wave equation:
\begin{equation}\label{eqbiwave}
	\partial_{tt}u+\Delta^2u=f\quad\mbox{in }Q:=\Omega\times(0,T),
\end{equation}
with initial conditions
\begin{equation}\label{inicon}
	u(\x,0)=0\mbox{ and }\partial_tu(\x,0)=0\quad\mbox{in }\Omega
\end{equation}
and clamped boundary conditions 
\begin{equation}\label{cbc}
	\displaystyle u=\frac{\partial u}{\partial n}=0\quad\mbox{on }\Sigma:=\partial\Omega\times(0,T),
\end{equation}
where $\Omega\subset\mathbb{R}^d, d=1,2$ is a bounded domain with smooth boundary $\partial\Omega$, $f$ is a given function, $\partial_tu$ and $\partial_{tt}u$ are first and second order time derivatives respectively, $n$ is the outward unit normal to the boundary $\partial\Omega$ and $\displaystyle\frac{\partial u}{\partial n}$ is the outer normal derivative of $u$ on $\partial\Omega$. Problem \eqref{eqbiwave} arises in many real-world phenomena such as plate bending and elasticity of thin plates. This equation also serves as a foundational model for more complex nonlinear problems, such as the Euler-Bernoulli equation which models the deflection of viscoelastic plates \cite{thinplates}. Further applications include the analysis of interacting water waves \cite{waterwaves}, the idealization of the suspension bridge \cite{LAZER1987243}, vibration of beams \cite{vibrationbeam}.\\\\
The numerical study of the fourth-order problems has been an attractive area of interest among the researchers. Typically for the fourth-order problems, the conforming Finite Element Methods (FEM) require $C^1$-continuous basis functions to obtain $H^2$-conforming finite dimensional spaces. While such conforming elements exist, such as the Argyris and Bogner-Fox-Schmit element, their practical implementation is often very challenging. Apart from the conforming FEMs, other alternatives available in the literature are nonconforming FEM using the Morley element \cite{morley}, discontinuous Galerkin FEM \cite{dGFEM}, $C^0$ interior penalty method \cite{C0interior} or the mixed FEM \cite{DAS202452} to name a few. A variety of finite element strategies have been employed for the numerical solution of fourth-order wave equations.  A mixed finite element formulation, which reduces the fourth-order problem to a system of second-order equations, was rigorously analyzed for the case of Neumann boundary conditions in \cite{HE20131}. Authors in \cite{TAO2021113230} developed the ultra-weak discontinuous Galerkin method for the discretization of nonlinear fourth-order wave equation and proved the $L^2$ error estimates. In contrast, a direct $C^1$-conforming approach was analyzed in \cite{C1bihar}, where the authors developed a Galerkin-collocation method utilizing the Bogner-Fox-Schmit element to naturally satisfy the smoothness requirements of problem \eqref{eqbiwave}. He et al. \cite{HE2023333} constructed a two-level discretization for fourth-order wave equation based on the Crank-Nicolson method and Lagrange finite elements, demonstrating its energy-conserving characteristics. Recently, in \cite{nataraj} a lowest-order nonstandard finite element method for problem \eqref{eqbiwave} was analyzed using both explicit and implicit time discretization schemes. \\\\
As an alternative to the classical FEM, Isogeometric Analysis (IgA) has been widely used for solving fourth-order partial differential equations, due to its capacity for constructing $H^2$-conforming discretization with ease. Introduced by Hughes et al. \cite{IgAIntro}, IgA was developed to bridge the gap between Computer-Aided Design (CAD) and numerical simulation. Its fundamental principle is to use the same smooth basis functions—typically B-splines or NURBS—for both representing the computational geometry and approximating the solution. A key advantage of IgA is the ability to represent complex geometries exactly, providing a significant improvement over the classical FEM. Furthermore, beyond standard $h$- and $p$-refinement, IgA features a unique $k$-refinement process, which increases the smoothness of the basis functions across element boundaries. This yields globally $C^k$-continuous functions, enabling the conforming discretization of the higher-order problems with standard Galerkin method \cite{IgAhigh}. Recently, in \cite{Sogn2023,MANNI2023116314,Meng2024} the IgA has been analyzed for various fourth-order problems.\\\\
The numerical analysis of time-dependent problems is primarily divided into two frameworks: time-stepping schemes and space-time schemes. The conventional time-stepping approach consists of discretizing the partial differential equations (PDEs) separately in the space (using the FEM or IgA) and in time (using finite difference scheme). In contrast, space-time methods treat the temporal variable as an additional spatial variable, enabling an altogether discretization of the entire space-time cylinder. This paradigm offers several significant advantages, including the capacity for simultaneous space-time adaptivity, a natural framework for handling problems with moving boundaries, and enhanced potential for parallelization across both spatial and temporal domains. Within this space-time framework, various finite element methodologies have been developed and analyzed, such as the discontinuous Galerkin method, the continuous Galerkin method, and the Petrov-Galerkin method; we refer to \cite{Steinbachreview} for a comprehensive overview. The space-time isogeometric analysis provides additional benefit of more accurate solution due to the use of smooth basis functions in time variable also. In \cite{LANGER2016342}, the space-time isogeometric method was first analyzed for the parabolic evolution problems. Space-time isogeometric Galerkin method and the space-time least square method for second order parabolic problem have been proposed in \cite{stigasteinbach} and \cite{leastsquare} respectively. The space-time isogeometric discretization of the time-dependent problems give rise to a large matrix system which requires a lot of memory and computational time. To overcome this, various efficient solvers have been proposed and analyzed in \cite{stigasteinbach,leastsquare,Chaudhary2026}.\\\\
In recent years, various space-time methods for second-order wave equation are proposed and analyzed. Henning et al. \cite{Henning2022} derived an ultra-weak formulation via double integration by parts in space and time and analyzed Petrov-Galerkin schemes for the same. A Hilbert transformation-based Galerkin-Bubnov FEM for second-order wave equation was proposed in \cite{coer}, where the well-posedness of the continuous formulation was established and the discrete scheme was proven to be stable under a Courant–Friedrichs–Lewy (CFL) condition. Steinbach et al. \cite{Steinbach2019} and Zank \cite{Zank2021} introduced stable formulations by adding a projection term to the discrete formulation. The framework of \cite{coer,Steinbach2019} was extended to space-time isogeometric analysis in \cite{Fraschini2024}, which adopted a non-consistent penalty term for stabilization, differing from the approach in \cite{Steinbach2019}. The stability of this formulation was further investigated through a matrix-based analysis in \cite{Ferrari2025}. Additional innovative space-time methods for wave propagation are discussed in \cite{Bignardi2025, arxivwave, arxivwave2}. However, despite this considerable progress for the classical second-order wave equation, the literature lacks an analysis of space-time methods for its higher-order counterpart, the biharmonic wave equation \eqref{eqbiwave}.\\\\
Motivated by above, this work presents the analysis of a space-time isogeometric method for the biharmonic wave equation \eqref{eqbiwave}. To the best of our knowledge, this is the first work that analyzes the biharmonic wave equation \eqref{eqbiwave} in space-time isogeometric framework. The primary contributions of this study are as follows:
\begin{itemize}
	\item We derive a space-time variational formulation for problem \eqref{eqbiwave} by applying integration by parts in space and time. The well-posedness of the continuous variational problem is established, following the analytical framework developed for second-order wave equations in \cite{coer}.
	\item We propose a space-time isogeometric discretization of \eqref{eqbiwave} using globally smooth splines basis functions, which naturally provide the required $H^2$-conformity in the spatial domain.
	\item In \cite{coer,Fraschini2024}, the stability analysis (CFL-type condition) of the space-time discretization to the second-order wave equation is extensively studied, however such analysis for the space-time discretization of fourth-order wave equation \eqref{eqbiwave} is not available in the literature. In this work, a CFL condition of the form $h_t < C_{\Omega} h_s^2$ for the space-time isogeometric discretization of \eqref{eqbiwave} is derived, where $h_t$, $h_s$ are temporal and spatial mesh sizes respectively and $C_{\Omega}$ is a constant depending on the domain $\Omega$.
	\item Building upon the ideas of \cite{Fraschini2024}, we provide a stabilized space-time isogeometric discretization for the fourth-order problem \eqref{eqbiwave} by adding a non-consistent penalty term. The proposed stabilization does not required any CFL restriction on the mesh sizes $h_s$ and $h_t$.
	\item By exploiting the tensor product structure of the isogeometric spaces, we analyze an efficient direct solver for the linear system arising from the stabilized isogeometric formulation of \eqref{eqbiwave}. 
\end{itemize}
The remainder of the paper is organized as follows: In Section \ref{sec:pre}, we recall some fundamental results regarding the second-order ordinary differential equation, which we will further use in our analysis. Section \ref{sec:stvf continuous} is devoted to the space-time variational formulation of \eqref{eqbiwave} and the well-posedness results. In Section \ref{sec4:stiga}, the space-time isogeometric discrete formulation is obtain which is shown to be stable under a CFL condition. We also propose a stabilized discrete formulation in this section. In Section \ref{sec:solver}, we develop an efficient solver to solve the linear system obtained from the stabilized space-time isogeometric formulation. Finally, in Section \ref{sec:numerical}, we provide results of some numerical experiments to support the stability and convergence properties of the proposed scheme and the computational efficiency of the linear solver.\\\\
\noindent\textbf{Notations:} By $H^m(\Omega)$, we denote the Sobolev space of order $m$ equipped with norm $\vert\vert u\vert\vert_{H^m(\Omega)}=\left(\displaystyle\sum_{\alpha\leq m}\vert\vert D^{\alpha}u\vert\vert^2_{L^2(\Omega)}\right)^{1/2}$. We consider $H^2_0(\Omega)$, a subspace of $H^2(\Omega)$ of functions having zero trace and equipped with the norm $\vert\vert u\vert\vert_{H^2_0(\Omega)}$=$\vert\vert\Delta u\vert\vert_{L^2(\Omega)}$. By $V'$ we denote the dual space of the Banach space $V$.  For a Banach space $V$, the Bochner space is defined as $L^2(V)=\displaystyle\left\{u:(0,T)\rightarrow V:\int_{0}^{T}\vert\vert u\vert\vert_V^2\ dt<\infty\right\}$.  To derive the space-time weak formulation of \eqref{eqbiwave}, we consider the following spaces:
\begin{equation*}
	X=H^{2,1}_{0;0,}(Q):=L^2(H^2_0(\Omega))\cap H^1_{0,}(L^2(\Omega))
\end{equation*}
and
\begin{equation*}
	Y=H^{2,1}_{0;,0}(Q):=L^2(H^2_0(\Omega))\cap H^1_{,0}(L^2(\Omega)),
\end{equation*}
where 
\begin{equation*}
	H^1_{0,}(L^2(\Omega))=\{v\in H^1(L^2(\Omega)):v(\x,0)=0\mbox{ in }\Omega\}
\end{equation*}
and
\begin{equation*}
	H^1_{,0}(L^2(\Omega))=\{v\in H^1(L^2(\Omega)):v(\x,T)=0\mbox{ in }\Omega\}.
\end{equation*}
The norms on the spaces $X$ and $Y$ are given as follow:
\begin{equation*}
	\norm{u}_X^2=\norm{\partial_tu}^2_{L^2(L^2(\Omega))}+\norm{\Delta u}^2_{L^2(L^2(\Omega))}
\end{equation*}
and\begin{equation*}
	\norm{u}_Y^2=\norm{\partial_tu}^2_{L^2(L^2(\Omega))}+\norm{\Delta u}^2_{L^2(L^2(\Omega))}
\end{equation*}
We define the spaces of functions depending only on time variable as follows:
\begin{equation*}
    H^1_{0,}(0,T)=\{v\in H^1(0,T):v(0)=0\}\mbox{ and } H^1_{,0}(0,T)=\{v\in H^1(0,T):v(T)=0\}
\end{equation*}equipped with the norm $\vert\vert v\vert\vert_{H^1_{0,}(0,T)}=\vert\vert v\vert\vert_{H^1_{,0}(0,T)}=\vert\vert\partial_t v\vert\vert_{L^2(0,T)}=\vert v\vert_{H^1(0,T)}$.

 \section{Preliminaries and foundational results}\label{sec:pre}
Before proceeding with the analysis of \eqref{eqbiwave}, we first  recall several essential theorems and lemmas concerning second-order ordinary differential equations (ODEs). These foundational results will be directly applied in the development of the space-time method for \eqref{eqbiwave}.\\
We consider the following ODE: 
\begin{equation}\label{ode}
	\partial_{tt}\varphi(t)+\mu \varphi(t)=f(t)\quad\mbox{for }t\in(0,T),\quad \varphi(0)=\partial_t\varphi(0)=0,
\end{equation}
 where $\mu>0$. To obtain the space-time variational formulation of \eqref{ode}, we multiply \eqref{ode} with a test function $v\in H^1_{ ,0}(0,T)$  and apply integration by parts once, to get the variational formulation as follows: find $\varphi\in H^1_{0, }(0,T)$ such that
\begin{equation}\label{odevf}
	-\int_{0}^{T}\partial_t\varphi(t)\partial_tv(t)\ dt+\mu\int_{0}^{T}\varphi(t)v(t)\ dt=\langle f,v\rangle_{(0,T)}\quad\mbox{for all }v\in H^1_{,0}(0,T), 
\end{equation}
where $f\in(H^1_{,0}(0,T))'$ is given. The unique solvability of \eqref{odevf} is studied by Steinbach et al. \cite{coer} and for this purpose, they considered the transformation $\overline{\mathcal{H}}_T:H^1_{0,}(0,T)\rightarrow H^1_{,0}(0.T)$ defined by
\begin{equation*}
	(\overline{\mathcal{H}}_Tv)(t):=v(T)-v(t),\quad t\in(0,T).
\end{equation*}
The map $\overline{\mathcal{H}}_T$ is an isometry with respect to $\vert\cdot\vert_{H^1(0,T)}$. Hence, the well-posedness of \eqref{odevf} is equivalent to the well-posedness of the following variational formulation: find $\varphi\in H^1_{0,}(0,T)$ such that
\begin{equation}\label{odevfh}
	-\int_{0}^{T}\partial_t\varphi(t)\partial_t(\overline{\mathcal{H}}_Tv)(t)\ dt+\mu\int_{0}^{T}\varphi(t)(\overline{\mathcal{H}}_Tv)(t)\ dt=\langle f,\overline{\mathcal{H}}_Tv\rangle_{(0,T)}\quad\mbox{for all }v\in H^1_{0,}(0,T).
\end{equation}
The well-posedness of \eqref{odevfh} can be proven by utilizing the Fredholm theory. In particular, we have the following result.
\begin{theorem}\label{wellposode}\cite{coer}
	Let $f\in (H^1_{,0}(0,T))'$. Then there exists a unique $\varphi\in H^1_{0,}(0,T)$ which satisfies \eqref{odevfh} and 
	\begin{equation*}
		\vert{\varphi}\vert_{H^1(0,T)}\leq \left(\frac{2+\sqrt{\mu}T}{2}\right)\norm{f}_{(H^1_{,0}(0,T))'}.
	\end{equation*}
\end{theorem}When considering $f\in (H^1_{,0}(0,T))'$, as discussed in Theorem \ref{wellposode}, the bound on $\varphi$ depends explicitly on $\mu.$ However, authors in \cite{coer} proved an estimate independent of $\mu$ when considering $f\in L^2(0,T)$ which is given in the following Lemma.
\begin{lemma}\label{odebound}\cite{coer}
    Let $f\in L^2(0,T)$. Then the unique solution $\varphi$ of \eqref{odevfh} satisfies the following
    \begin{equation*}
        \vert{\varphi}\vert^2_{H^1(0,T)}+\mu\norm{\varphi}^2_{L^2(0,T)}\leq \frac{1}{2}T^2\norm{f}^2_{L^2(0,T)}.
    \end{equation*}
\end{lemma}
Next, we discuss the numerical discretization of \eqref{odevf}. For finite dimensional subspaces $X_h\subset H^1_{0, }(0,T)$ and $Y_h\subset H^1_{ ,0}(0,T)$, the discrete variational formulation of \eqref{odevf} is given as follows : find $\varphi_h\in X_h$ such that
\begin{equation}\label{odedisvf}
	-\int_{0}^{T}\partial_t\varphi_h(t)\partial_tv_h(t)\ dt+\mu\int_{0}^{T}\varphi_h(t)v_h(t)\ dt=\langle f,v_h\rangle_{(0,T)}\quad\mbox{for all }v_h\in Y_h. 
\end{equation}
The basis functions of $X_h$ and $Y_h$ can be linear hat basis functions of classical FEM or the B-spline basis functions with maximal regularity of IgA. We can write the discrete variational formulation \eqref{odedisvf} as follows: find $\varphi_h\in X_h$ such that
\begin{equation}\label{odedisbf}
	a(\varphi_h,v_h)=\langle f,v_h\rangle_{(0,T)}\quad\mbox{for all }v_h\in Y_h,
\end{equation}
where $\displaystyle a(\varphi_h,v_h)=-\int_{0}^{T}\partial_t\varphi_h(t)\partial_tv_h(t)\ dt+\mu\int_{0}^{T}\varphi_h(t)v_h(t)\ dt$. As for the continuous case,
the variational formulation \eqref{odedisbf} is equivalent to the following variational formulation: find $\varphi_h\in X_h$ such that
\begin{equation}\label{odedisbfh}
	a(\varphi_h,\overline{\mathcal{H}}_Tv_h)=\langle f,\overline{\mathcal{H}}_Tv_h\rangle_{(0,T)}\quad\mbox{for all }v_h\in X_h.
\end{equation}
When space $X_h$ consists of B-spline basis functions of maximal regularity, the well-posedness of \eqref{odedisbfh} can be proven using the theory of Galerkin method applied to G\r{a}rding-type problems \cite{Spence2015,Moiola2021Scattering}.
 This result is stated in the following theorem.
\begin{theorem}\label{diswellprrof}\cite{sarathes}
There exists two positive constants $C$ and $h^*$, such that for $h\leq h^*$, there exists a unique solution $\varphi_h$ of equation \eqref{odedisbfh}, and it satisfies the following estimate
\begin{equation}\label{weephih}
    \vert \varphi_h\vert_{H^1(0,T)}\leq C\norm{f}_{(H^1_{ ,0}(0,T))'}.
\end{equation}
\end{theorem}
\noindent The bound $h^*$ on the mesh-size $h$ and the stability constant $C$ in \eqref{weephih} depend on the final time $T$ and $\mu$. For the explicit values of $h^*$ and $C$, we refer to \cite[Theorem 3.2.13]{sarathes}.\\
\par Note that, for the case when finite element space $X_h$ consists of linear continuous basis functions (as in classical FEM), the well-posedness of \eqref{odedisbfh} can be established using compact perturbation argument under the following mesh condition \cite[Theorem 4.7]{coer}.
\begin{equation}\label{femhb}
h \leq \frac{2\sqrt{3}}{(2+\sqrt{\mu}T)\mu T}.
\end{equation}
Moreover, the discrete variational formulation \eqref{odedisbf} with space $X_h$ having linear hat basis functions can be viewed as a two-step method in the case of uniform mesh. 
Consequently, the analysis of the root-condition for the stability of this two-step method yields the following mesh restriction instead of \eqref{femhb} (for more details, see \cite[Remark 4.9]{coer}):

\begin{equation}\label{FEModeCFL}
h<\sqrt{\frac{12}{\mu}}.
\end{equation}
The stability analysis of the discrete variational formulation \eqref{odedisbf} for space $X_h$ having splines of maximal regularity as basis functions  with uniform mesh has been investigated in \cite{Ferrari2025}. This analysis was carried out by examining the properties of a family of matrices that arise from discrete variational formulation \eqref{odedisbf}. These matrices have a special structure known as Toeplitz band form, and the analysis was focused on the condition number of matrices. Specifically, the family of matrices is said to be weakly well-conditioned if the condition number of matrices associated to \eqref{odedisbf} does not grow exponentially as the system size increases. Under this framework, \cite[Theorem 5.10]{Ferrari2025} establishes that the family of matrices associated with the discrete variational formulation \eqref{odedisbf} is weakly well-conditioned if and only if the following condition on the mesh size is satisfied:

\begin{equation}\label{IgAodeCFL}
h<\sqrt{\frac{\rho_p}{\mu}},
\end{equation} where the values $\rho_p$ are explicitly obtained in \cite{Ferrari2025}. Such condition on mesh size leads to the Courant–Friedrichs–Lewy (CFL) condition of the type $h_t<Ch_s$ when analyzing the stability of the discrete space-time variational formulation of the second-order wave equation. Here, $h_t$ and $h_s$ are the mesh sizes in the temporal and the spatial direction respectively and $C>0$ is a constant depending on the spatial domain and the polynomial degree $p$.
\par To overcome the mesh condition \eqref{FEModeCFL} in FEM discretization of \eqref{odevf} with linear hat basis functions, authors in \cite{Steinbach2019} proposed the following stabilized bilinear form:
\begin{equation}\label{perbf}
	a_h(\varphi_h,v_h)=a(\varphi_h,v_h)-\frac{\mu h^2}{12}(\partial_t\varphi_h,\partial_tv_h)_{L^2(0,T)}.
\end{equation} 
This stabilization was further extended to higher order finite element basis function in \cite{Zank2021}. In a similar spirit, the authors in \cite{Fraschini2024} proposed a stable space-time isogeometric discretization of \eqref{odevf} with high degree B-splines having maximal regularity by considering the following stabilized bilinear form 
\begin{equation}\label{igaodest}
	a_h(\varphi_h,v_h)=a(\varphi_h,v_h)-\mu\delta h^{2p}(\partial_t^p\varphi_h,\partial_t^pv_h)_{L^2(0,T)},
\end{equation}
for a parameter $\delta\geq 0$. The stability of the formulation \eqref{igaodest} was proven in \cite[Theorem 5.10]{Ferrari2025} by showing that the family of Toeplitz band matrices associated with the bilinear form \eqref{igaodest} is weakly well-conditioned if and only if $\delta\geq\delta_p$ and the optimal values of the penalty term $\delta_p$ are also presented in \cite{Ferrari2025}.\\

 \section{{Well-posedness} of space-time {variational} formulation}\label{sec:stvf continuous}
In this section, we analyze the space-time variational formulation of \eqref{eqbiwave} which is given as: find $u\in X$ such that
\begin{equation}\label{wavevf}
	-\int_{0}^{T}\int_{\Omega}\partial_tu\partial_tv\ d\Omega dt+\int_{0}^{T}\int_{\Omega}\Delta u\Delta v\ d\Omega dt=\int_{0}^{T}\int_{\Omega}fv\ d\Omega dt\quad\mbox{for all }v\in Y.
\end{equation}
Following the similar arguments of \cite[Theorem 5.1]{coer}, in the next theorem we prove the unique solvability of \eqref{wavevf} when assuming $f\in L^2(L^2(\Omega))$.
\begin{theorem} For $f\in L^2(L^2(\Omega))$, there exists a unique $u\in\tv$ which satisfies \eqref{wavevf} and the following estimate holds
	\begin{equation*}
		\norm{u}_{\tv}\leq \frac{1}{\sqrt{2}}T\norm{f}_{L^2(L^2(\Omega))}.
	\end{equation*}
\end{theorem}
\begin{proof}
	First, we consider the following spatial eigenvalue problem
	\begin{equation}\label{eig}
		\begin{split}
			\Delta^2\psi&=\lambda\psi\quad\mbox{in }\Omega,\\
			\psi=\frac{\partial\psi}{\partial n}&=0\quad\mbox{on }\partial\Omega,
		\end{split}
	\end{equation}
with $\norm{\psi}_{L^2(L^2(\Omega)}=1$. Problem \eqref{eig} admits an increasing sequence of eigenvalues $\{\lambda_n\}_{n=1}^{\infty}$ with $0<\lambda_1\leq\lambda_2\cdots\leq\lambda_n\rightarrow\infty$ as $n\rightarrow\infty$ and the corresponding family of eigen functions $\{\psi_n\}_{n=1}^{\infty}$ forms an orthonormal basis of $L^2(\Omega)$ and an orthogonal basis of $H_0^2(\Omega)$. Also, let $v_k, k=0,1,2,\ldots$ be the eigen functions of the eigenvalue problem 
	\begin{equation*}
		\begin{split}
			-\partial_{tt}\phi(t)&=\lambda\phi(t)\quad\mbox{in }(0,T)\\	
			\phi(0)=\partial_t\phi(T)&=0.
		\end{split}
	\end{equation*}
	Then, ${v_k}$ forms an orthogonal basis of $L^2(0,T)$. For $\varphi\in H^1_{0,}(0,T)$, we can write
	\begin{equation*}
		\varphi(t)=\sum_{k=0}^{\infty}\varphi_kv_k(t),
	\end{equation*}
	where $\displaystyle \varphi_k=\frac{2}{T}\int_{0}^{T}\varphi(t)v_k(t)\ dt$. Then, the Fourier series representation of a function $u\in L^2(L^2(\Omega))$ is given as
	\begin{equation*}
		\begin{split}
			u(\x,t)=&\sum_{i=1}^{\infty}\sum_{k=0}^{\infty}u_{i,k}v_k(t)\psi_i(\x)\\
			&=\sum_{i=1}^{\infty}U_i(t)\psi_i(\x),
		\end{split}
	\end{equation*}
	where $U_i(t)=\displaystyle\sum_{k=0}^{\infty}u_{i,k}v_k(t)$ with the coefficients $u_{i,k}=\displaystyle \frac{2}{T}\int_{0}^{T}\int_{\Omega}u(\x,t)v_k(t)\psi_i(\x)\ d\Omega dt$.
	For fix $j\in\mathbb{N}$, take $v(\x,t)=V(t)\psi_j(\x)\in\tw$ as a test function with $V\in H^1_{,0}(0,T)$. The variational formulation \eqref{wavevf} becomes
	\begin{equation*}
		\begin{split}
			-\int_{0}^{T}\int_{\Omega}\sum_{i=1}^{\infty}\partial_tU_i(t)\psi_i(\x)\partial_tV(t)\psi_j(\x)\ d\Omega dt&+\int_{0}^{T}\int_{\Omega}\sum_{i=1}^{\infty}U_i(t)\Delta \psi_i(\x)V(t)\Delta\psi_j(\x)\ d\Omega dt\\
			&=\int_{0}^{T}\int_{\Omega}f(\x,t)V(t)\psi_j(\x)\ d\Omega dt
		\end{split}
	\end{equation*}
	which implies
	\begin{equation}\label{eigode}
		-\int_{0}^{T}\partial_tU_j(t)\partial_tV(t)\ dt+\lambda_j\int_{0}^{T}U_j(t)V(t)\ dt=\int_{0}^{T}f_j(t)V(t)\ dt,
	\end{equation} where $f_j=\int_{\Omega}f\psi_j\ d\Omega$. We note that \eqref{eigode} is of the form \eqref{odevf} and, as established in Section \ref{sec:pre} has an equivalent formulation of the form \eqref{odevfh}. Theorem \ref{wellposode} therefore guarantees the existence of a unique solution $U_j$ to \eqref{eigode} for each $j$, from which we conclude that \eqref{wavevf} also has a unique solution $u$.\\
	Next, to prove the a priori estimate, note that for $f\in L^2(L^2(\Omega)),$ we have
	\begin{equation*}
		f(\x,t)=\sum_{j=1}^{\infty}f_j(t)\psi_j(\x)
	\end{equation*}
	and $\displaystyle\norm{f}_{L^2(L^2(\Omega))}^2$=$\displaystyle\sum_{j=1}^{\infty}\norm{f_j(t)}_{L^2(0,T)}^2$.
	Then using the bound given in Lemma \ref{odebound}, we have
	\begin{align*}
	    \norm{u}_{\tv}^2&=\int_{0}^{T}\int_{\Omega}\left(\vert\partial_tu(\x,t)\vert^2+\vert\Delta u(\x,t)\vert^2\right)\ d\Omega dt\\
			&=\sum_{i=1 }^{\infty}\sum_{j=1}^{\infty}\int_{0}^{T}\partial_t U_i(t)\partial_t U_j(t)\ dt\int_{\Omega}\psi_i(\x)\psi_j(\x)\ d\Omega\\
			&+\int_{0}^{T}U_i(t)U_j(t)\ dt\int_{\Omega}\Delta\psi_i(\x)\Delta\psi_j(\x)\ d\Omega\\
			&=\sum_{i=1}^{\infty}\left(\int\vert\partial_tU_i(t)\vert^2+\lambda_i\vert U_i(t)\vert^2\ dt\right)\\
			&=\sum_{i=1}^{\infty}\left(\norm{U_i}_{H^1_{0,1}(0,T)}^2+\lambda_i\norm{U_i}_{L^2(0,T)}^2\right)\\
			&\leq \frac{1}{2}T^2\sum_{i=1}^{\infty}\norm{f_i}^2_{L^2(0,T)}=\frac{1}{2}T^2\norm{f}^2_{L^2(L^2(\Omega))}
	\end{align*}
	which is the required estimate.
\end{proof}

 \section{Space-time isogeometric discretization and stability analysis}\label{sec4:stiga}
We discuss the space-time discretization of the variational formulation \eqref{wavevf}. First, we give an overview of the isogeometric spaces. For more details on splines and their properties, we refer to \cite{nurbsbook}. \\
We start by defining the univariate B-splines. Given positive integer $m\in\mathbb{N}$ and $q>0$ consider a knot vector $\Xi=\{0=\eta_1\leq\eta_2\leq\cdots\leq\eta_{m+p+1}=1\}$. We assume that $\Xi$ is an open knot vector, i.e., $\eta_1=\cdots=\eta_p$ and $\eta_{m+1}=\cdots=\eta_{m+p+1}$. By the Cox-de-Boor recursion formula, the univariate B-spline functions are defined as follows:\\
For $p=0$,
\begin{equation*}
	B^0_j(\zeta)=\begin{cases}
		1,\quad\mbox{if }\zeta\in[\eta_j,\eta_{j+1})\\
		0,\quad\mbox{otherwise. }
	\end{cases}
\end{equation*}
For $p\geq1$,
\begin{equation*}
	B^p_j(\zeta)=\frac{\zeta-\eta_j}{\eta_{j+p}-\eta_j}B^{p-1}_j(\zeta)+\frac{\eta_{j+p+1}-\zeta}{\eta_{j+p+1}-\eta_{j+1}}B^{p-1}_{j+1}(\zeta),
\end{equation*}
where by convention $0/0=0$. Then the univariate spline space is given by 
\begin{equation*}
	S^p_h(\Xi)=\mbox{span}\{B^p_j:j=1,\ldots,m\},
\end{equation*}
where $h=\mbox{max}\{|\eta_j-\eta_{j-1}\vert:j=2,\ldots,m+p+1\}$ denotes the mesh size.
The $p^{\text{th}}$ degree B-spline function is $C^{p-l}$ time regular at a knot having multiplicity $l$ and $C^{\infty}$ regular elsewhere. By taking the tensor product of the univariate B-splines, we can define the multivariate B-splines. Given positive integers $m_1,\ldots,m_d$, $m_t$ and $p_1,\ldots,p_d,\ p_t\in\mathbb{N}\cup\{0\}$, consider $d+1$ open knot-vectors $\Xi_l=\{\eta_{l,1}\leq\eta_{l,2}\leq\cdots\leq\eta_{l,m_l+p_l+1}\}$ for $l=1,\ldots,d$ and $\Xi_t=\{\eta_{t,1}\leq\eta_{t,2}\leq\cdots\leq\eta_{t,m_t+p_t+1}\}$. By $h_l$, we denote the mesh size for the knot vector $\Xi_l$ for $1\leq l\leq d$. Similarly, $h_t$ denotes the mesh size for the knot vector $\Xi_t$. We set $h_s=\mbox{max}\{h_l:l=1,\ldots,d\}$. The degree knot vector is defined as $\mathbf{p}=(\mathbf{p}_s,p_t)$, where $\mathbf{p}_s=(p_1,\ldots,p_d)$. For simplicity, in our analysis we take $p_1=\cdots=p_d=p_s$. The multivariate B-splines on $\widehat{\Omega}\times(0,1)=(0,1)^d\times(0,1)$ are defined as 
\begin{equation*}
	B^{\mathbf{p}}_{\mathbf{j}}(\mathbf{\zeta},\tau)=B^{\mathbf{p}_s}_{\mathbf{j}_s}(\mathbf{\zeta})B^{p_t}_{j_t}(\tau),
\end{equation*} 
where
\begin{equation*}
	B^{\mathbf{p}_s}_{\mathbf{j}_s}(\mathbf{\zeta})=B^{p_1}_{j_1}(\zeta_1)\cdots B^{p_d}_{j_d}(\zeta_d),
\end{equation*}
$\mathbf{\zeta}=(\zeta_1,\ldots,\zeta_d)$, $\mathbf{j}=(\mathbf{j}_s,j_t)$, $\mathbf{j}_s=(j_1,\ldots,j_d)$. The corresponding spline space is given as
\begin{equation*}
	S^{\mathbf{p}}_h=\mbox{span}\{B^{\mathbf{p}}_{\mathbf{j}}: 1\leq j_l\leq m_l, l=1,\ldots,d, 1\leq j_t\leq m_t\},
\end{equation*}
where $h=\mbox{max}\{h_s,h_t\}$. Due to the tensor product structure, we have $S^{\mathbf{p}}_h=S^{\mathbf{p}_s}_{h_s}\otimes S^{p_t}_{h_t}$, where
\begin{equation*}
	S^{\mathbf{p}_s}_{h_s}=\mbox{span}\{B^{\mathbf{p}_s}_{\mathbf{j}_s}:1\leq j_l\leq m_l,l=1,\ldots,d\}
\end{equation*}
and 
\begin{equation*}
	S^{p_t}_{h_t}=\mbox{span}\{B^{p_t}_{j_t}:1\leq j_t\leq m_t\}.
\end{equation*}
\begin{ass}
	We have $p_s\geq 2$, $p_t\geq 1$ and $S^{\mathbf{p}_s}_{h_s}\subset C^{p_s-1}(\widehat{\Omega})$, $S^{p_t}_{h_t}\subset C^{p_t-1}(0,1)$.
\end{ass}
Next, we define the isogeometric spaces. Consider a parameterization map $\mathbf{F}:\widehat{\Omega}\rightarrow\Omega$ with $\mathbf{F}\in \left(S^{\mathbf{p}_s}_{h_s}\right)^d$.
\begin{ass}
	We assume that $\mathbf{F}^{-1}$ exists and has piecewise bounded derivatives of all order.
\end{ass}
The space-time cylinder $\Omega\times(0,T)$ is then parameterized using the map $\widetilde{\mathbf{F}}:\widehat{\Omega}\times(0,1)\rightarrow\Omega\times(0,T)$, $\widetilde{\mathbf{F}}(\mathbf{\zeta},\tau)=(\mathbf{F}(\mathbf{\zeta}),T\tau)$. To incorporate the initial and boundary conditions, we consider
\begin{equation*}
	\widehat{X}_h=\{\widehat{w}_h\in S^{\mathbf{p}}_h:\widehat{w}_{h}=\frac{\partial\widehat{w}_h}{\partial n}=0\mbox{ on }\partial\widehat{\Omega}\times(0,1)\mbox{ and }\widehat{w}_h(\mathbf{\zeta},0)=0\}
\end{equation*}
and 
\begin{equation*}
	\widehat{Y}_h=\{\widehat{w}_h\in S^{\mathbf{p}}_h:\widehat{w}_{h}=\frac{\partial\widehat{w}_h}{\partial n}=0\mbox{ on }\partial\widehat{\Omega}\times(0,1)\mbox{ and }\widehat{w}_h(\mathbf{\zeta},1)=0\}.
\end{equation*}
Note that,
\begin{equation*}
	\widehat{X}_h=S^{\mathbf{p}_s}_{h_s,0}\otimes S^{p_t}_{h_t;0,}
\end{equation*}
and
\begin{equation*}
	\widehat{Y}_h=S^{\mathbf{p}_s}_{h_s,0}\otimes S^{p_t}_{h_t;,0},
\end{equation*}
where
\begin{equation*}
	S^{\mathbf{p}_s}_{h_s,0}=\{\widehat{w}_{h_s}\in S^{\mathbf{p}_s}_{h_s}: \widehat{w}_{h_s}=\frac{\partial\widehat{w}_{h_s}}{\partial n}=0\mbox{ on }\partial\widehat{\Omega}\times(0,1)\}
\end{equation*}
\begin{equation*}
	S^{p_t}_{h_t;0,}=\{\widehat{w}_{h_t}\in S^{p_t}_{h_t}: \widehat{w}_{h_t}(0)=0\}
\end{equation*}
and
\begin{equation*}
	S^{p_t}_{h_t;,0}=\{\widehat{w}_{h_t}\in S^{p_t}_{h_t}: \widehat{w}_{h_t}(1)=0\}.
\end{equation*}
Here the dimensions of the spaces $\widehat{X}_h$ and $\widehat{Y}_h$ are same, i.e., $\mbox{dim}(\widehat{X}_h)=\mbox{dim}\widehat{Y}_h={N}_{dof}$, where ${N}_{dof}=\mathbf{n}_sn_t$ with $\mathbf{n}_s=(m_1-4)(m_2-4)\cdots(m_d-4)$ and $n_t=m_t-1$.\\
Finally, the isogeometric spaces are defined as
\begin{equation}\label{xhdef}
	X_h=\{w_h=\widehat{w}_h\circ\widetilde{\mathbf{F}}^{-1}:\widehat{w}_h\in\widehat{X}_h\}
\end{equation}
and
\begin{equation}\label{yhdef}
	Y_h=\{w_h=\widehat{w}_h\circ\widetilde{\mathbf{F}}^{-1}:\widehat{w}_h\in\widehat{Y}_h\}.
\end{equation}
We have
\begin{equation*}
	X_h=W_{s,h_s}\otimes X_{t,h_t}
\end{equation*}
and
\begin{equation*}
	Y_h=W_{s,h_s}\otimes Y_{t,h_t},
\end{equation*}
where
\begin{equation*}
	W_{s,h_s}=\{w_{h_s}=\widehat{w}_{h_s}\circ\mathbf{F}^{-1}:\widehat{w}_{h_s}\in S^{\mathbf{p}_s}_{h_s,0}\},
\end{equation*}
\begin{equation*}
	X_{t,h_t}=\{w_{h_t}=\widehat{w}_{h_t}(\cdot/T):\widehat{w}_{h_t}\in S^{p_t}_{h_t;0,}\}
\end{equation*}
and
\begin{equation*}
	Y_{t,h_t}=\{w_{h_t}=\widehat{w}_{h_t}(\cdot/T):\widehat{w}_{h_t}\in S^{p_t}_{h_t;,0}\}.
\end{equation*}
Using the discrete test and trial spaces as defined in \eqref{xhdef} and \eqref{yhdef} respectively, the space-time isogeometric discretization of \eqref{wavevf} is given as: find $u_h\in X_h$ such that
\begin{equation}\label{disvf}
	-\int_{0}^{T}\int_{\Omega}\partial_tu_h\partial_tv_h\ d\Omega dt+\int_{0}^{T}\int_{\Omega}\Delta u_h\Delta v_h\ d\Omega dt=\int_{0}^{T}\int_{\Omega} fv_h\ d\Omega dt,\mbox{ for all }v_h\in Y_h.
\end{equation}

Since we have tensor product structure of the test and trial spaces $X_h$ and $Y_h$, we can write
\begin{equation*}
	u_h(\x,t)=\sum_{i=1}^{\mathbf{n}_s}\sum_{j=1}^{n_t}u_{i,j} \phi_{s,i}(\x)\psi_{t,j}(t)=\sum_{i=1}^{\mathbf{n}_s}U_{t,i}(t)\phi_{s,i}(\x),\quad U_{t,i}(t)=\sum_{j=1}^{n_t}u_{i,j}\psi_{t,j}(t),
\end{equation*}
where $\phi_{s,i}$ are the basis functions of $W_{s,h_s}$ and $\psi_{t,j}$ are the basis functions of $X_{t,h_t}$.
To discuss the well-posedness of the variational formulation \eqref{disvf}, we first obtain a semi-discrete formulation as done in \cite{coer}.\\
Consider 
\begin{equation*}
	\widetilde{u}(\x,t)=\sum_{i=1}^{\mathbf{n}_s}\widetilde{U}_{i}(t)\phi_{s,i}(\x).
\end{equation*}
We can write the semi-discrete formulation as
\begin{equation*}
	\mathbf{M}_{s}\partial_{tt}\overline{\mathbf{U}}(t)+\mathbf{N}_s\overline{\mathbf{U}}(t)=\overline{\mathbf{f}}(t)\mbox{ for }t\in(0,T),\quad \overline{\mathbf{U}}(0)=\partial_t\overline{\mathbf{U}}(0)=\overline{0},
\end{equation*}
where $\overline{\mathbf{U}}=[\widetilde{U}_1,\ldots,\widetilde{U}_{{\mathbf{n}}_s}]'$, $\mathbf{M}_s$ and $\mathbf{N}_s$ are respectively the spatial mass matrix and stiffness matrix given by
\begin{equation*}
	[\mathbf{M}_s]_{j,k}=\int_{\Omega}\phi_{s,k}(\x)\phi_{s,j}(\x)\ d\Omega,
\end{equation*}
\begin{equation*}
	[\mathbf{N}_s]_{j,k}=\int_{\Omega}\Delta \phi_{s,k}(\x)\Delta \phi_{s,j}(\x)\ d\Omega,
\end{equation*}
and the load vector
\begin{equation*}
	[\mathbf{f}]_j=\int_{\Omega}f(\x,t)\phi_{s,j}(\x)\ d\Omega.
\end{equation*}
Using the $LL'$ factorization, we have
\begin{equation*}
	\mathbf{M}_s=\mathbf{L}_s\mathbf{L}_s',\ \mathbf{A}_s=\mathbf{L}_s^{-1}\mathbf{N}_s\mathbf{L}_s'^{-1},\ \overline{\mathbf{W}}=\mathbf{L}_s'\overline{\mathbf{U}},\ \mathbf{g}(t)=\mathbf{L}_s^{-1}\mathbf{f}(t).
\end{equation*}
Then
\begin{equation*}
	\partial_{tt}\overline{\mathbf{W}}(t)+\mathbf{A}_s\overline{\mathbf{W}}(t)=\mathbf{g}(t)\mbox{ for }t\in(0,T),\ \overline{\mathbf{W}}(0)=\partial_t\overline{\mathbf{W}}(0)=\overline{0}.
\end{equation*}
The matrix $\mathbf{A}_s$ is positive-definite and symmetric, hence it has a diagonal representation
\begin{equation*}
	\mathbf{A}_s=\mathbf{V}_s\mathbf{D}_s\mathbf{V}_s',\quad\mathbf{D}_s=\mbox{diag}(\lambda_i(\mathbf{A}_s))_{i=1}^{\mathbf{n}_s},
\end{equation*} 
where $\mathbf{V}_s=({\mathbf{v}}_1,\ldots,{\mathbf{v}}_{\mathbf{n}_s})$ and $\mathbf{A}_s\mathbf{v}_i=\lambda_i(\mathbf{A}_s)\mathbf{v}_i$. Using $\overline{\mathbf{Z}}(t)=\mathbf{V}_s'\overline{\mathbf{W}}(t)$, we get the following $\mathbf{n}_s$ scalar ODEs:
\begin{equation}\label{sysode}
	\partial_{tt}\overline{\mathbf{Z}}(t)+\mathbf{D}_s\overline{\mathbf{Z}}(t)=\widetilde{\mathbf{g}}(t)\mbox{ for }t\in(0,T),\ \overline{\mathbf{Z}}(0)=\partial_t\overline{\mathbf{Z}}(0)=\overline{0},
\end{equation}
where $\widetilde{\mathbf{g}}(t)=\mathbf{V}_s'\overline{\mathbf{g}}(t)$.
The discrete formulation related with \eqref{sysode} is given as, find $z_{h_t,i}\in X_{t,h_t}$ such that
\begin{equation}\label{semi-dis}
	-\int_{0}^{T}\partial_tz_{h_t,i}\partial_tv_{h_t}\ dt+\lambda_i(\mathbf{A}_s)\int_{0}^{T}z_{h_t,i}v_{h_t}\ dt=\int_{0}^{T}\overline{\widetilde{\mathbf{g}}}_iv_{h_t}\ dt\quad\mbox{for all }v_{h_t}\in X_{t,h_t}.
\end{equation}
Note that $\overline{\mathbf{Z}}_h(t)=\mathbf{V}_s'\mathbf{L}_s'\mathbf{U}_s(t)$ with $\mathbf{U}_s(t)=(U_{s,1}(t),\ldots,U_{s,\mathbf{n}_s}(t))$. The vector $\mathbf{U}_s$ is the vector of the unknown functions of the approximation function $u_h$. The variational formulation \eqref{semi-dis} has the form \eqref{odedisvf}. Hence, the stability and the well-posedness of the solutions $z_{h_t,i}$ under sufficiently small mesh size $h_t$, follows from Theorem \ref{diswellprrof}. However, as discussed in the Section \ref{sec:pre}, for uniform mesh refinement, the stability of the discretization \eqref{semi-dis} is achieved if and only if the mesh condition \eqref{IgAodeCFL} is satisfied. This condition turns out to be a CFL condition for discrete variational formulation \eqref{disvf}, which we discuss in the following subsection. 
\subsection{Stability analysis} We recall from Section \ref{sec:pre} that the discretization \eqref{semi-dis} is stable if and only if 
\begin{equation}\label{CFLodefourth}
    h_t<\sqrt{\frac{\rho_{p_t}}{\lambda_i}},
\end{equation} where the values of $\rho_{p_t}$ for different polynomial degree $p_t$ are displayed in Table \ref{rhop} \cite{Ferrari2025}.
\begin{table}
	\centering
	\begin{tabular}{ |c|c|c|c|c|c|c| } 
		\hline
		$p_t$ & 1 & 2 &3 &4 & 5 & 6 \\
		\hline 
		$\rho_{p_t}$ &\ \ 12 \ \  &\ \ 10  \ \ &168/17&306/31&2349/238&7797/790  \\ 
		\hline
	\end{tabular}
	\caption{Values of $\rho_{p_t}$}
	\label{rhop}
\end{table}
Here,  $\lambda_i$, $1\leq i\leq \mathbf{n}_s$ denote the eigen values of the matrix $\mathbf{A}_s$ associated with the bi-Laplacian operator. These eigenvalues are given by:
\begin{equation*}
	\lambda_i=\frac{(\mathbf{A}_s\mathbf{v}_i,\mathbf{v}_i)}{(\mathbf{v}_i,\mathbf{v}_i)}=\frac{(\mathbf{N}_s(\mathbf{u}_i),\mathbf{u}_i)}{(\mathbf{u}_i,\mathbf{u}_i)}=\frac{\norm{\Delta u_{s,i}}_{L^2(\Omega)}^2}{\norm{u_{s,i}}_{L^2(\Omega)}^2},
\end{equation*}
where $\mathbf{u}_i=\mathbf{L}_s'^{-1}\mathbf{v}_i$ and $u_{s,i}\in W_{s,h_s}$ are related functions. Then from \eqref{CFLodefourth}, we have the stability of \eqref{semi-dis} if and only if
\begin{equation}\label{CFL1}
    \lambda_i=\frac{\norm{\Delta u_{s,i}}_{L^2(\Omega)}^2}{\norm{u_{s,i}}_{L^2(\Omega)}^2}< \frac{\rho_{p_t}}{h_t^2}.
\end{equation}
Using the inverse inequality for the spline functions \cite{invineq,invineq2}, we get
\begin{equation}\label{invinq}
	\norm{\Delta v_{h_s}}^2_{L^2(\Omega)}\leq C_{\Omega}h_s^{-4}\norm{v_{h_s}}_{L^2(\Omega)}^2\mbox{ for all }v_{h_s}\in W_{s,h_s},
\end{equation}
where $C_{\Omega}$ is a constant depending on the spatial domain $\Omega$. In view of the inverse inequality \eqref{invinq}, condition \eqref{CFL1} holds provided that
\begin{equation*}
	C_{\Omega}h_s^{-4}<\rho_{p_t}h_t^{-2}.
\end{equation*}
This gives us
\begin{equation}\label{cfl}
	h_t<h_s^2\sqrt{\frac{\rho_{p_t}}{C_{\Omega}}}.
\end{equation}
Thus, the stability of \eqref{disvf} follows if and only if the mesh-sizes $h_t$ and $h_s$ satisfy the CFL condition \eqref{cfl}.
\begin{remark}
It is important to emphasize that the CFL condition \eqref{cfl}, which is necessary to ensure the stability of the space–time isogeometric discretization of the fourth-order wave equation \eqref{eqbiwave}, is different from the CFL condition associated with the space–time discretization of a second-order wave equation, as presented in \cite{coer,Fraschini2024}.
\end{remark}
\par Motivated from \cite{Fraschini2024,Ferrari2025}, in order to overcome the CFL restriction \eqref{cfl}, we employ the stabilized bilinear form \eqref{igaodest} to the space–time isogeometric discretization of the equation \eqref{wavevf} in the following manner.\\\\
\noindent{\bf{Stabilized formulation:}}
A stabilized space-time isogeometric discretization of \eqref{wavevf} is achieved by adding a non-consistent penalty term in \eqref{disvf} and it is given below: find $u_h\in X_h$ such that
\begin{equation}\label{stabwave}
	\begin{split}
		-\int_{0}^{T}\int_{\Omega}\partial_tu_h\partial_tv_h\ d\Omega dt+\int_{0}^{T}\int_{\Omega}\Delta u_h\Delta v_h\ d\Omega dt&-\delta h_t^{2p_t}\int_{0}^{T}\int_{\Omega}\Delta\partial_t^{p_t}u_h\Delta\partial_t^{p_t}v_h\ d\Omega dt\\
		&	=\int_{0}^{T}\int_{\Omega}fv_h\ d\Omega dt\mbox{ for all } v_h\in Y_h,
	\end{split}
\end{equation}
where $\delta\geq\delta_{p_t}$. The explicit values of $\delta_{p_t}$ for different polynomial degrees $p$ is given in Table \ref{delta_{p_t}} \cite{Ferrari2025}.
\begin{table}[H]
	\centering
	\begin{tabular}{ |c|c|c|c|c|c|c| } 
		\hline
        $p_t$&1&2&3&4&5&6\\
		\hline 
		$\delta_{p_t}$ &\ \ 1/12 \ \  &\ \ 1/120  \ \ &17/20160&5/58529&2/231067&1/1140271  \\ 
		\hline
	\end{tabular}
	\caption{Values of $\delta_{p_t}$}
	\label{delta_{p_t}}
\end{table}.
\begin{remark}
   In \cite{Steinbach2019,Zank2021}, an unconditionally stable space-time finite element discretization of second-order wave equation was proposed which is reads: 
   find $u_h\in X_h$ such that
\begin{equation}\label{stabwavefem2}
	\begin{split}
		-\int_{0}^{T}\int_{\Omega}\partial_tu_h\partial_tv_h\ d\Omega dt+\int_{0}^{T}\int_{\Omega}\nabla u_h(\mathcal{P}_{h_t}^{p_t-1,dis}\nabla v_h)\ d\Omega dt\\
		&	=\int_{0}^{T}\int_{\Omega}fv_h\ d\Omega dt 
	\end{split}
\end{equation}
for all $v_h\in Y_h,$
where $X_h$ and $Y_h$ are spaces of globally continuous piece-wise polynomials of degree $p_t$ in temporal direction, $\mathcal{P}_{h_t}^{p_t-1,dis}: L^2(L^2(\Omega))\rightarrow L^2(\Omega)\otimes Y_{t,h_t}^{p_t-1,dis}(0,T)$ is the $L^2$-orthogonal projection onto the space of piecewise polynomial discontinuous function with respect to the time variable. In the same spirit, a stabilized space-time finite element discretization of fourth-order wave equation \eqref{wavevf} can be formulated as: find $u_h\in X_h$ such that
\begin{equation}\label{stabwavefem}
	\begin{split}
		-\int_{0}^{T}\int_{\Omega}\partial_tu_h\partial_tv_h\ d\Omega dt+\int_{0}^{T}\int_{\Omega}\Delta u_h(\mathcal{P}_{h_t}^{p_t-1,dis}\Delta v_h)\ d\Omega dt\\
		&	=\int_{0}^{T}\int_{\Omega}fv_h\ d\Omega dt
	\end{split}
\end{equation}
for all $v_h\in Y_h$. If $p_t=1$ and $\delta=\frac{1}{12}$, then the two variational formulations \eqref{stabwave} and \eqref{stabwavefem} coincide. 
\end{remark}
\section{{Efficient solver}}\label{sec:solver}
The matrix system associated with \eqref{stabwave} is given as follows:
\begin{equation}\label{stabms}
	\mathbf{A}\mathbf{x}=\mathbf{f},
\end{equation}
where
\begin{equation*}
	[\mb{A}]_{i,j}=-\int_{0}^{T}\int_{\Omega}\partial_tw_j\partial_tv_i\ d\Omega dt+\int_{0}^{T}\int_{\Omega}\Delta w_j\Delta v_i\ d\Omega dt-\delta h_t^{2p_t}\int_{0}^{T}\int_{\Omega}\Delta\partial_t^{p_t}w_j\Delta\partial_t^{p_t}v_i\ d\Omega dt
\end{equation*}
and
\begin{equation*}
	[\mb{f}]_i=\int_{0}^{T}\int_{\Omega}fv_i\ d\Omega dt.
\end{equation*}
Here we assume that $w_j,j=1,\ldots,N_{dof}$ are the basis functions of the space $X_h$ and $v_i,i=1,\ldots,N_{dof}$ are the basis functions of $Y_h$. Due to the tensor product structure of the spaces $X_h$ and $Y_h$, the matrix $\mb{A}$ can be written as
\begin{equation*}
	\mb{A}=(\mb{M}_t-\mb{P}_t)\otimes\mb{K}_x-\mb{K}_t\otimes\mb{M}_x,
\end{equation*}
where
\begin{equation*}
	[\mb{M}_t]_{i,j}=\int_{0}^{T}\psi_{j,t}\psi_{i,t}\ dt\quad [\mb{P}_t]_{i,j}=\delta h_t^{2p_t}\int_{0}^{T}\partial_t^{p_t}\psi_{j,t}\partial_t^{p_t}\psi_{i,t}\ dt,
\end{equation*}
\begin{equation*}
	[\mb{K}_t]_{i,j}=\int_{0}^{T}\psi_{j,t}'\psi_{i,t}'\ dt
\end{equation*}
and
\begin{equation*}
	[\mb{M}_x]_{i,j}=\int_{\Omega}\phi_{i,x}\phi_{j,x}\ d\Omega\quad [\mb{K}_x]_{i,j}=\int_{\Omega}\Delta\phi_{i,x}\Delta\phi_{j,x}\ d\Omega,
\end{equation*}
with $\psi_{j,t}\in Y_{t,h_t},\psi_{i,t}\in X_{t,h_t}$ and $\phi_{i,x}\in W_{s,h_s}$.
One of the major drawbacks of the space-time methods is the high-computational cost of solving a linear system in $(d+1)$ dimensional space-time cylinder. To address this issue, we propose an efficient solver based on the Algorithm 1 of \cite{Loli2025} for solving the linear system \eqref{stabms}.
\par By exploiting the Kronecker product properties, the system can be solved without assembling the whole space-time matrix. First, we perform the complex generalized Schur factorization on the matrix pencil $(\mathbf{K}_t,\mathbf{M}_t-\mathbf{P}_t)$ and obtain
\begin{equation}\label{fact}
	\mathbf{C}_t\mathbf{K}_t\mathbf{D}_t=\mathbf{E}_t\mbox{ and }\mathbf{C}_t(\mathbf{M}_t-\mathbf{P}_t)\mathbf{D}_t=\mathbf{F}_t,
\end{equation}
where $\mathbf{C}_t$ and $\mathbf{D}_t$ are two unitary matrices, and $\mathbf{E}_t$ and $\mathbf{F}_t$ are two upper triangular matrices. Let $\mathbf{B}_t=\mathbf{E}^{-1}_t\mathbf{F}_t$. Then the matrix $\mathbf{A}$ can be written as
\begin{equation*}
	\mathbf{A}=(\mathbf{C}_t^*\mathbf{E}_t\otimes \mathbf{I}_{x})(\mathbf{B}_t\otimes\mathbf{K}_x-\mathbf{I}_t\otimes\mathbf{M}_x)(\mathbf{D}_t^*\otimes\mathbf{I}_x),
\end{equation*}
where $*$ denotes the conjugate transpose. Using the fact that $(\mathbf{A}_1\otimes\mathbf{A}_2)^{-1}=\mathbf{A}_1^{-1}\otimes\mathbf{A}_2^{-1}$, the inverse of $\mathbf{A}$ can be written as
\begin{equation*}
	\mathbf{A}^{-1}=(\mb{D}_t\otimes\mb{I}_x)(\mb{B}_t\otimes\mb{K}_x-\mb{I}_t\otimes\mb{M}_x)^{-1}(\mb{E}_t^{-1}\mb{C}_t\otimes\mb{I}_x).
\end{equation*}
In Algorithm \ref{al:solving}, the procedure to solve the system is given.
\begin{algorithm}
	\caption{Solving Algorithm}\label{al:solving}
	\begin{algorithmic}[1] 
		\State Compute the factorization \eqref{fact} and $\mb{B}_t=\mb{E}_t^{-1}\mb{F}_t$;
		\State Solve $(\mb{E}_t\otimes\mb{I}_x)\mb{s}_1=(\mb{C}_t\otimes\mb{I}_x)\mb{x}$ ;
		\State Solve $(\mb{B}_t\otimes\mb{K}_x-\mb{I}_t\otimes\mb{M}_x)\mb{s}_2=\mb{s}_1$;
		\State Compute $\mb{x}=(\mb{D}_t\otimes\mb{I}_x)\mb{s}_2$.
	\end{algorithmic}
\end{algorithm}
\\\noindent{\bf{Computational cost:}}
Next, we discuss the total computational cost taken by Algorithm \ref{al:solving}. For more details, we refer to \cite{Loli2025,Ferrarihamil}. In the Step 1, we have to perform a Schur factorization which takes $\mathcal{O}(n_t^3)$  Floating Point Operations (FLOPs), and the calculation of $\mb{B}_t=\mb{E}_t^{-1}\mb{F}_t$ takes $\mathcal{O}(n_t^2p)$ FLOPs. Hence, the total cost of Step 1 is $\mathcal{O}(n_t^3+n_t^2p)$ FLOPs. Step 2 involves solving $\mb{n}_s$ independent upper triangular systems associated with $\mb{E}_t$ which takes $\mathcal{O}(n_tp^2+\mb{n}_sn_tp)$ FLOPs. Also, in Step 2, $\mb{n}_s$ matrix-vector products associated with $\mb{C}_t$ are performed taking $\mathcal{O}(\mb{n}_sn_t^2)$ FLOPs. Step 3 solves a block upper triangular system in which each box has dimensions $\mb{n}_s\times\mb{n}_s$. Suppose $C_{1_x}$ is the cost of solving a single diagonal block system and $C_{2_x}$ is the cost of a matrix vector product involving one of the $\mathcal{O}(n_t^2)$ off diagonal block. The resulting computational cost is $\mathcal{O}(C_{1_x}n_t+C_{2_x}n_t^2)$. Finally, in Step 4 the $\mb{n}_s$ matrix-vector product involving $\mathbf{D}_t$ causes $\mathcal{O}(\mathbf{n}_sn_t^2)$ FLOPs. Hence, the total cost is $\mathcal{O}(n_t^3+n_t^2p+n_tp^2+\mb{n}_sn_tp+\mb{n}_sn_t^2+C_{1_x}n_t+C_{2_x}n_t^2)$ FLOPs. If we consider uniform refinement in spatial and temporal direction, when $n_t$ doubles, $\mathbf{n}_s$ grows with $\mathcal{O}(2^d)$ and hence the total number of degrees of freedom $N_{dof}$ increases by a factor $\mathcal{O}(2^{d+1})$. So, we have $n_t\sim \mathbf{n}_s^{1/d}\sim N_{dof}^{1/d+1}$.  Then, the total computational cost when assuming $p\leq N_{dof}$ is $\mathcal{O}(N_{dof}^\frac{d+2}{d+1}+C_{1_x}N_{dof}^\frac{1}{d+1}+C_{2_x}N_{dof}^\frac{2}{1+d})$.

\begin{remark}
    For the dense spatial matrices, the associated computational costs scale as $C_{1_x}=\mathcal{O}(\mathbf{n}_s^3)$ and $C_{2_x}=\mathcal{O}(\mathbf{n}_s^2)$. In contrast, Step 3 of Algorithm \ref{al:solving} involves the sparse matrices $\mathbf{M}_x$ and $\mathbf{K}_x$. Modern computers execute sparse matrix operations efficiently, yielding a substantial reduction in total computational cost. Specifically, for a two-dimensional domain, the dominant cost scales as $\mathcal{O}(N_{dof}^{4/3})$.
\end{remark}

  \section{Numerical experiments}\label{sec:numerical}
In this section, we present three numerical tests to illustrate the accuracy, stability, and convergence properties of the proposed space-time isogeometric discrete formulation \eqref{stabwave} for the biharmonic wave equation, as well as the efficiency of the linear solver discussed in Section \ref{sec:solver}. For all tests, we set the final time $T=1$. In the first test, we verify the stability of the space-time discrete formulations \eqref{stabwave},\eqref{stabwavefem} and \eqref{disvf} with splines of different regularities. The second experiment analyzes the performance of the linear solver discussed in Section \ref{sec:solver} in terms of computational cost. In the last test, we directly compare the accuracy and efficiency of the space-time IgA formulation \eqref{stabwave} against the space-time FEM formulation \eqref{stabwavefem} with respect to total number of degrees of freedom. 
All numerical simulations are performed with the GeoPDEs toolbox \cite{geopde} in MATLAB 2024b with a single-core configuration, on a Intel Xeon processor running at 3.90 GHz, with 256 GB RAM. The MATLAB code used in this study is developed based on the implementation available for the wave equation in \cite{github}. Throughout this section, $p$ denotes the spline degree, with $p\geq 2$, in both the spatial and temporal directions. In what follows, we write IgA-stab formulation for the stabilized space-time isogeometric method \eqref{stabwave}, and FEM-stab formulation  for the stabilized space-time finite element method \eqref{stabwavefem}.

\begin{figure}[htbp]
	\centering
	\begin{subfigure}{0.3\textwidth}
		\includegraphics[height=6cm,width=6cm]{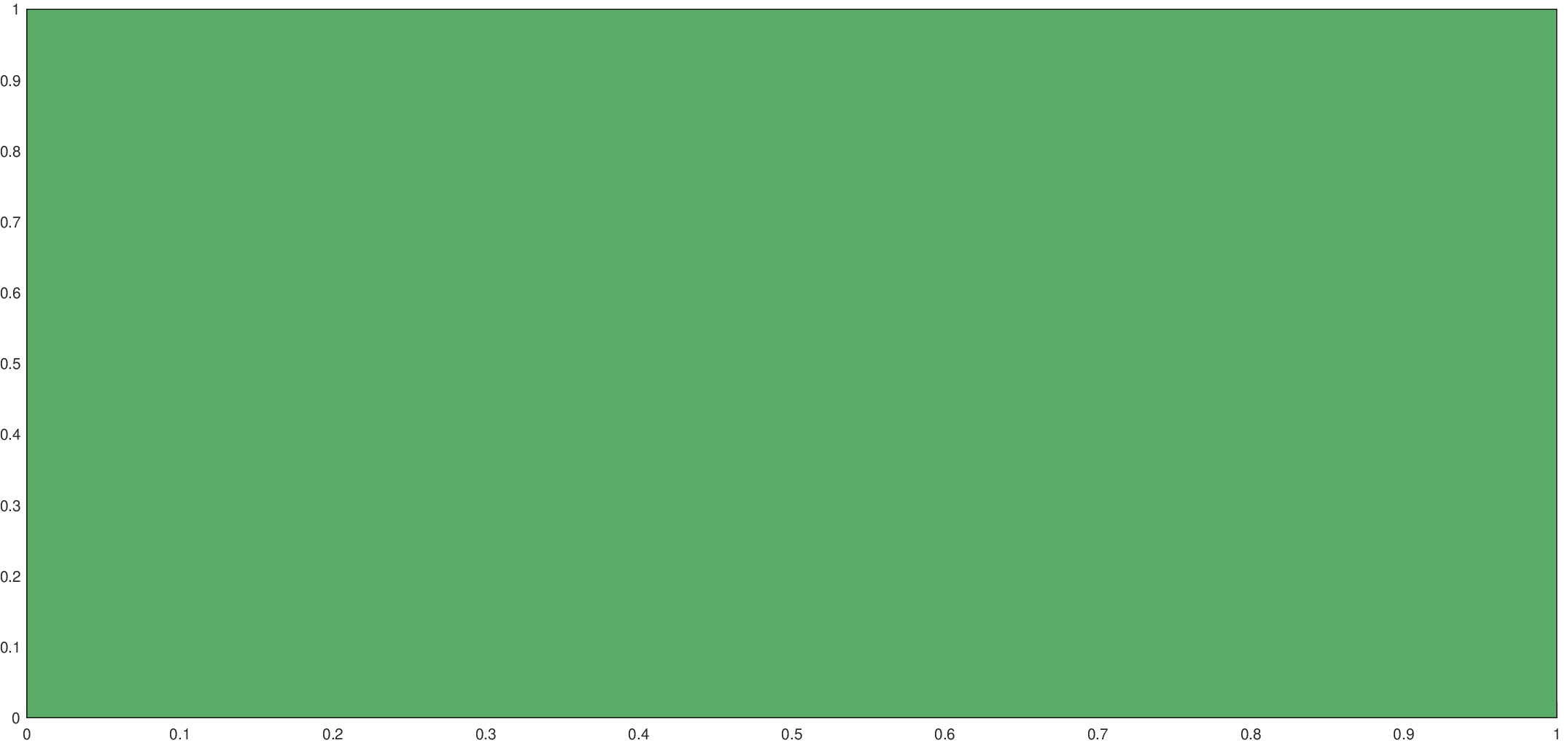}
		\caption{Square domain.}
		\label{square_domain}
	\end{subfigure}
	\hspace{1.5cm}
	\begin{subfigure}{0.3\textwidth}
		\includegraphics[height=6cm,width=6cm]{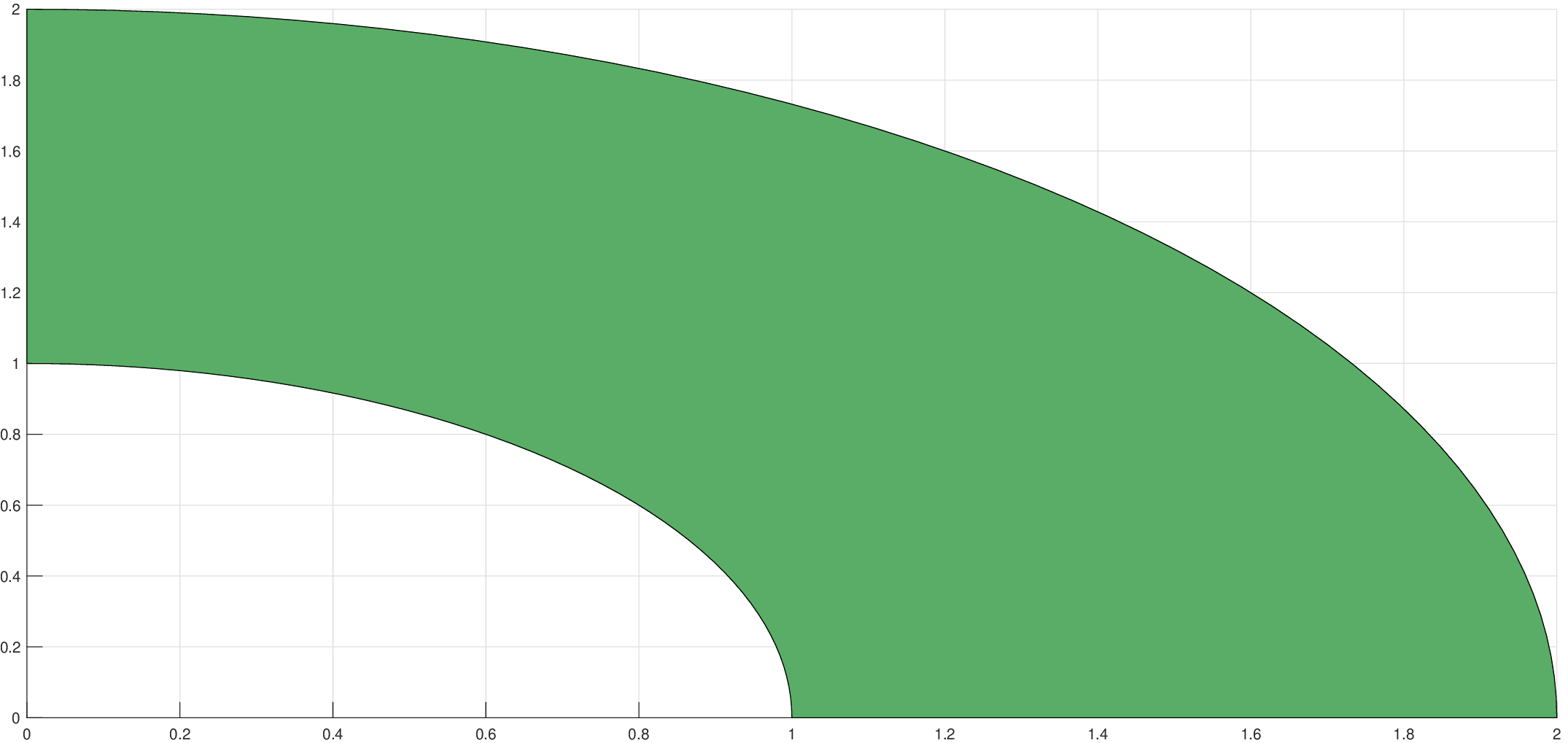}
		\caption{Ring shaped domain.}
		\label{ring_domain}
	\end{subfigure}
	\caption{Computational domains.}
	\label{domains}
\end{figure}

\noindent{\bf{Example 1. Unconditional stability:}} 
In this example, we consider the biharmonic wave equation \eqref{eqbiwave} on the unit square domain $\Omega=(0,1)\times(0,1)$, as illustrated in Figure \ref{square_domain}.  The data $f$ in equation \eqref{eqbiwave} is chosen such that the manufactured solution of \eqref{eqbiwave} is $u(x,y,t)=t^2\sin^2(\pi x)\sin^2(\pi y)$. We apply the proposed IgA-stab formulation \eqref{stabwave}, FEM-stab formulation \eqref{stabwavefem} and the formulation \eqref{disvf} without any stabilization to solve problem \eqref{eqbiwave} numerically. For the IgA-stab method, we choose $\delta=10^{-p}$ for splines of degree $p=2,\ 3,\ 4$. This choice follows from the numerical evidence given in \cite[Figure 1]{Fraschini2024} and the theoretical justification given in \cite[Remark 5.8]{Ferrari2025}. We investigate the stability properties of all three discretization for splines of different regularity. In particular, we use $C^{p-1}$ regularity splines in spatial direction, and in temporal direction we choose $C^{p-1},\ C^{p-2},\ C^0$ continuous splines. To assess stability and convergence, we compute relative errors in $L^2(L^2(\Omega))$, $L^2(H^1_0(\Omega))\cap H^1(L^2(\Omega))$ and $X$ norm. The spatial and temporal mesh sizes are set to be  $h_s=h_t=h=\frac{1}{2},\ \frac{1}{4},\ \frac{1}{8},\ \frac{1}{16},\ \frac{1}{32},\ \frac{1}{64}$. Note that, for this choice of $h_s$ and $h_t$, the CFL condition \eqref{cfl} is not satisfied. In Figure \ref{igastabmax}, we display the errors obtained from IgA-stab formulation with $C^{p-1}$ regularity splines in both space and time direction. From Figure \ref{igastabmax}, we observe that the rate of convergence in $L^2(L^2(\Omega))$ norm is $\mathcal{O}(h^{p+1})$ (except for $p=2$) ,  $L^2(H^1_0(\Omega))\cap H^1(L^2(\Omega))$ norm is $\mathcal{O}(h^p)$ and in $X$ norm, it is $\mathcal{O}(h^{p-1})$ and the IgA-stab formulation is unconditionally stable since the error remains bounded as we reduce the mesh-sizes. However, when we reduce the regularity of splines in temporal direction for the IgA-stab formulation, instability kicks and the relative errors grows with reduction in mesh-size, see Figures \ref{igastab_C_P_2} and \ref{igastab_C_0}. Furthermore, the errors displayed in Figure \ref{femstabc0} shows the unconditional stability of the FEM-stab formulation \eqref{stabwavefem} for splines of $C^0$ regularity in temporal direction. In contrast, Figures \ref{femstabmax} and \ref{femstabc_p_2} reveal that the FEM-stab formulation becomes unstable when applied to splines with higher regularity ($C^{p-1}$ and $C^{p-2}$) in temporal variable. Finally, we examine the unstabilized formulation \eqref{disvf}. The errors obtained from the formulation \eqref{disvf} without any stabilization for splines of regularity $C^{p-1},\ C^{p-2},\ C^0$ in temporal direction are displayed in Figures \ref{nostabmax}, \ref{nostabCP_P_2} and \ref{nostabC_0} respectively. We observe from these figures that, since the CFL condition \eqref{cfl} is not satisfied for the chosen mesh-sizes, instability emerges for the splines of regularity $C^{p-1},\ C^{p-2},\ C^0$ in temporal direction.  Based on these observations, in Table \ref{stabcomp}, we summarize the stability properties of all three discretization according to the regularity of splines.
\begin{table}[h]
	\centering
	\caption{Comparison of stability obtained for the IgA FEM and without stabilized discretization with respect to different regularity of splines}
	\label{stabcomp}
	\begin{tabular}{c c c c}
		\hline
		Stabilization & Spline regularity in space& Spline regularity in time & Stability\\
		\hline
		None& $C^k,1\leq k\leq p-1$ &$C^k,0\leq k\leq p-1$& Conditionally stable\\
		IgA& $C^k,1\leq k\leq p-1$ &$C^{p-1}$& Stable\\
		IgA& $C^k,1\leq k\leq p-1$ &$C^k,0\leq k< p-1$& Conditionally stable\\
		FEM& $C^k,1\leq k\leq p-1$ &$C^0$& Stable\\
		FEM& $C^k,1\leq k\leq p-1$ &$C^k,0< k\leq p-1$& Conditionally stable\\
		\hline
	\end{tabular}
\end{table}
\begin{figure}[htbp]
	\begin{subfigure}{0.3\textwidth}
		\includegraphics[height=6cm,width=5cm]{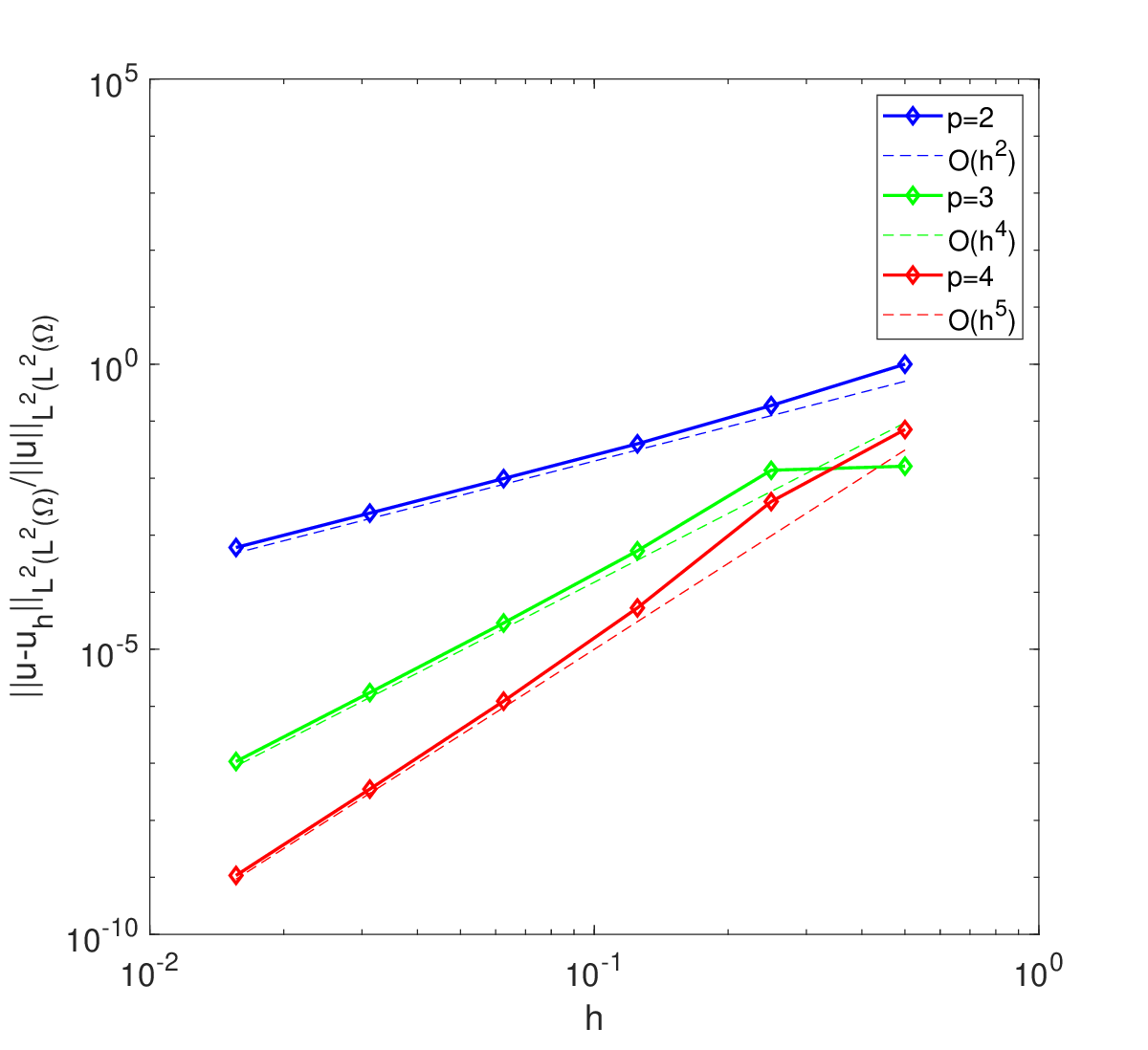}
		\caption[Caption]{}
		\label{}
	\end{subfigure}
	\hspace{0.55cm}
	\begin{subfigure}{0.3\textwidth}
		\includegraphics[height=6cm,width=5cm]{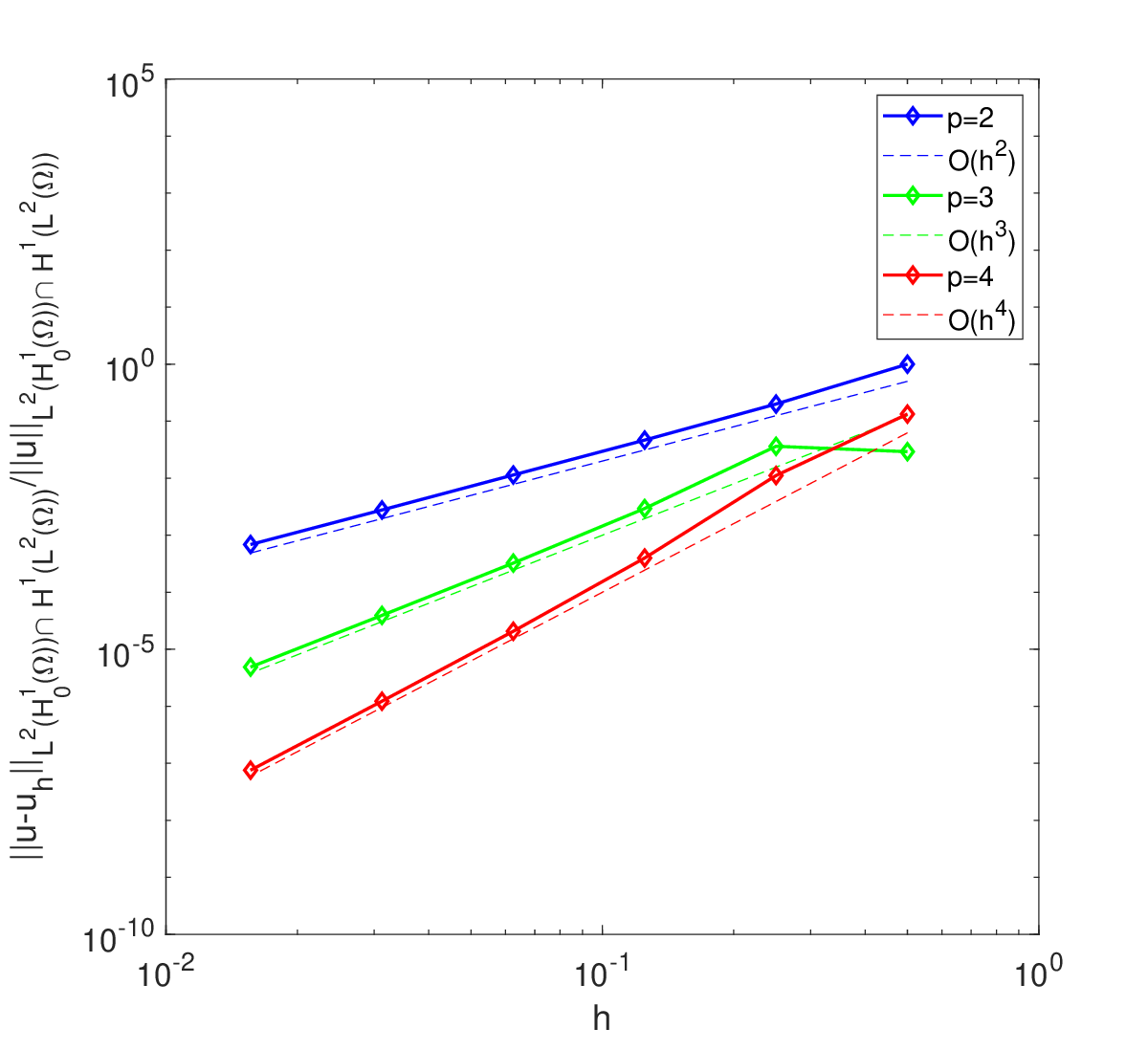}
		\caption[Caption]{}
	\end{subfigure}
	\hspace{0.55cm}
	\begin{subfigure}{0.3\textwidth}
		\includegraphics[height=6cm,width=5cm]{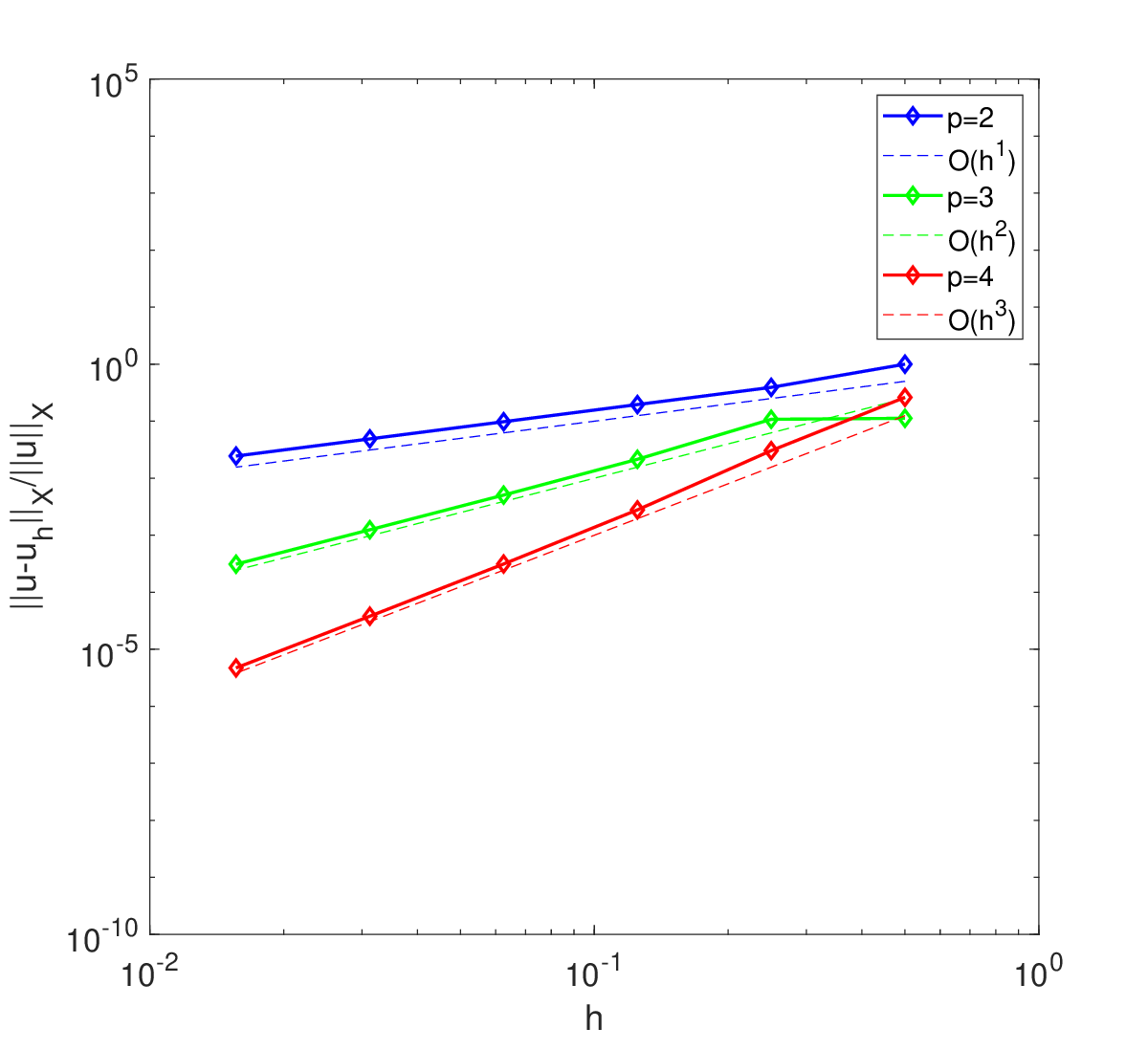}
		\caption[Caption]{}
	\end{subfigure}
	
	\caption{Relative errors in (a) $L^2(L^2(\Omega))$ norm, (b) $L^2(H^1_0(\Omega))\cap H^1(L^2(\Omega))$ norm and (c) $X$ norm for the IgA stabilization method with splines of maximum regularity in both space and time direction for Square domain.}
	\label{igastabmax}
\end{figure}
\begin{figure}[htbp]
	\begin{subfigure}{0.3\textwidth}
		\includegraphics[height=6cm,width=5cm]{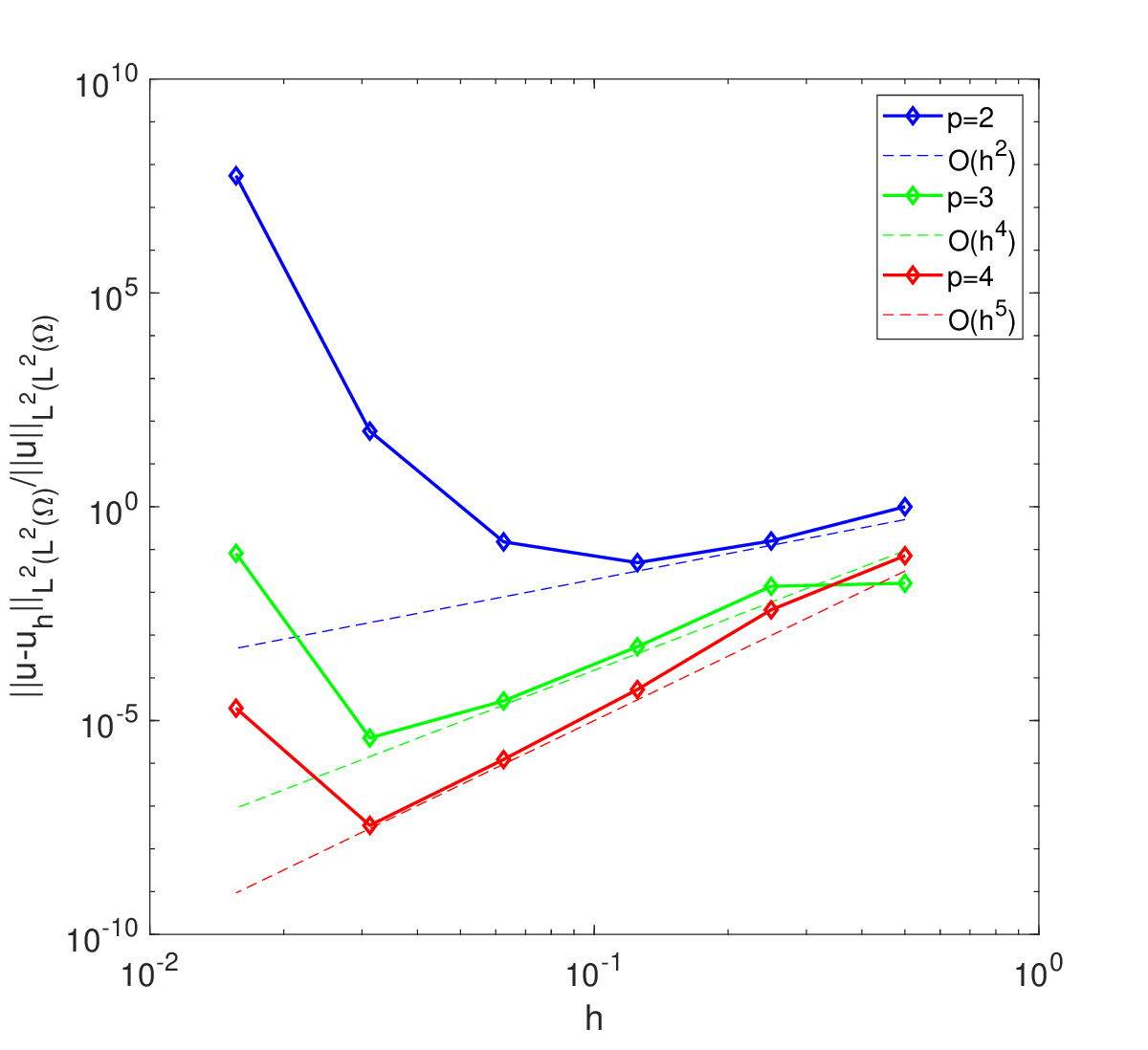}
		\caption[Caption]{}
		\label{}
	\end{subfigure}
	\hspace{0.55cm}
	\begin{subfigure}{0.3\textwidth}
		\includegraphics[height=6cm,width=5cm]{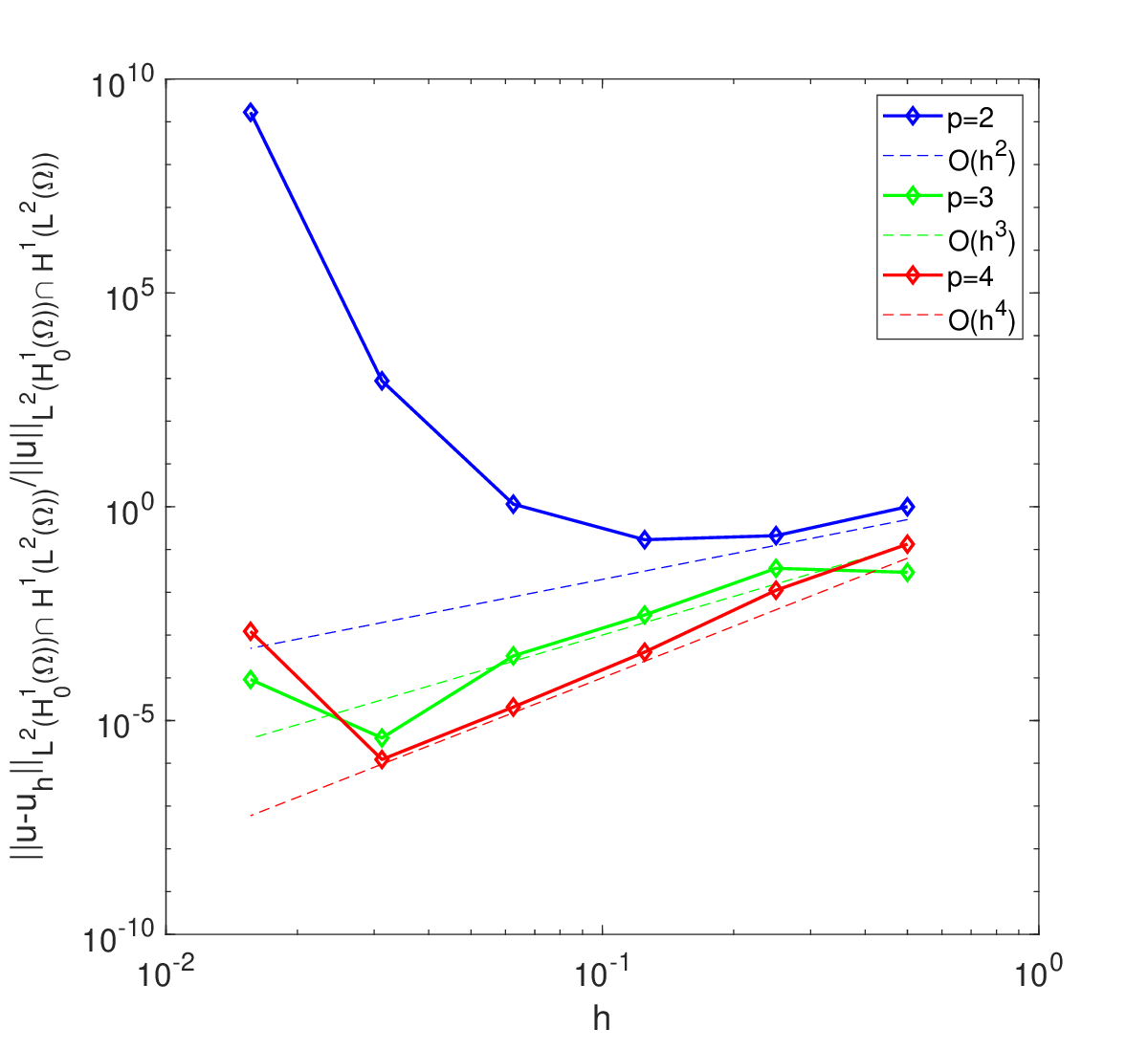}
		\caption[Caption]{}
	\end{subfigure}
	\hspace{0.55cm}
	\begin{subfigure}{0.3\textwidth}
		\includegraphics[height=6cm,width=5cm]{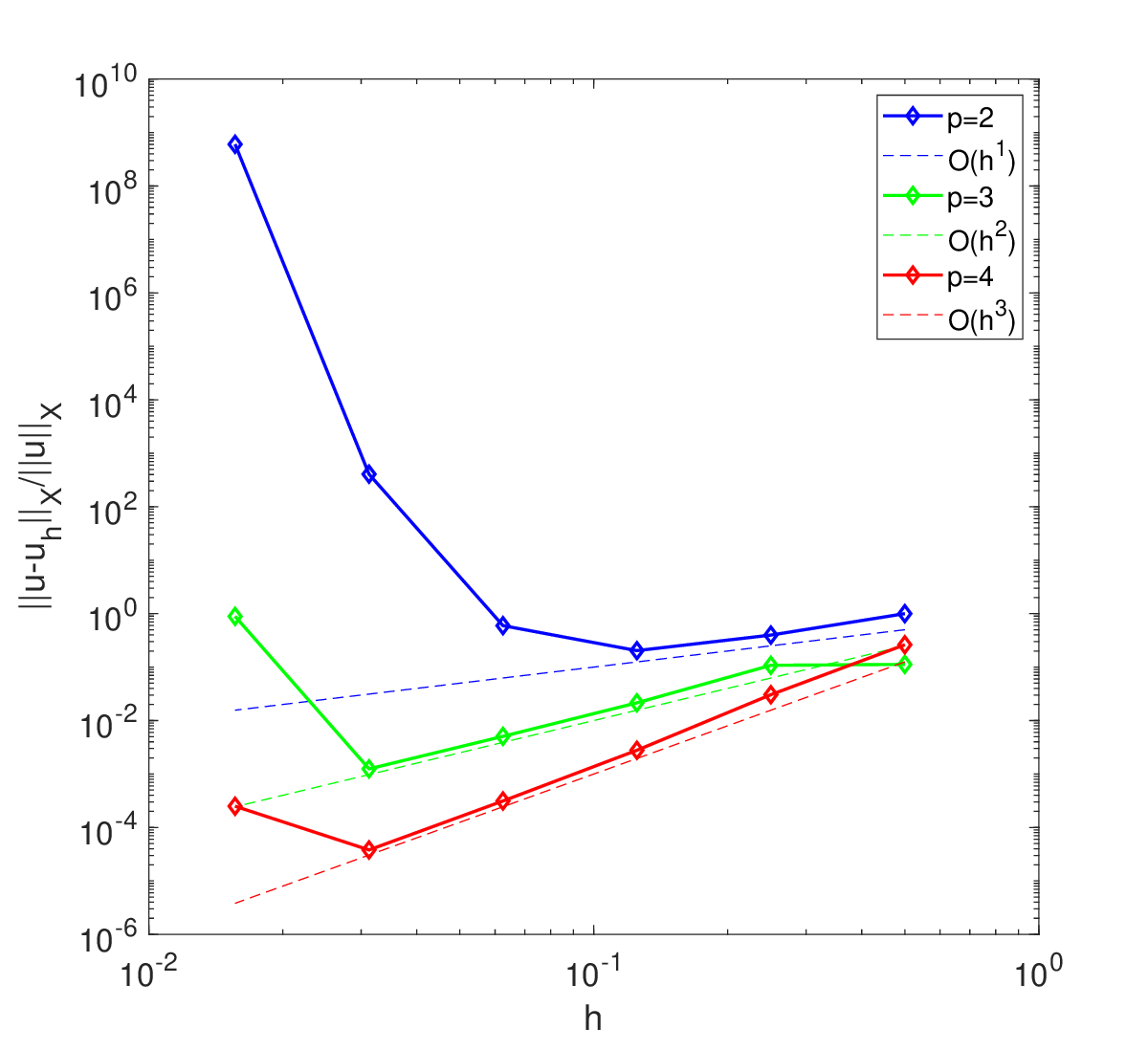}
		\caption[Caption]{}
	\end{subfigure}
	
	\caption{Relative errors in (a) $L^2(L^2(\Omega))$ norm, (b) $L^2(H^1_0(\Omega))\cap H^1(L^2(\Omega))$ norm and (c) $X$ norm for the IgA stabilization method with splines of maximum regularity in space and $C^{p-2}$ regularity in time direction for Square domain.}
	\label{igastab_C_P_2}
\end{figure}
\begin{figure}[htbp]
	\begin{subfigure}{0.3\textwidth}
		\includegraphics[height=6cm,width=5cm]{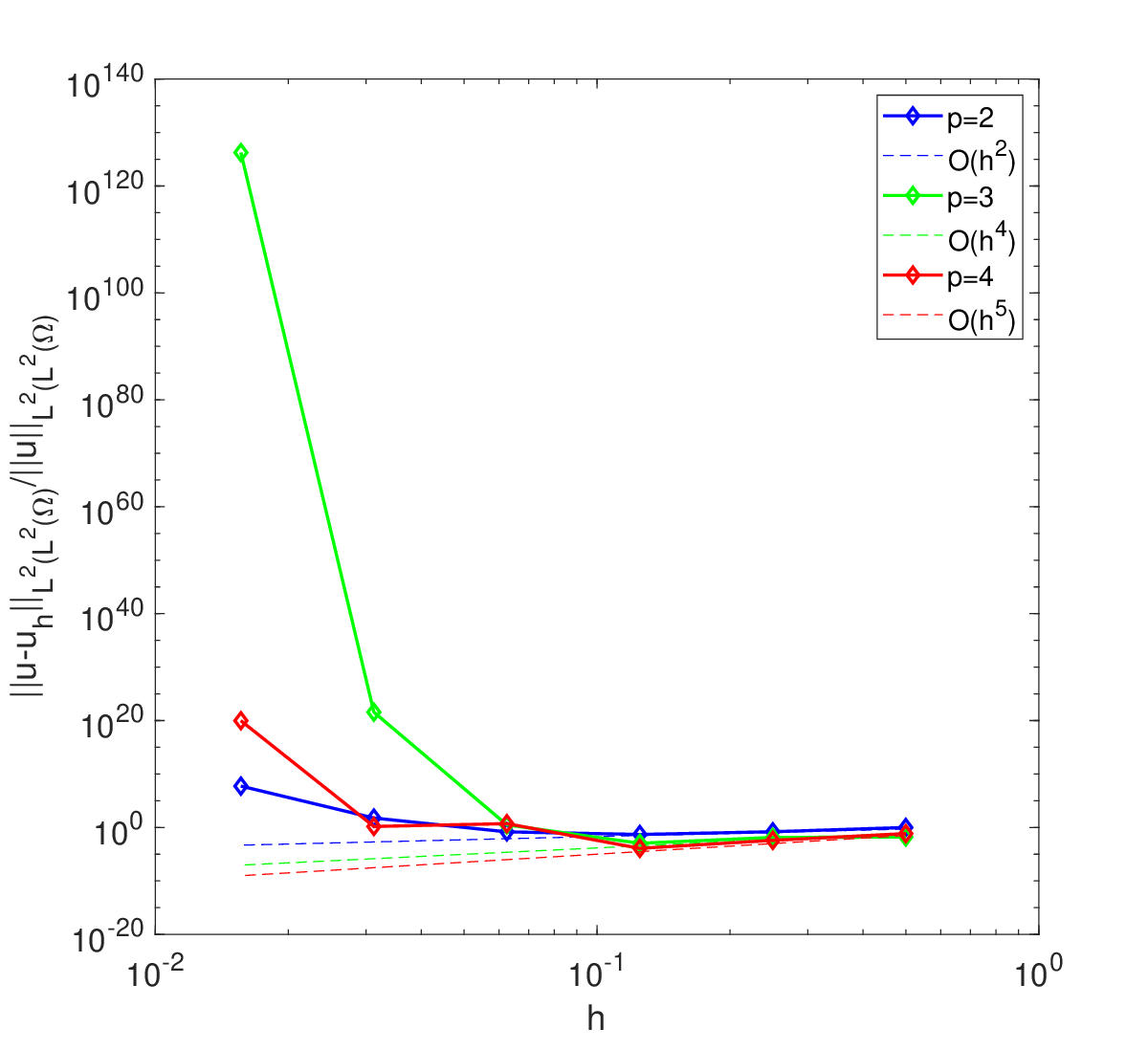}
		\caption[Caption]{}
		\label{}
	\end{subfigure}
	\hspace{0.55cm}
	\begin{subfigure}{0.3\textwidth}
		\includegraphics[height=6cm,width=5cm]{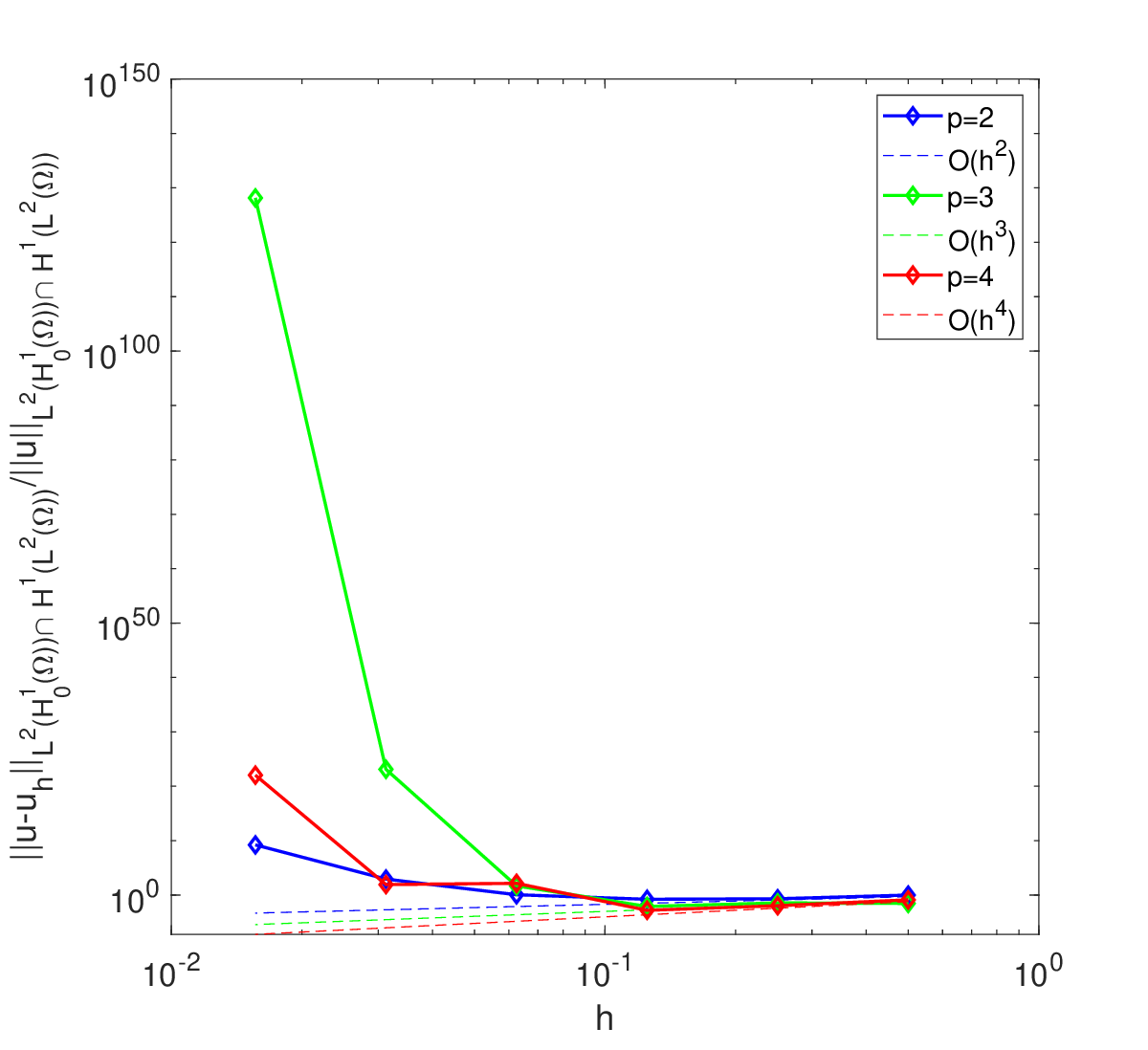}
		\caption[Caption]{}
	\end{subfigure}
	\hspace{0.55cm}
	\begin{subfigure}{0.3\textwidth}
		\includegraphics[height=6cm,width=5cm]{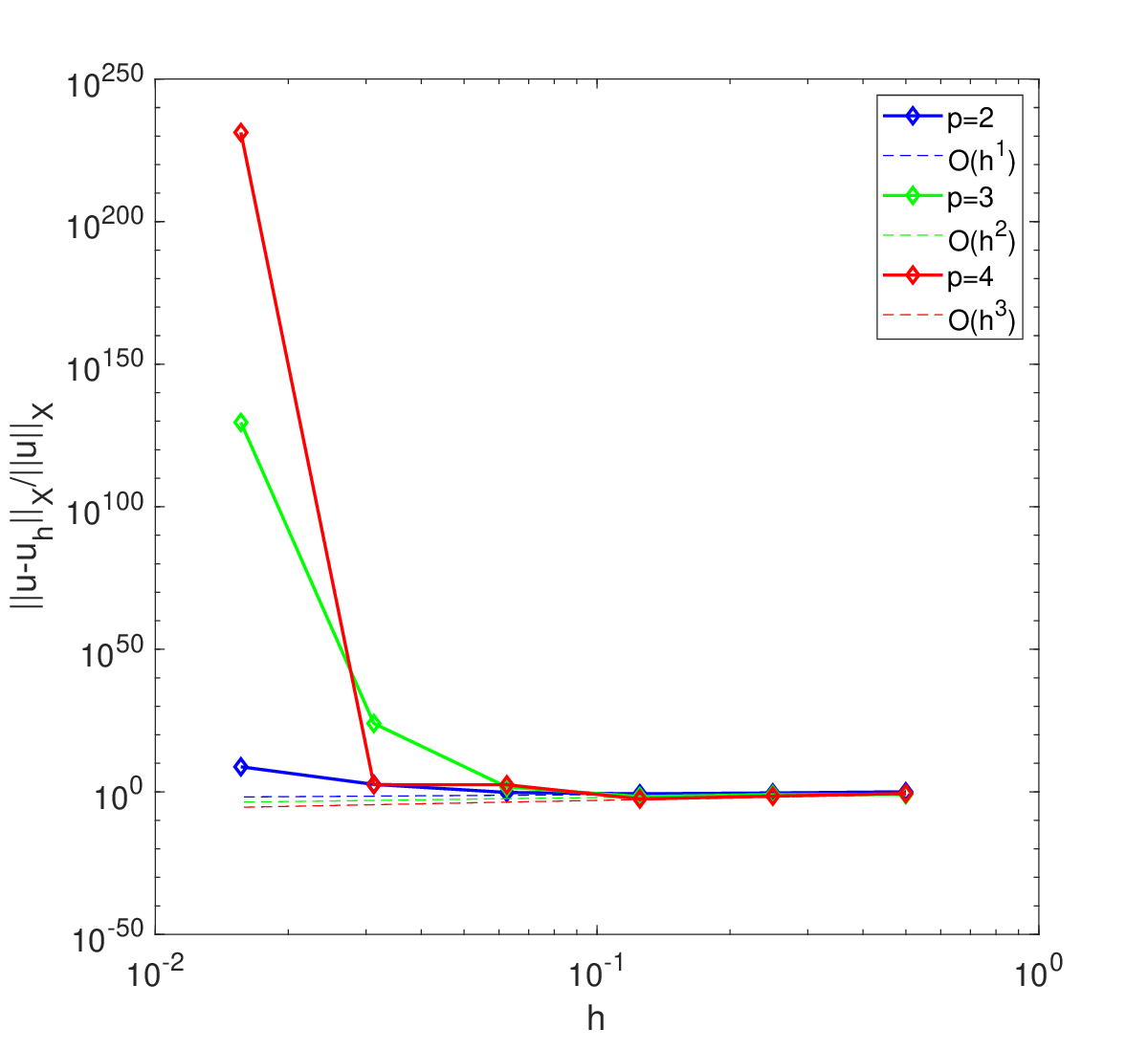}
		\caption[Caption]{}
	\end{subfigure}
	
	\caption{Relative errors in (a) $L^2(L^2(\Omega))$ norm, (b) $L^2(H^1_0(\Omega))\cap H^1(L^2(\Omega))$ norm and (c) $X$ norm for the IgA stabilization method with splines of maximum regularity in space and $C^0$ regularity in time direction for Square domain.}
	\label{igastab_C_0}
\end{figure}
\begin{figure}[htbp]
	\begin{subfigure}{0.3\textwidth}
		\includegraphics[height=6cm,width=5cm]{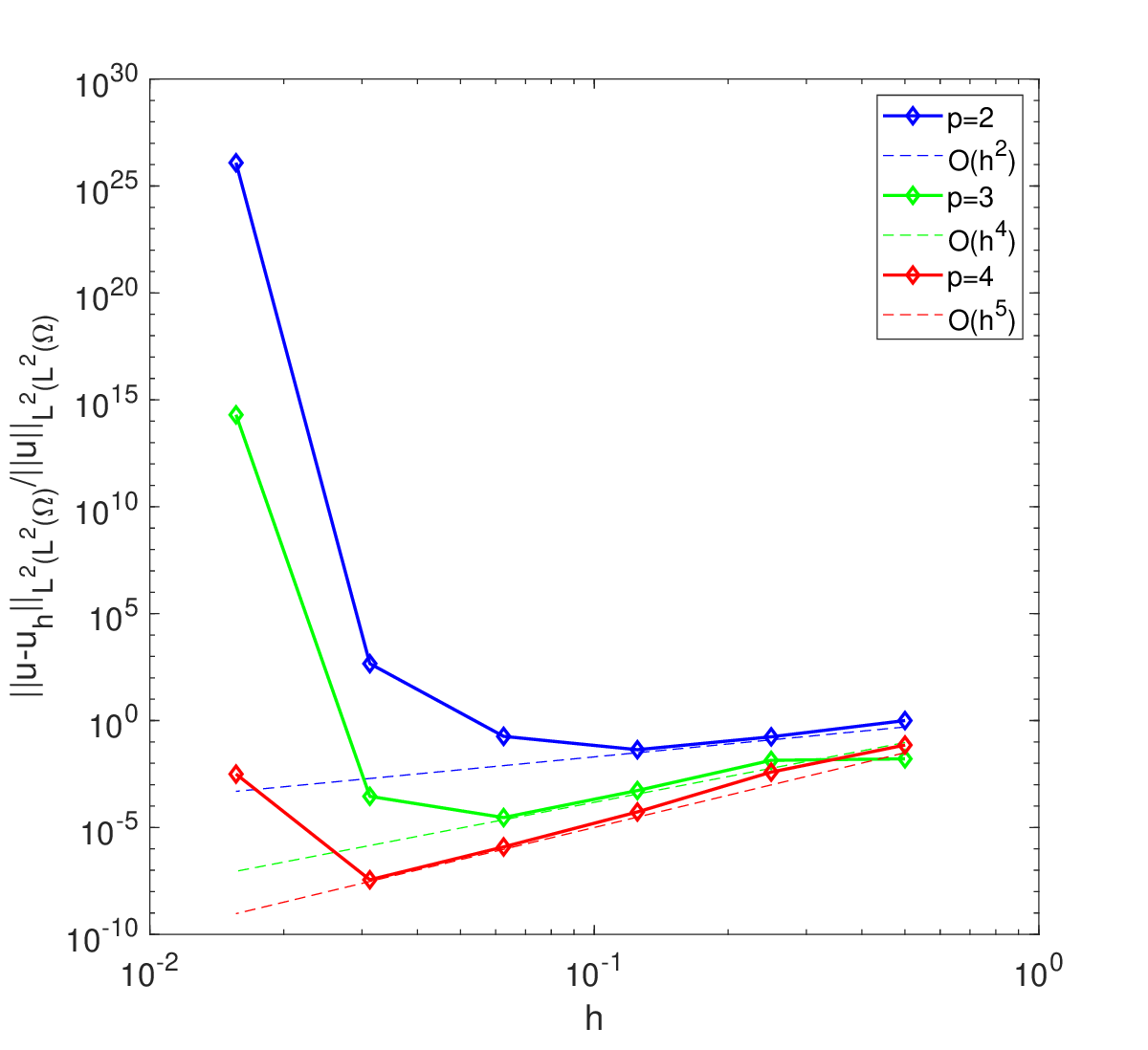}
		\caption[Caption]{}
		\label{}
	\end{subfigure}
	\hspace{0.55cm}
	\begin{subfigure}{0.3\textwidth}
		\includegraphics[height=6cm,width=5cm]{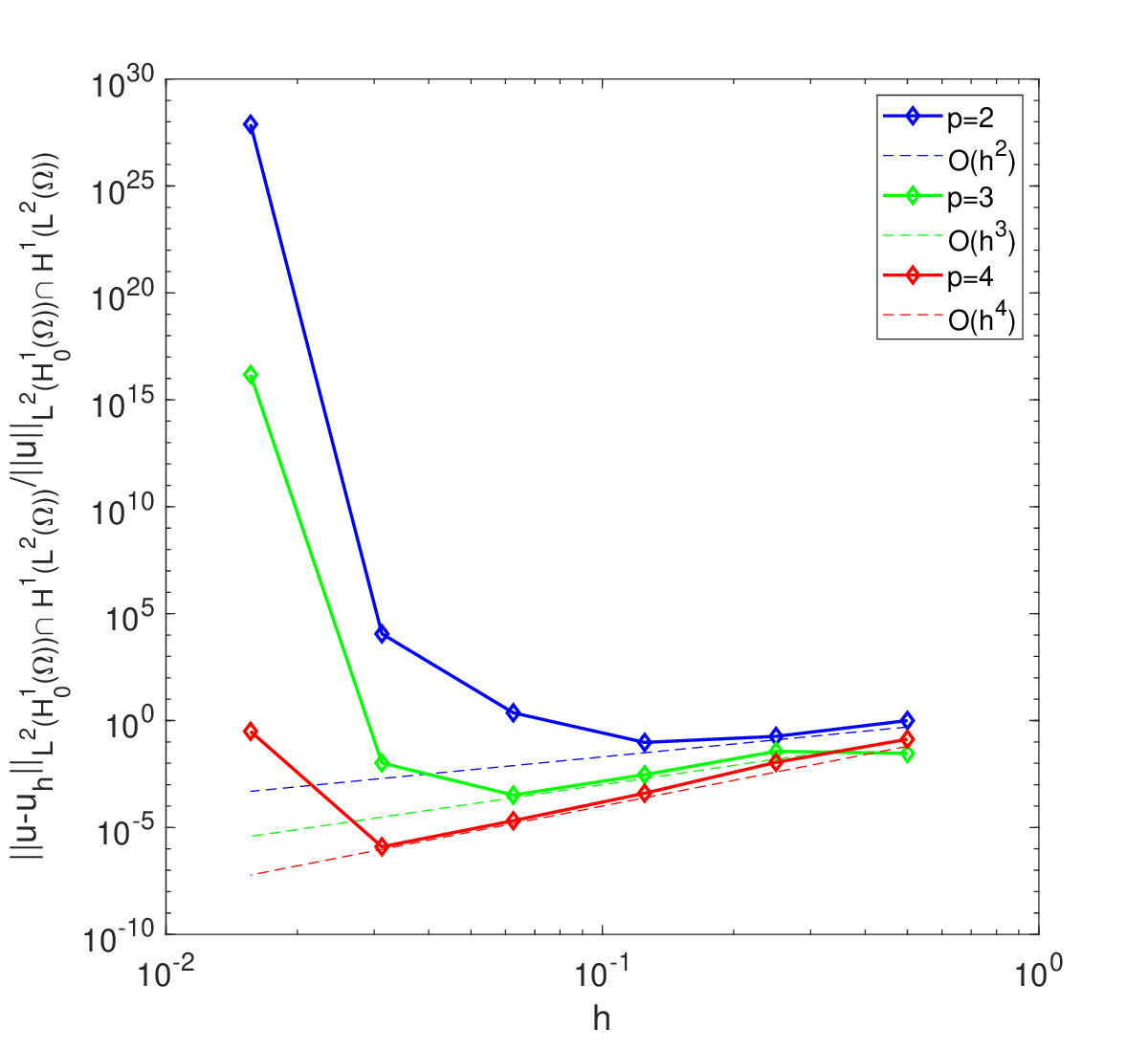}
		\caption[Caption]{}
	\end{subfigure}
	\hspace{0.55cm}
	\begin{subfigure}{0.3\textwidth}
		\includegraphics[height=6cm,width=5cm]{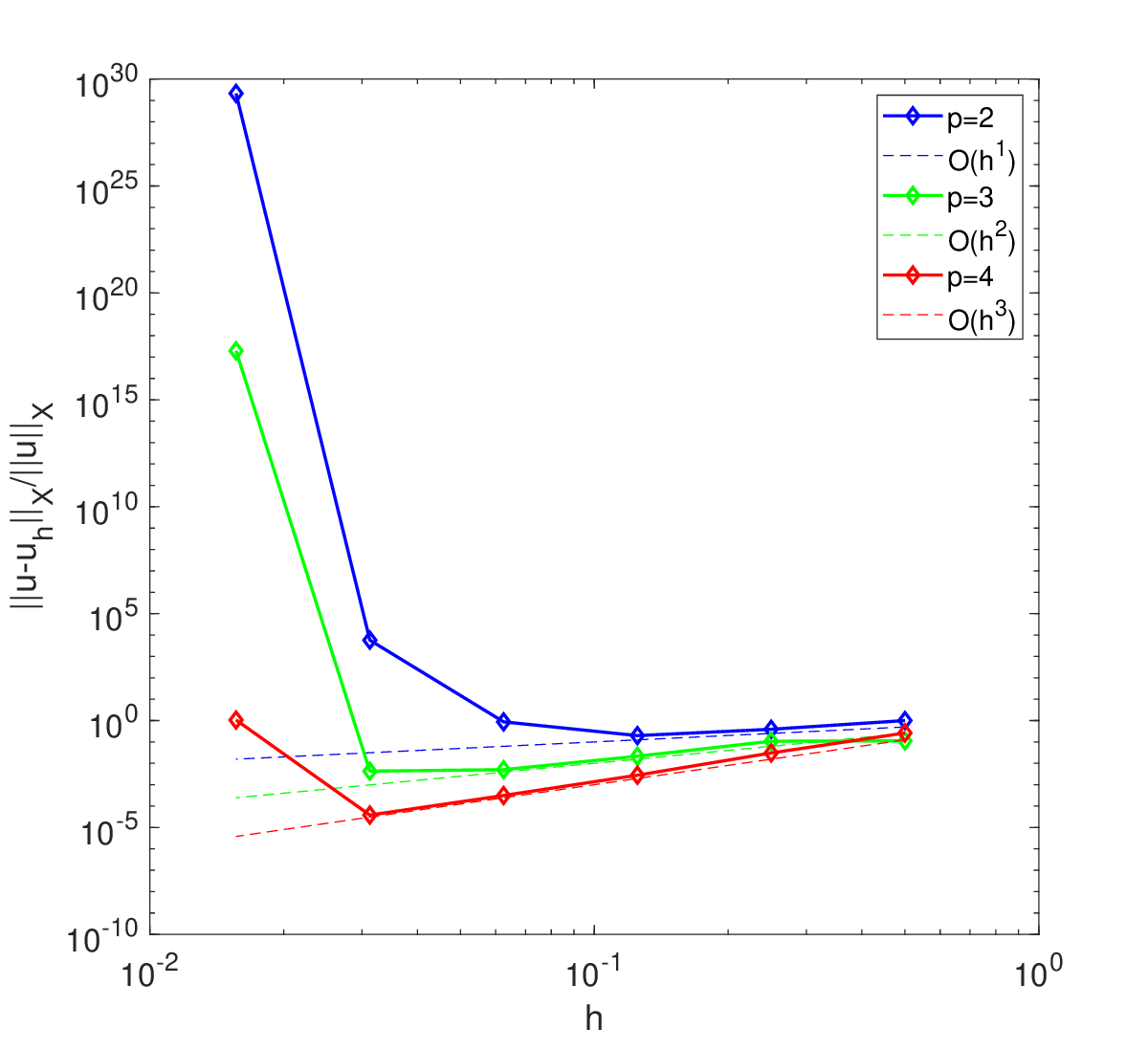}
		\caption[Caption]{}
	\end{subfigure}
	
	\caption{Relative errors in (a) $L^2(L^2(\Omega))$ norm, (b) $L^2(H^1_0(\Omega))\cap H^1(L^2(\Omega))$ norm and (c) $X$ norm for the FEM stabilization method with splines of maximum regularity in both space and time direction for Square domain.}
	\label{femstabmax}
\end{figure}

\begin{figure}[htbp]
	\begin{subfigure}{0.3\textwidth}
		\includegraphics[height=6cm,width=5cm]{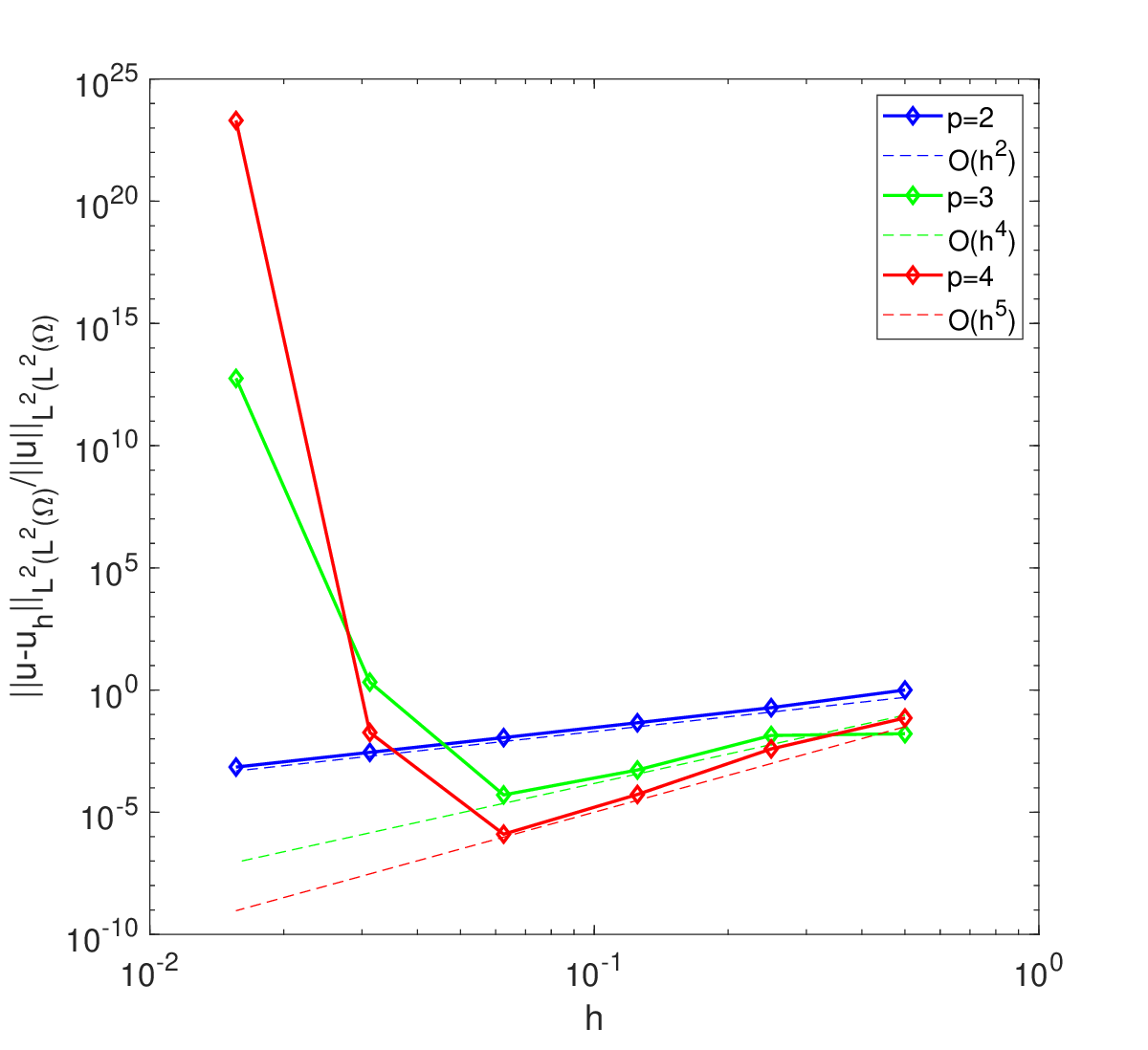}
		\caption[Caption]{}
		\label{}
	\end{subfigure}
	\hspace{0.55cm}
	\begin{subfigure}{0.3\textwidth}
		\includegraphics[height=6cm,width=5cm]{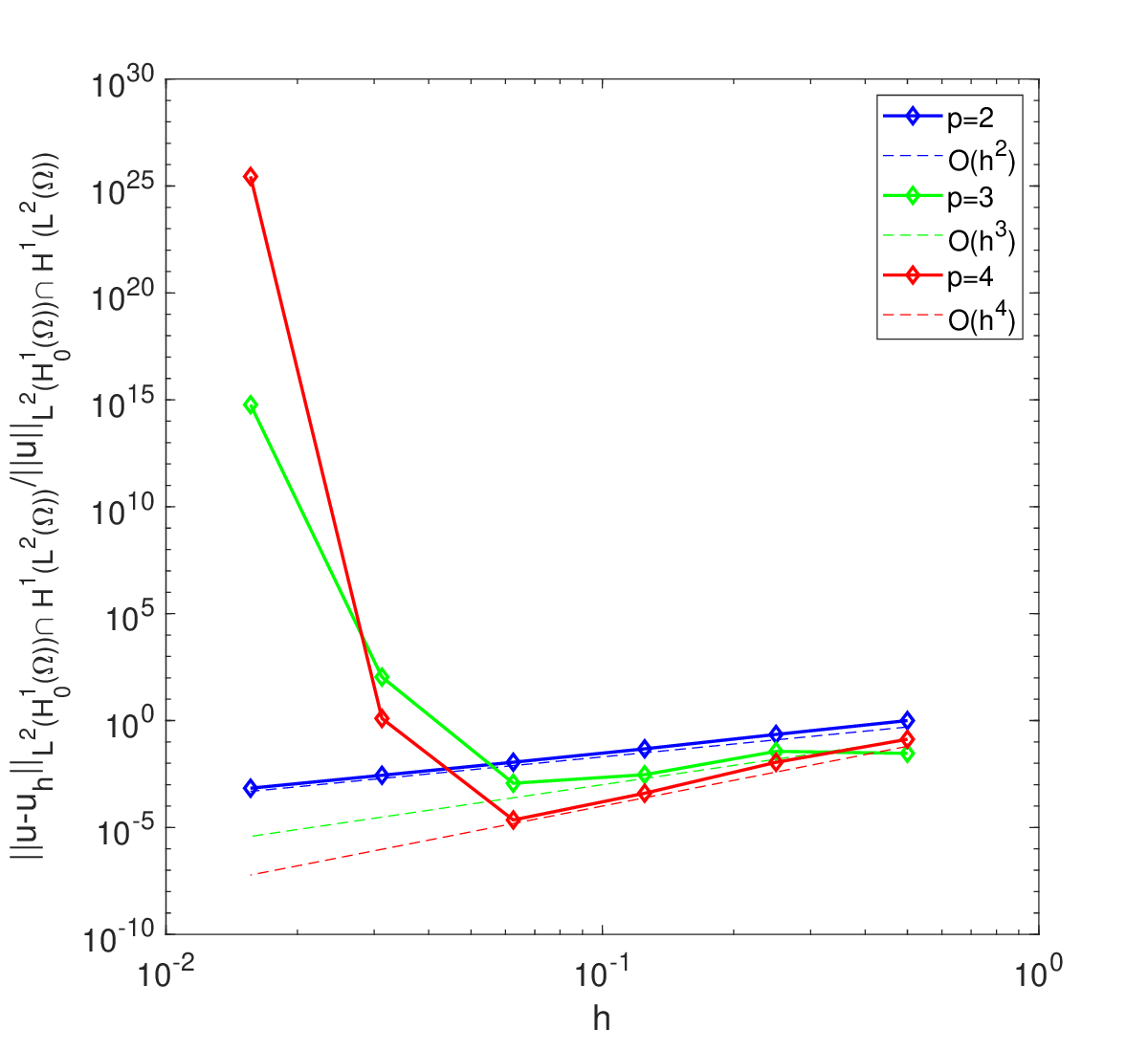}
		\caption[Caption]{}
	\end{subfigure}
	\hspace{0.55cm}
	\begin{subfigure}{0.3\textwidth}
		\includegraphics[height=6cm,width=5cm]{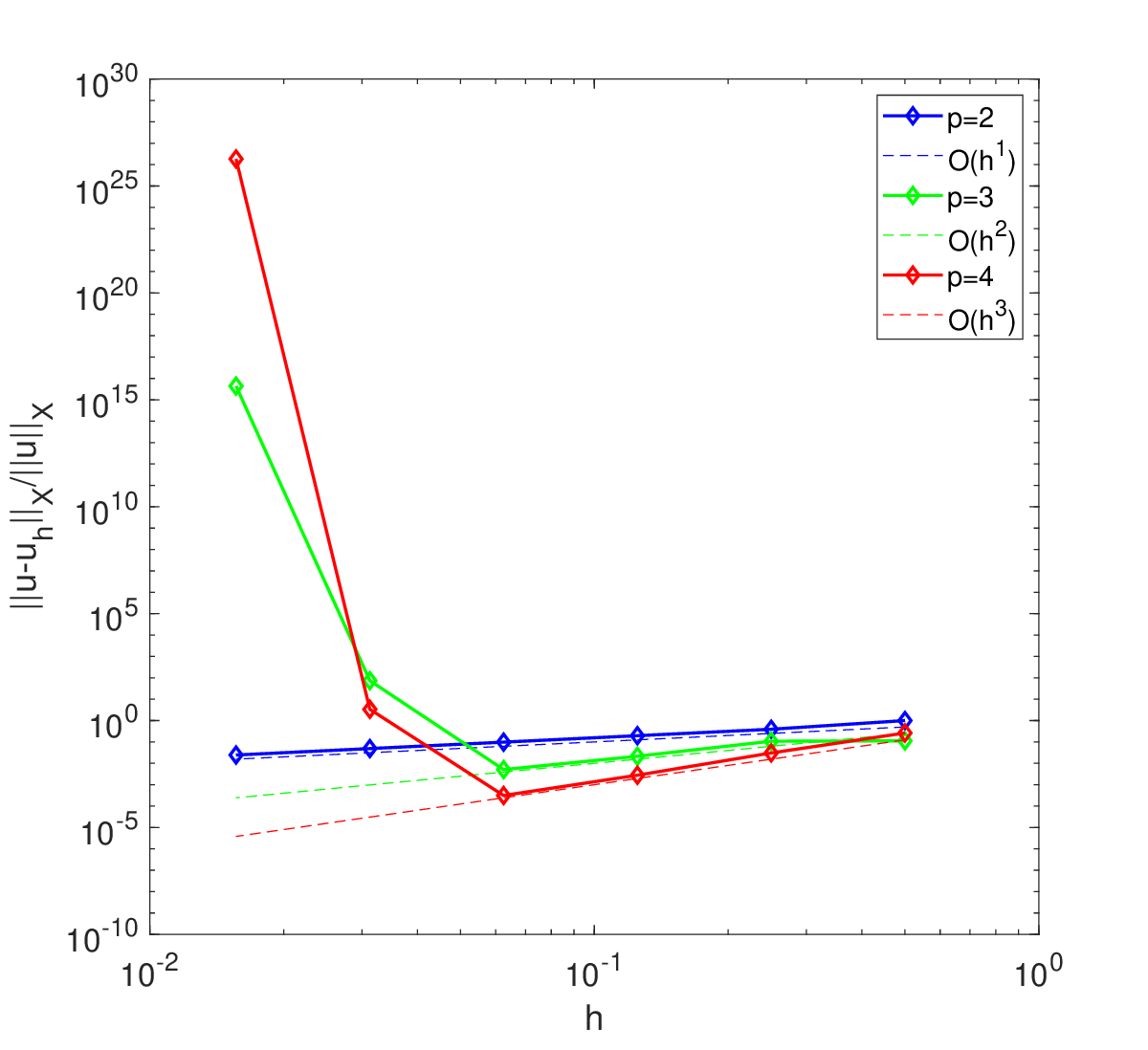}
		\caption[Caption]{}
	\end{subfigure}
	
	\caption{Relative errors in (a) $L^2(L^2(\Omega))$ norm, (b) $L^2(H^1_0(\Omega))\cap H^1(L^2(\Omega))$ norm and (c) $X$ norm for the FEM stabilization method with splines of maximum regularity in space and $C^{p-2}$ regularity in time direction for Square domain.}
	\label{femstabc_p_2}
\end{figure}

\begin{figure}[htbp]
	\begin{subfigure}{0.3\textwidth}
		\includegraphics[height=6cm,width=5cm]{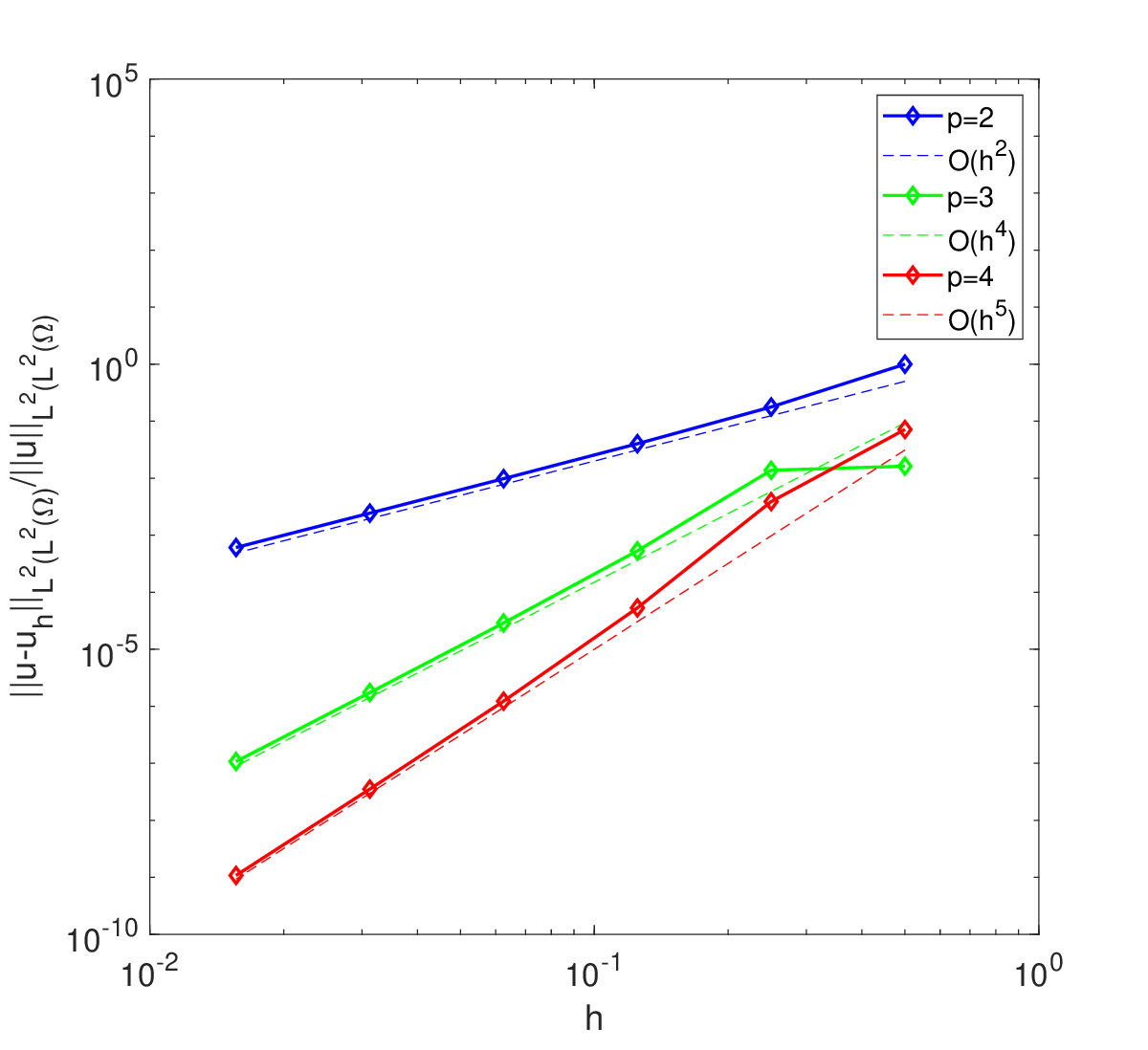}
		\caption[Caption]{}
		\label{}
	\end{subfigure}
	\hspace{0.55cm}
	\begin{subfigure}{0.3\textwidth}
		\includegraphics[height=6cm,width=5cm]{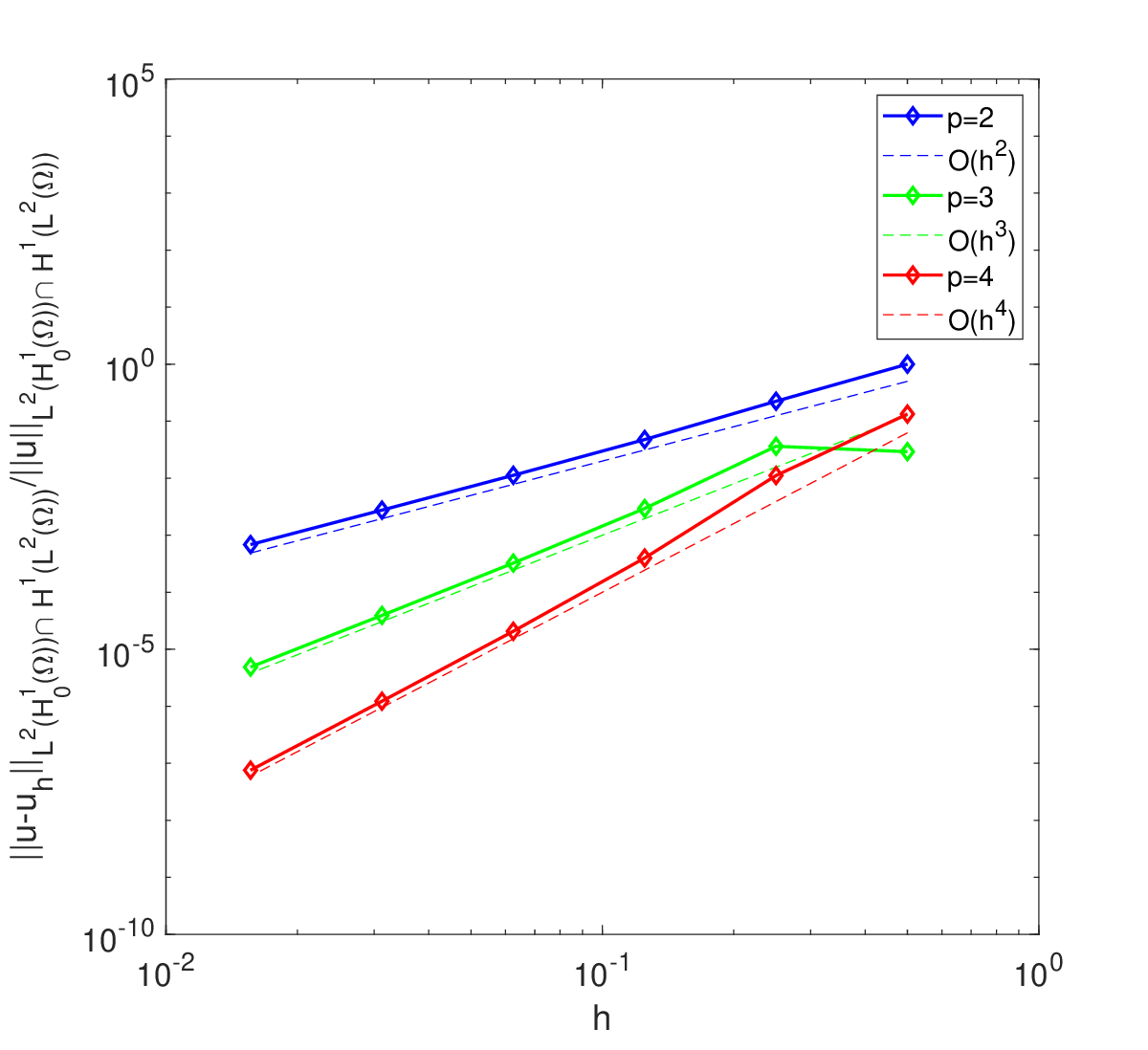}
		\caption[Caption]{}
	\end{subfigure}
	\hspace{0.55cm}
	\begin{subfigure}{0.3\textwidth}
		\includegraphics[height=6cm,width=5cm]{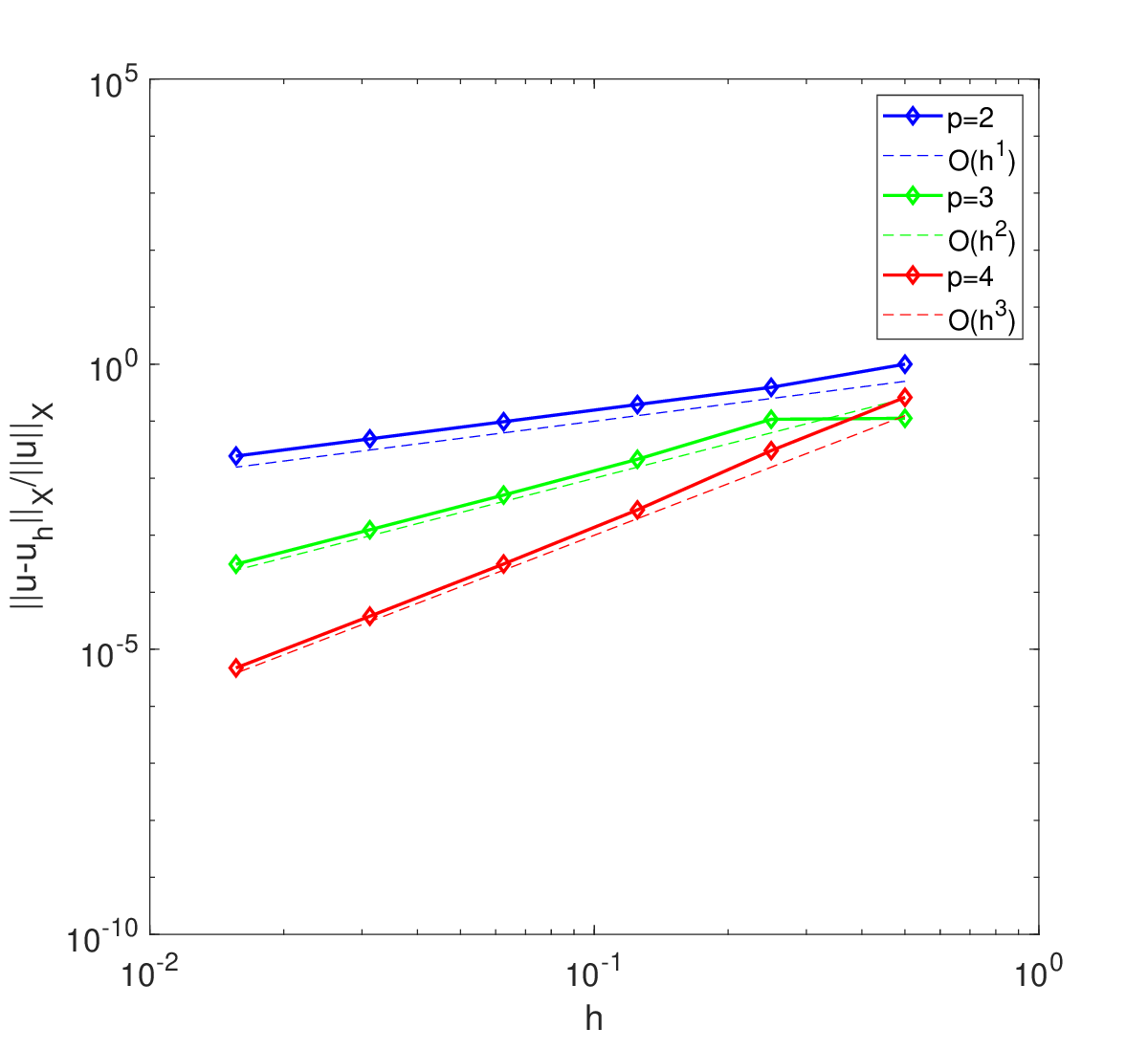}
		\caption[Caption]{}
	\end{subfigure}
	
	\caption{Relative errors in (a) $L^2(L^2(\Omega))$ norm, (b) $L^2(H^1_0(\Omega))\cap H^1(L^2(\Omega))$ norm and (c) $X$ norm for the FEM stabilization method with splines of maximum regularity in space and $C^0$ regularity in time direction for Square domain.}
	\label{femstabc0}
\end{figure}

\begin{figure}[htbp]
	\begin{subfigure}{0.3\textwidth}
		\includegraphics[height=6cm,width=5cm]{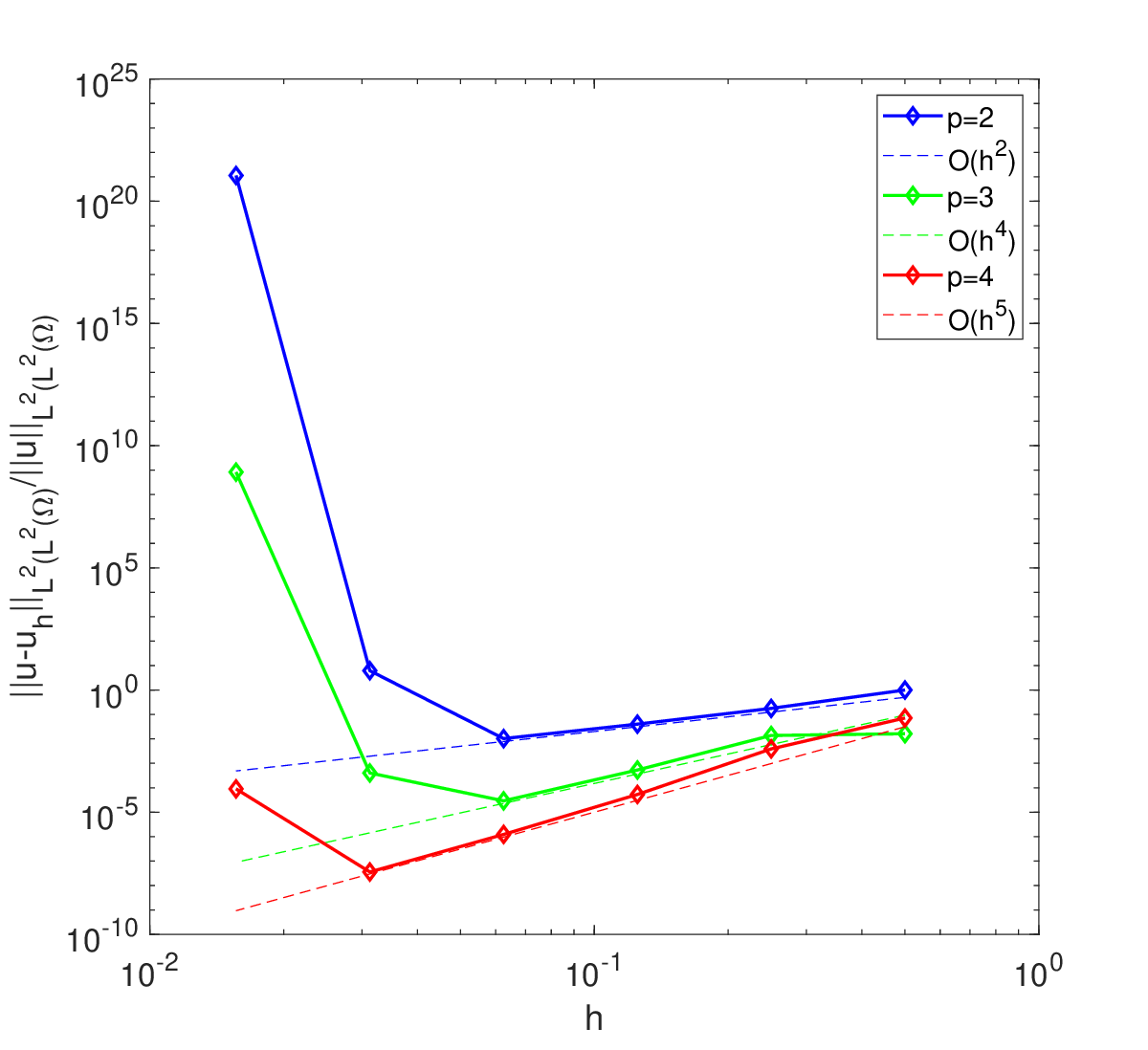}
		\caption[Caption]{}
		\label{}
	\end{subfigure}
	\hspace{0.55cm}
	\begin{subfigure}{0.3\textwidth}
		\includegraphics[height=6cm,width=5cm]{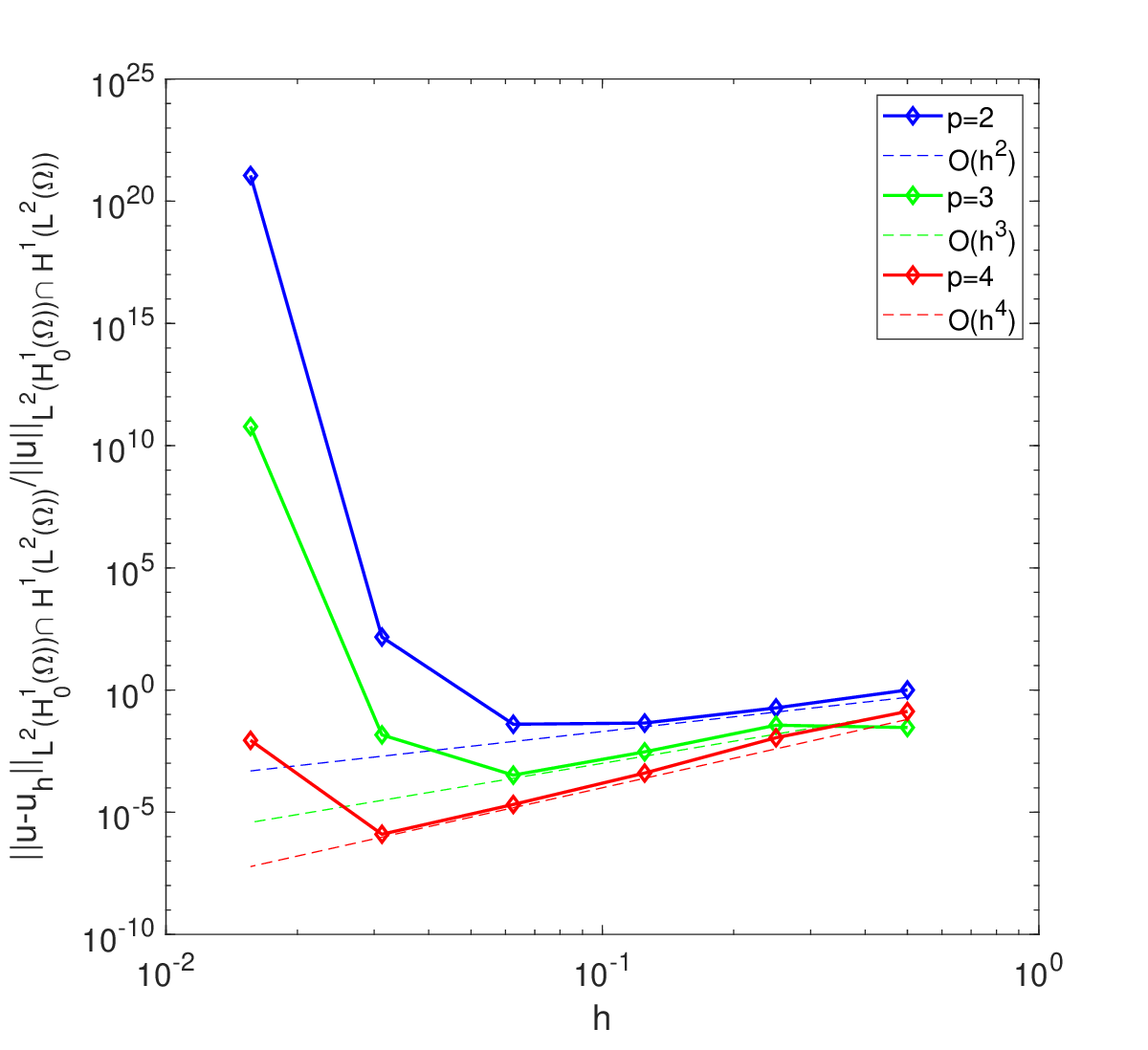}
		\caption[Caption]{}
	\end{subfigure}
	\hspace{0.55cm}
	\begin{subfigure}{0.3\textwidth}
		\includegraphics[height=6cm,width=5cm]{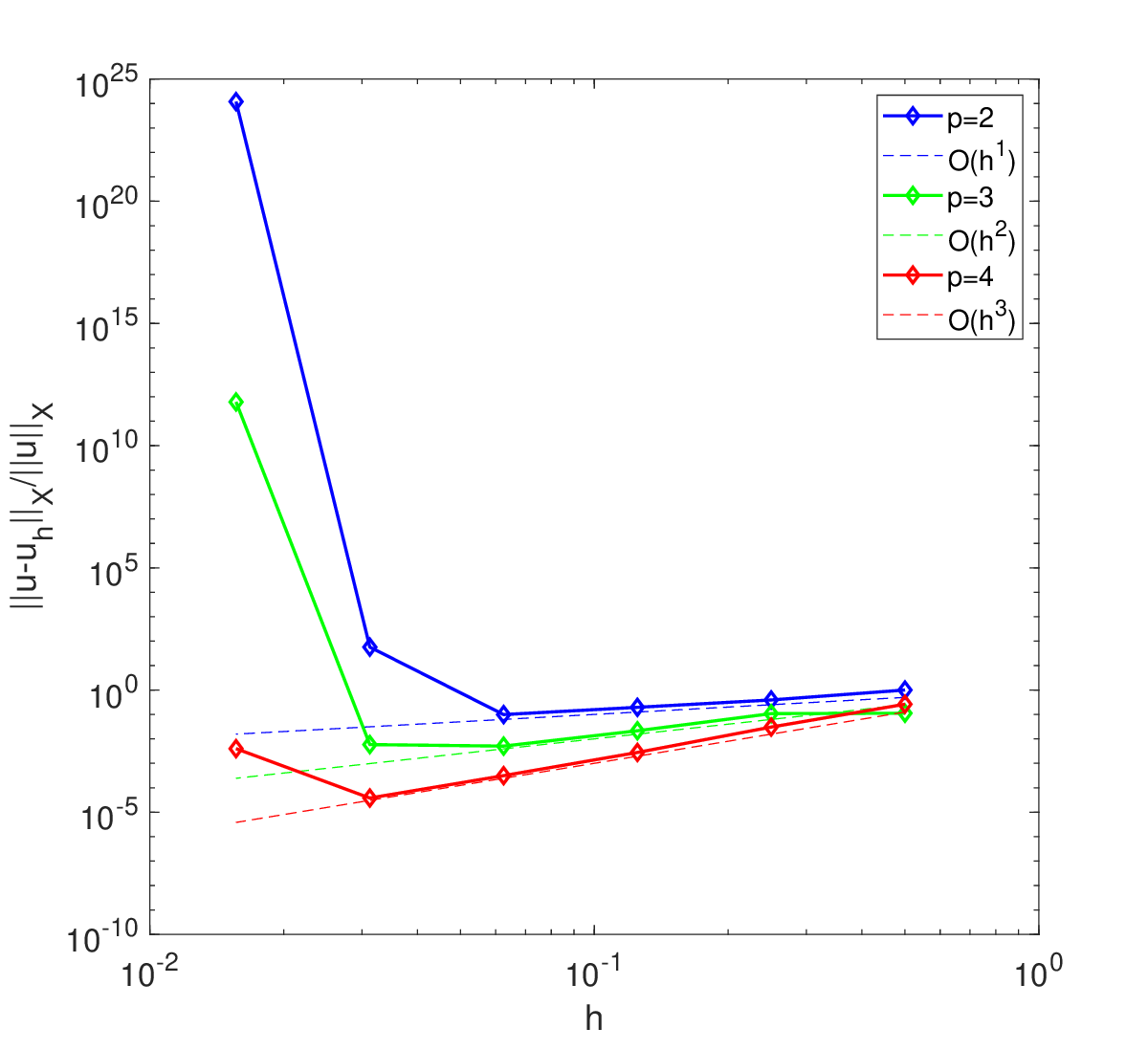}
		\caption[Caption]{}
	\end{subfigure}
	
	\caption{Relative errors in (a) $L^2(L^2(\Omega))$ norm, (b) $L^2(H^1_0(\Omega))\cap H^1(L^2(\Omega))$ norm and (c) $X$ norm for the without stabilization method with splines of maximum regularity in both space and time direction for Square domain.}
	\label{nostabmax}
\end{figure}
\begin{figure}[htbp]
	\begin{subfigure}{0.3\textwidth}
		\includegraphics[height=6cm,width=5cm]{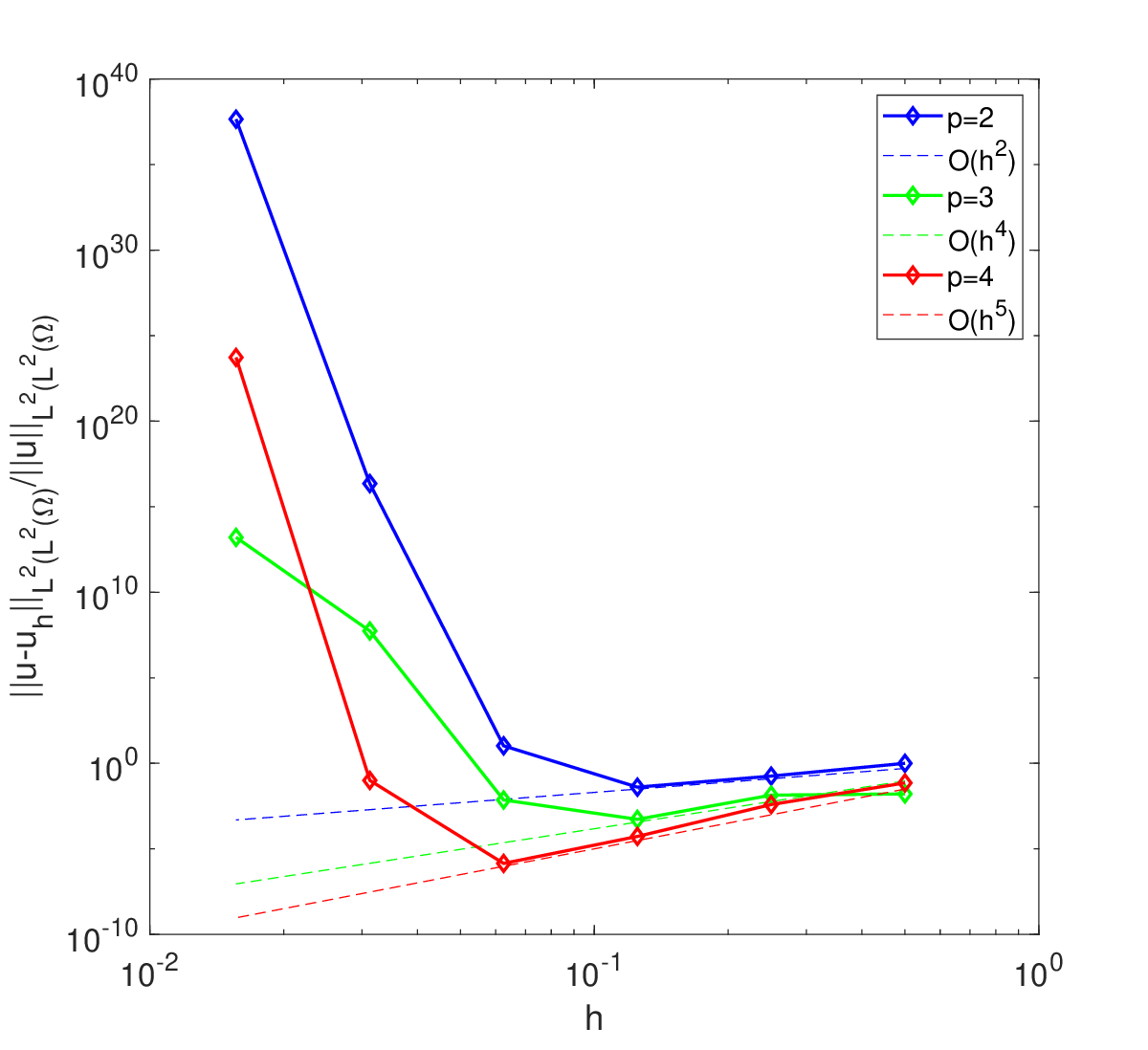}
		\caption[Caption]{}
		\label{}
	\end{subfigure}
	\hspace{0.55cm}
	\begin{subfigure}{0.3\textwidth}
		\includegraphics[height=6cm,width=5cm]{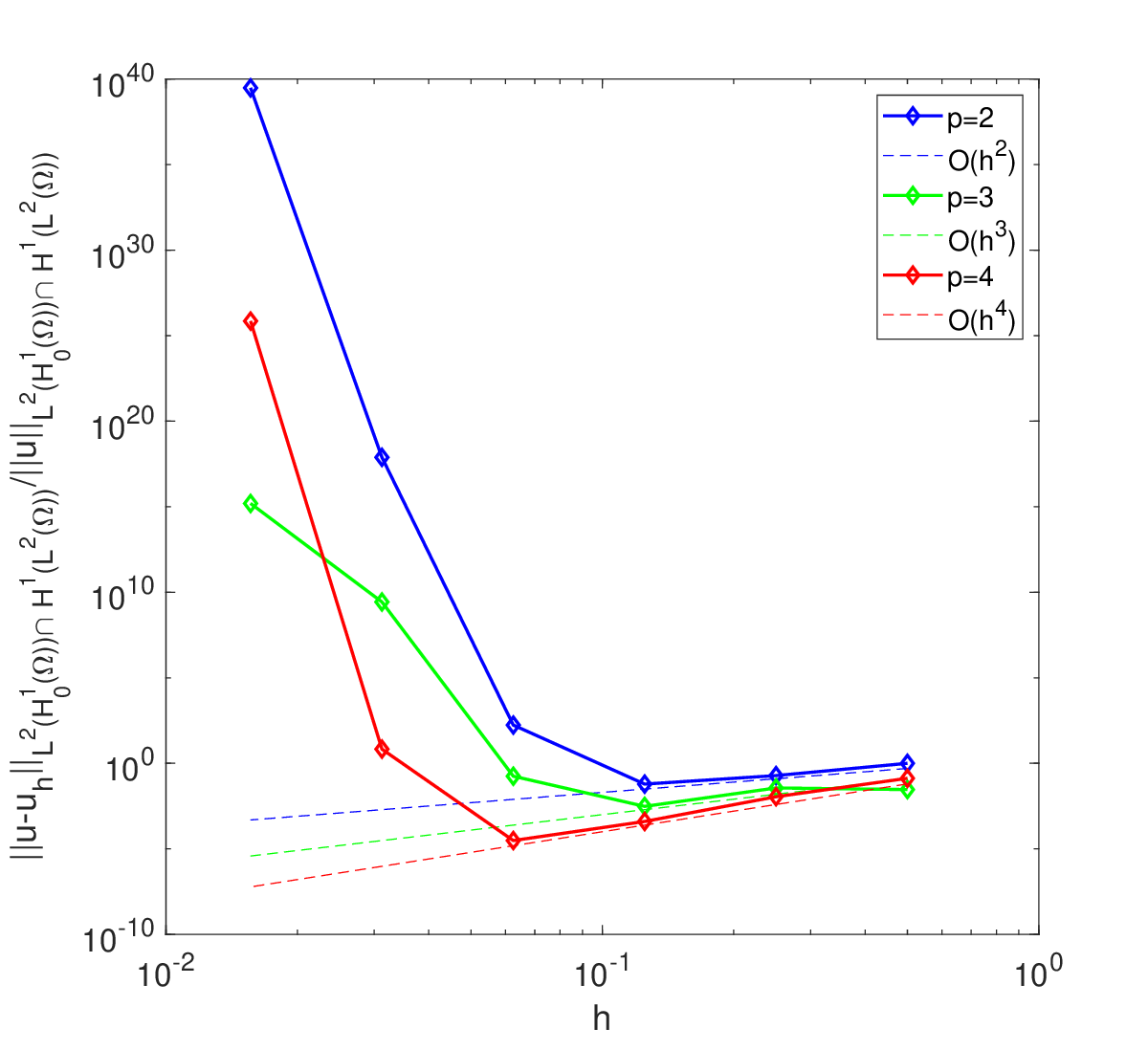}
		\caption[Caption]{}
	\end{subfigure}
	\hspace{0.55cm}
	\begin{subfigure}{0.3\textwidth}
		\includegraphics[height=6cm,width=5cm]{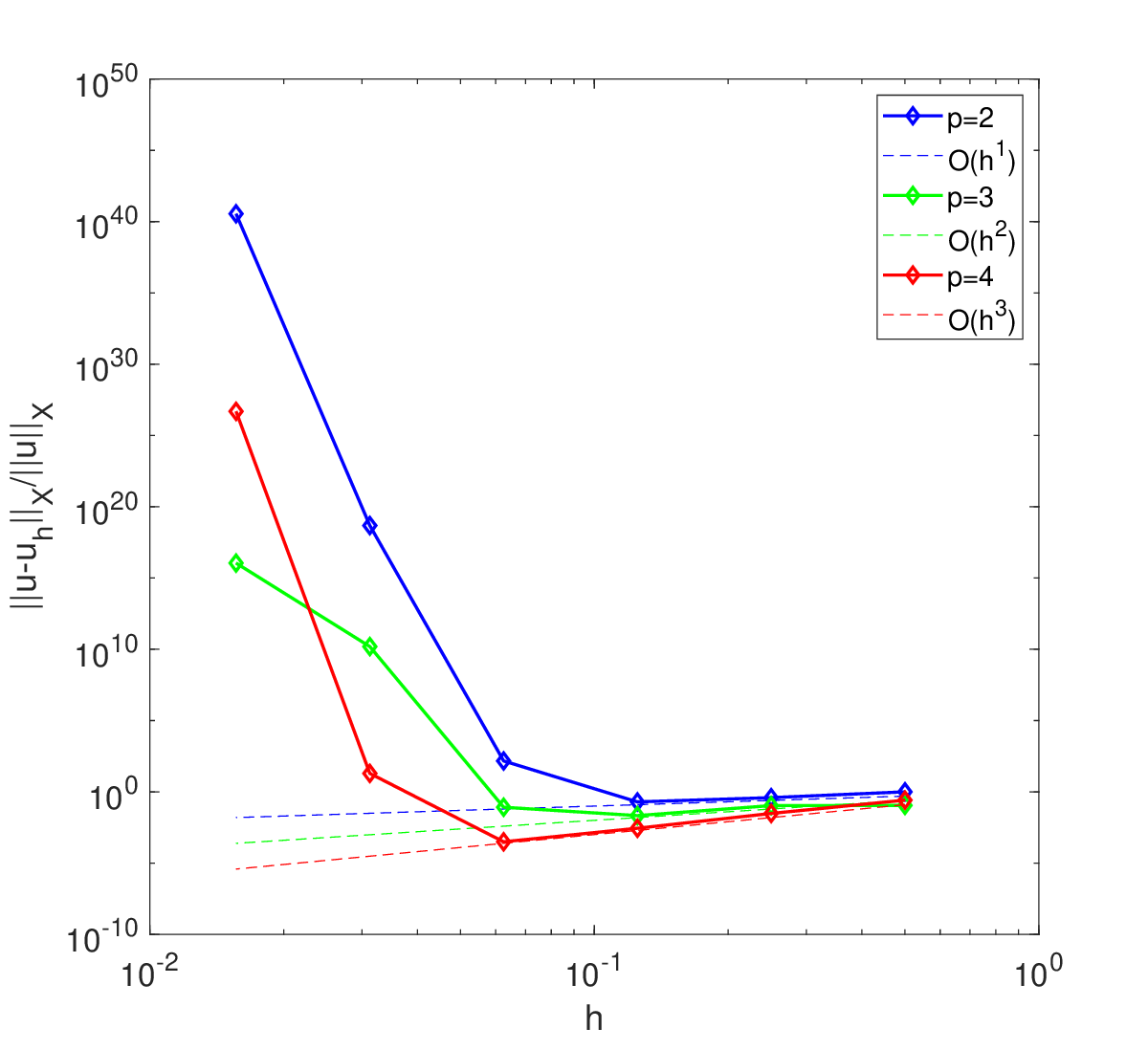}
		\caption[Caption]{}
	\end{subfigure}
	
	\caption{Relative errors in (a) $L^2(L^2(\Omega))$ norm, (b) $L^2(H^1_0(\Omega))\cap H^1(L^2(\Omega))$ norm and (c) $X$ norm for the without stabilization method with splines of maximum regularity in space and $C^{p-2}$ regularity in time direction for Square domain.}
	\label{nostabCP_P_2}
\end{figure}

\begin{figure}[htbp]
	\begin{subfigure}{0.3\textwidth}
		\includegraphics[height=6cm,width=5cm]{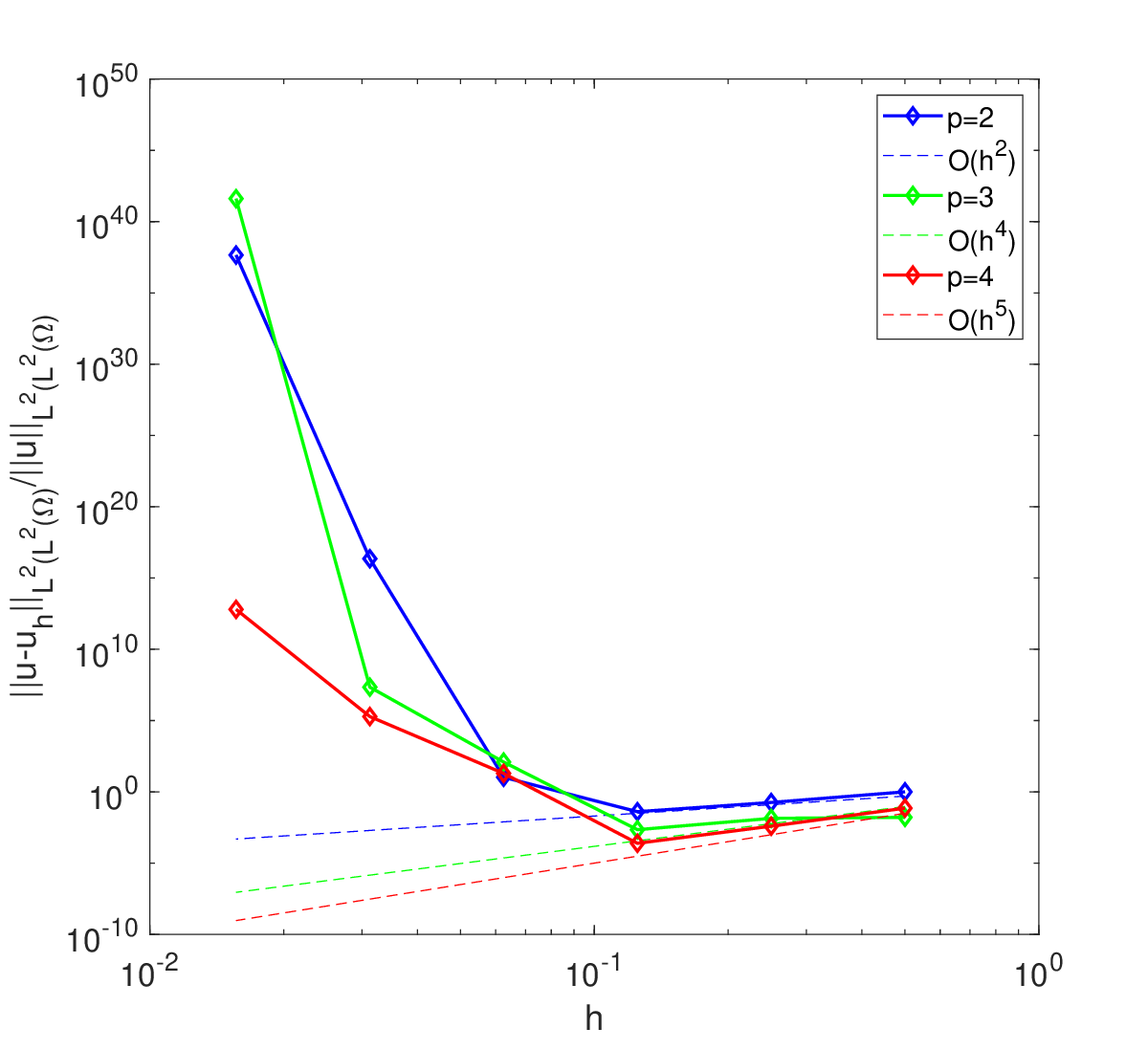}
		\caption[Caption]{}
		\label{}
	\end{subfigure}
	\hspace{0.55cm}
	\begin{subfigure}{0.3\textwidth}
		\includegraphics[height=6cm,width=5cm]{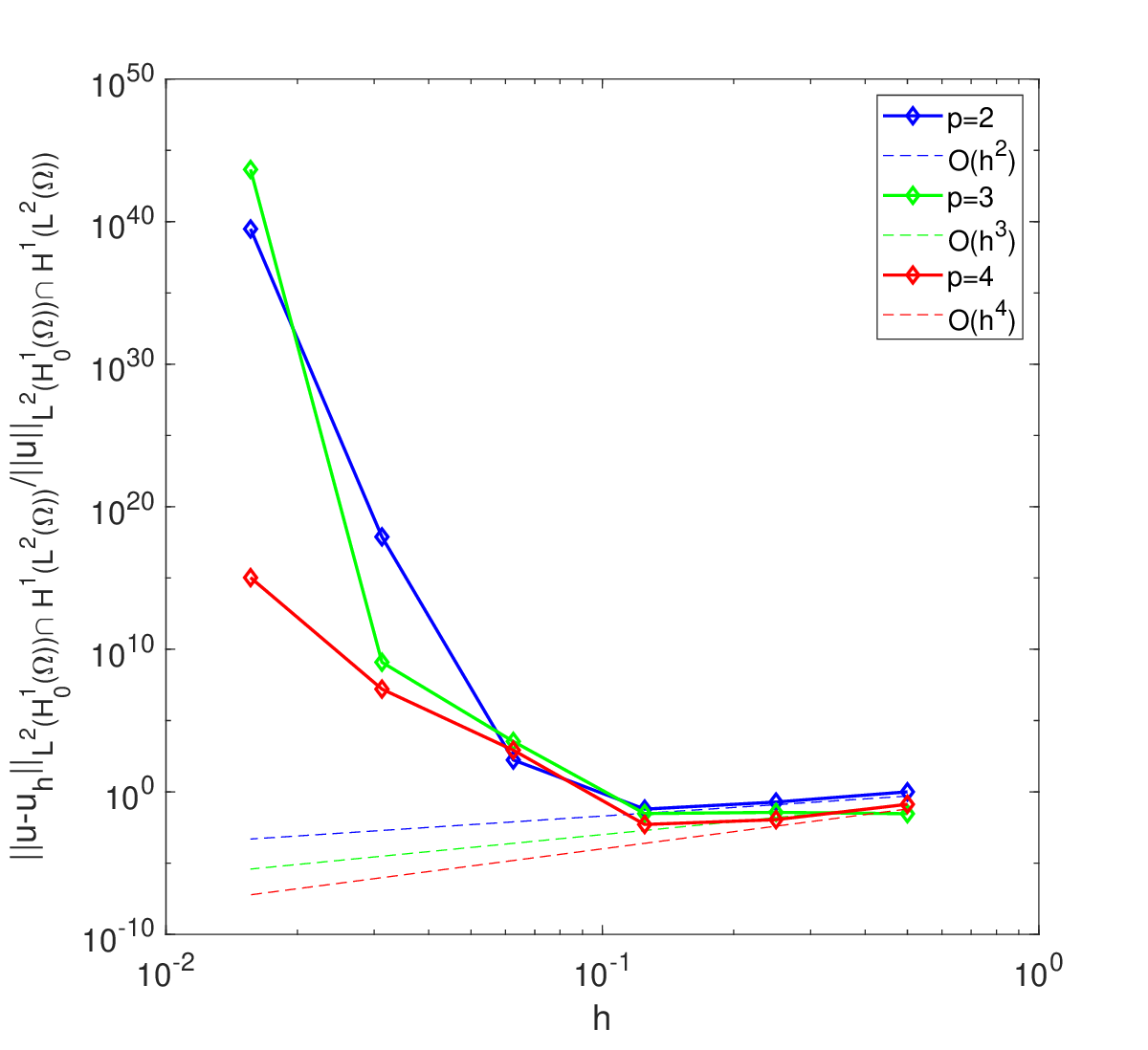}
		\caption[Caption]{}
	\end{subfigure}
	\hspace{0.55cm}
	\begin{subfigure}{0.3\textwidth}
		\includegraphics[height=6cm,width=5cm]{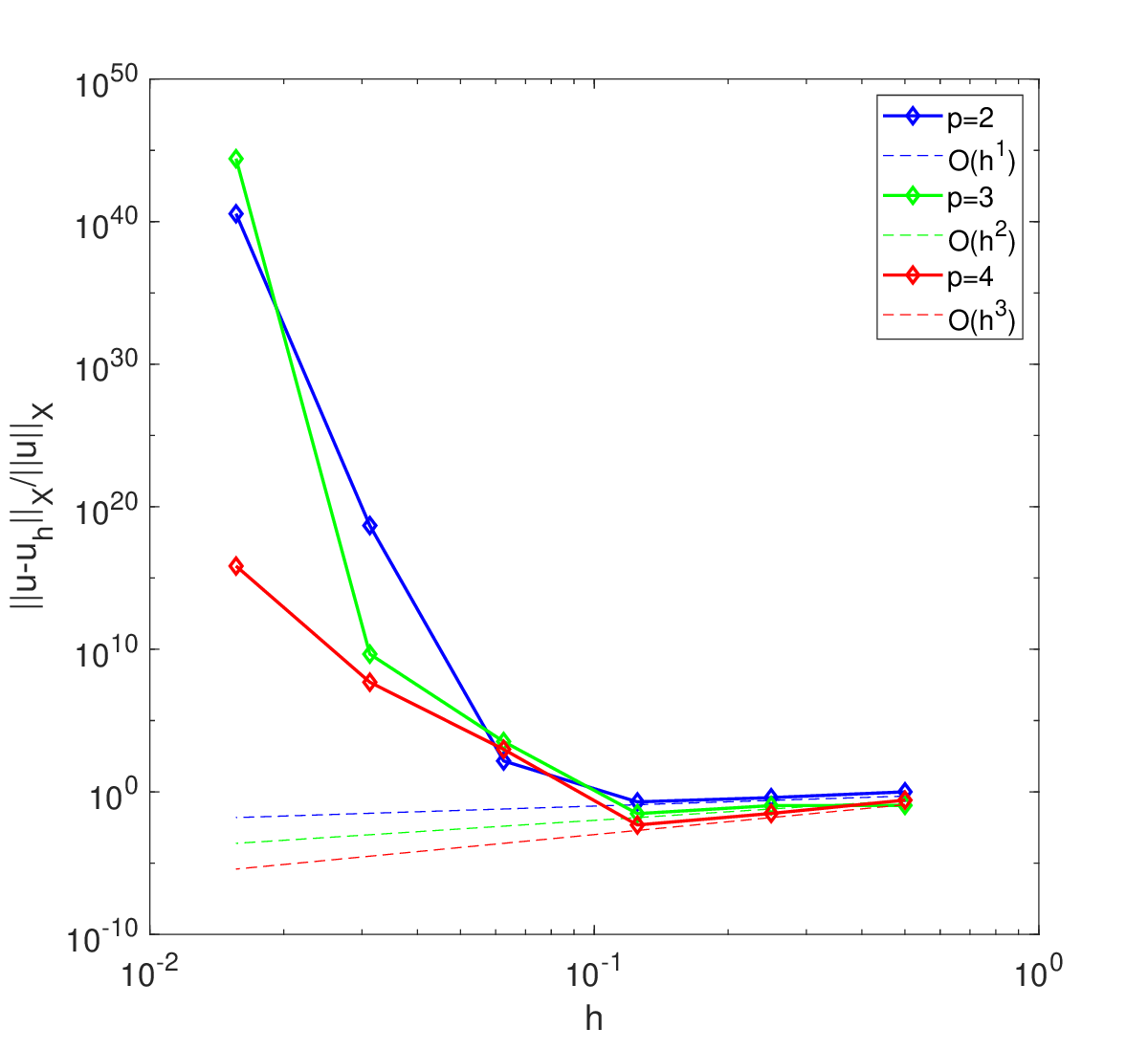}
		\caption[Caption]{}
	\end{subfigure}
	
	\caption{Relative errors in (a) $L^2(L^2(\Omega))$ norm, (b) $L^2(H^1_0(\Omega))\cap H^1(L^2(\Omega))$ norm and (c) $X$ norm for the without stabilization method with splines of maximum regularity in space and $C^0$ regularity in time direction for Square domain.}
	\label{nostabC_0}
\end{figure}
\noindent{\textbf{Example 2. Performance of solver:}}
To evaluate the performance of solver proposed in Section \ref{sec:solver} for solving the linear system that raised from the IgA-stab formulation \eqref{stabwave}, here we perform one numerical experiment. We consider the problem \eqref{eqbiwave} on a non-trivial geometry namely a quarter of annulus domain which is displayed in Figure \ref{ring_domain}. We choose the problem data such that the exact solution is $u(x,y,t)=t^2x^2y^2(x^2+y^2-1)^2(x^2+y^2-4)^2$. In Table \ref{tablefortime}, we report the total computational time required for solving the linear system \eqref{stabms} using Algorithm \ref{al:solving} on meshes of size $h_s=h_t=h=\frac{1}{16},\ \frac{1}{32},\ \frac{1}{64},\ \frac{1}{128},\ \frac{1}{256}$ with splines of degree $p=2,\ 3,\ 4.$ From Table \ref{tablefortime}, we observe that for $h=\frac{1}{256}$, the computational time grows with factors 11.47, 14.02 and 13.52 for polynomial degrees 2, 3 and 4 respectively, which are smaller than the factor 16 resulting from the complexity $\mathcal{O}(N_{dof}^{4/3})$. The results are in agreement with the theoretical computational cost discussed in Section \ref{sec:solver}.
	\begin{table}[h]
			\caption{Computational time for ring shaped domain}
			\label{tablefortime}
	\centering
	\begin{tabular}{ccccccc}
	
		\toprule
		\multicolumn{1}{c}{} & \multicolumn{2}{c}{$p=2$}&\multicolumn{2}{c}{$p=3$}&\multicolumn{2}{c}{$p=4$} \\
		\cmidrule(lr){2-3} \cmidrule(lr){4-5} \cmidrule{6-7}
		$h$ & $N_{dofs}$ &  Time(s)  & $N_{dofs}$ &  Time(s)  & $N_{dofs}$ & Time(s) \\
		\midrule
		$\frac{1}{16}$ & 3332 & 0.0713 & 4050&0.0492&4864 &0.0262\\
		$\frac{1}{32}$ & 29700 &0.2614& 32674& 0.5062&35840& 0.7103\\
		$\frac{1}{64}$ & 249860 & 2.9349 & 261954&4.9037&274432& 9.4208\\
		$\frac{1}{128}$ & 2048004 & 32.9743& 2096770&57.2154&2146304&1.0048e+02 \\
		$\frac{1}{256}$ & 16580612 & 3.7833e+02& 16776450& 8.0257e+02&16973824&1.3591e+03 \\
		\hline 
	\end{tabular}
\end{table}\\\\
\noindent{\textbf{Example 3. Accuracy of the IgA-stab formulation:}}
In this final example, we compare the performance of the proposed IgA-stab formulation \eqref{stabwave} with the FEM-stab formulation \eqref{stabwavefem} for the biharmonic wave equation \eqref{eqbiwave}. We consider the one-dimensional spatial domain $\Omega = (0,1)$ with a manufactured solution
\begin{equation*}
    u(x,t)=t^2\sin^2(\pi x).
\end{equation*}
We obtain the numerical solution of \eqref{eqbiwave} using the IgA-stab formulation \eqref{stabwave} for splines of maximal regularity and FEM-stab formulation \eqref{stabwavefem} for splines of $C^1$ regularity in space and $C^0$ regularity in time with splines of degree $p=2,\ 3,\ 4,\ 5$. To ensure a fair comparison for each polynomial degree, we fix the total number of degrees of freedom to $N_{dof} = 8400$ for both methods. This is achieved by appropriately choosing the mesh sizes $h_s$ and $h_t$ for each $p$ which satisfy $h_t\approx h_s$. In Figure \ref{compare}, we display the relative errors in $L^2(\Omega)$, $H^1_0(\Omega)$ and $H^2_0(\Omega)$ norm obtained at final-time $T$. From Figure \ref{compare}, we observe that for all polynomial degrees $p \geq 2$ and across all three norms, the IgA-stab formulation yields smaller errors than the FEM-stab formulation. This advantage becomes more pronounced as $p$ increases, demonstrating that the proposed IgA-stab formulation \eqref{stabwave} achieves better accuracy per degree of freedom compared to the FEM-stab formulation \eqref{stabwavefem}.
\begin{figure}[htbp]
	\begin{subfigure}{0.3\textwidth}
		\includegraphics[height=6cm,width=5cm]{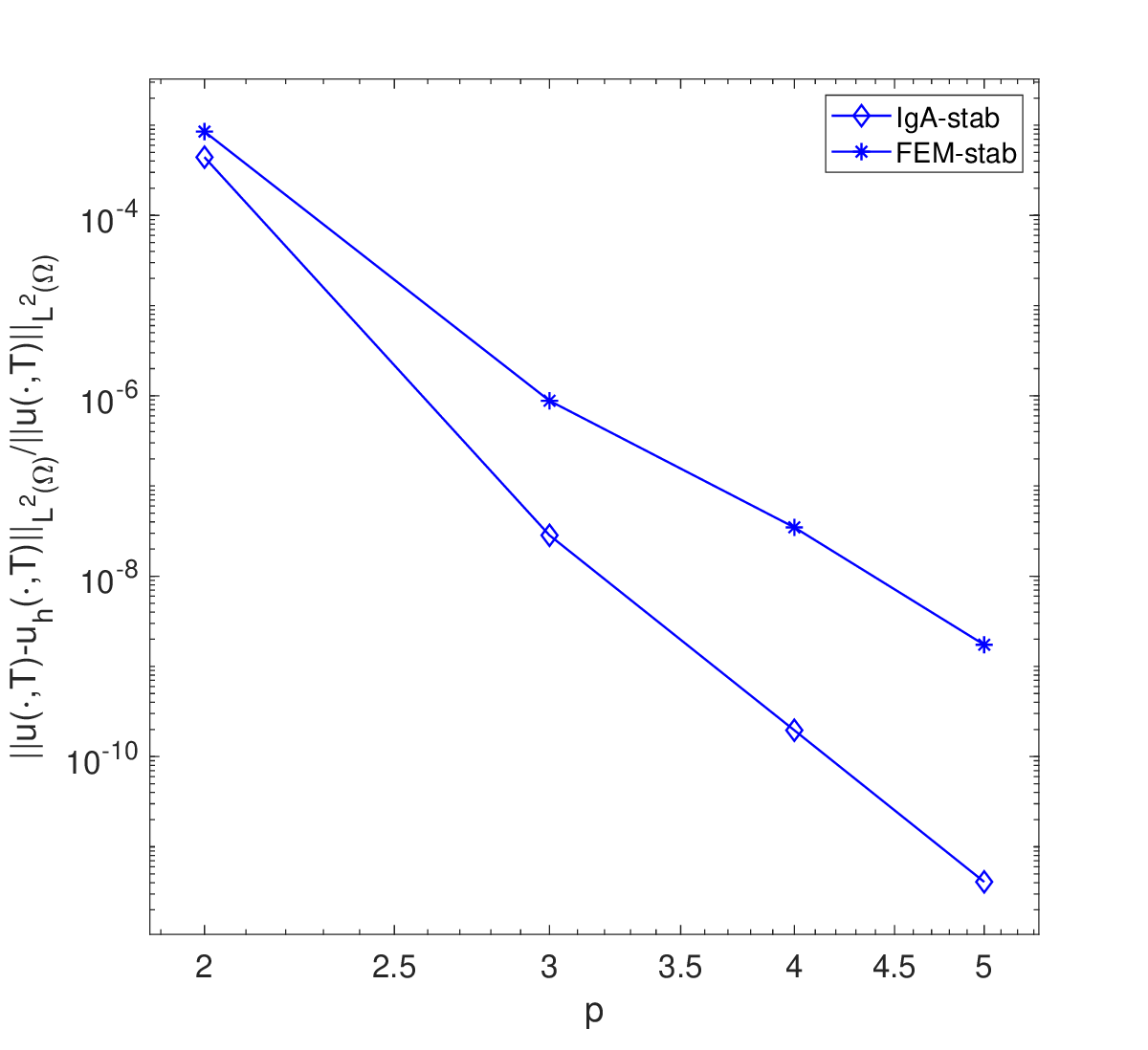}
		\caption[Caption]{}
		\label{}
	\end{subfigure}
	\hspace{0.55cm}
	\begin{subfigure}{0.3\textwidth}
		\includegraphics[height=6cm,width=5cm]{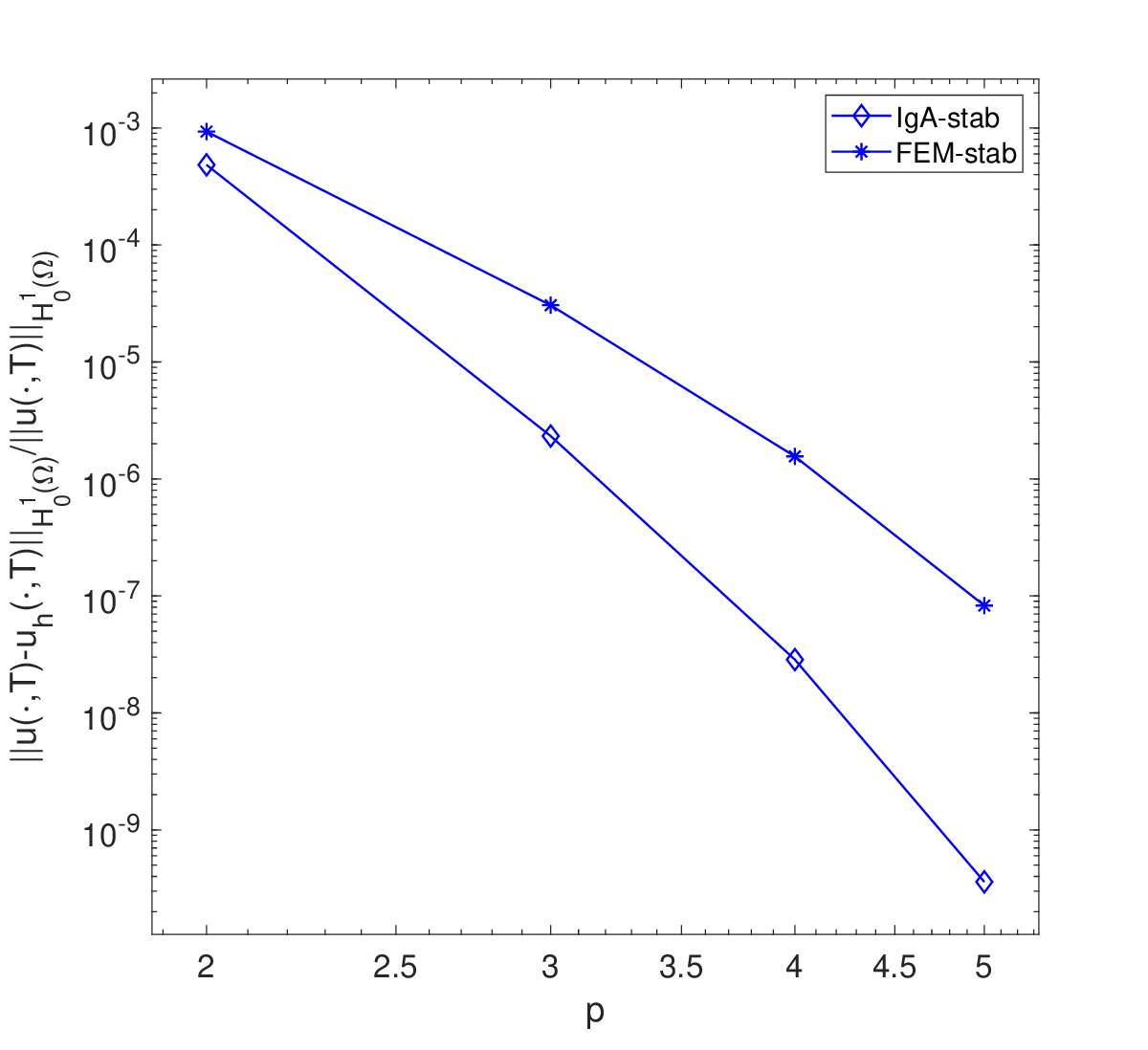}
		\caption[Caption]{}
	\end{subfigure}
	\hspace{0.55cm}
	\begin{subfigure}{0.3\textwidth}
		\includegraphics[height=6cm,width=5cm]{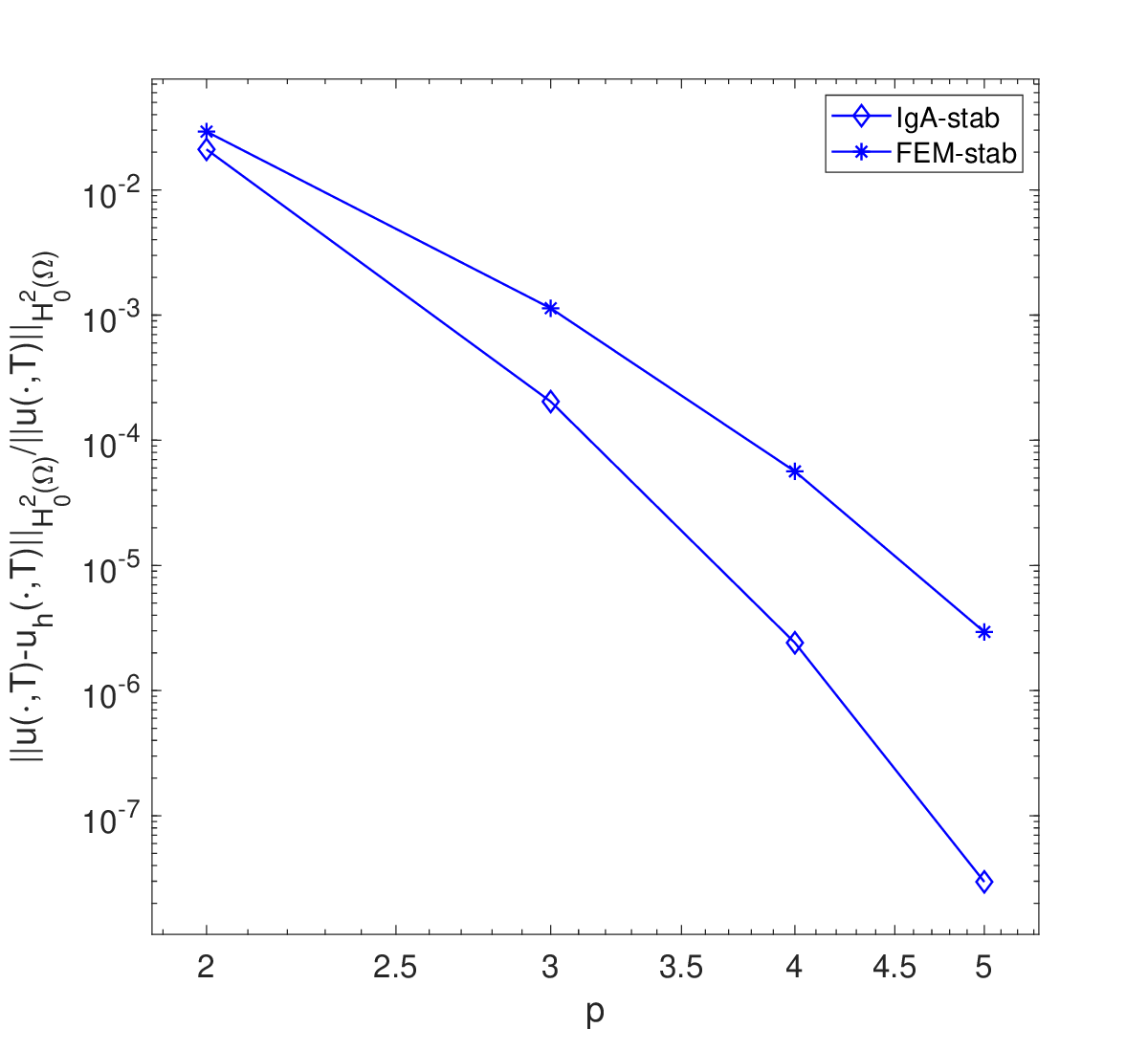}
		\caption[Caption]{}
	\end{subfigure}
	
	\caption{Relative errors in (a) $L^2(\Omega)$ norm, (b) $H^1_0(\Omega)$ norm and (c) $H^2_0(\Omega)$ norm for the IgA stabilization and FEM stabilization.}
	\label{compare}
\end{figure}

 \section{Conclusions}\label{sec:conclusions}
In this work, we have analyzed the space-time isogeometric method for biharmonic wave equation \eqref{eqbiwave}. We proved the well-posedness of the space-time variational formulation when $f\in L^2(L^2(\Omega))$ by exploiting the Fourier expansion of the trial and test functions. By considering the smooth spline functions, we obtained a conforming space-time isogeometric discretization of \eqref{eqbiwave}. Upon analyzing the space-time discrete variational formulation, we observed that the discrete scheme is table if and only if the mesh sizes satisfy a CFL condition. To overcome this restriction, we proposed an unconditionally stable space-time isogeometric discretization of \eqref{eqbiwave}, inspired by the work of Fraschini et al. \cite{Fraschini2024}. Additionally, we also provided an efficient solver to solve the matrix system raised from the space-time discrete formulation \eqref{stabwave}. Our numerical results shows that, when considering the maximal regularity splines in temporal direction, the space-time isogeometric discrete formulation \eqref{stabwave} is stable in the absence of any CFL condition.  Moreover, the proposed solver is numerically efficient as it is seen in the numerical test , the computational cost is less than $\mathcal{O}(N_{dof}^{4/3})$. Also, the proposed stabilized space-time isogeometric formulation is more accurate than the stabilized space-time FEM formulation per number of degrees of freedom. Though our proposed space-time isogeometric discrete formulation \eqref{stabwave} is unconditionally stable, the existence-uniqueness results and the theoretical convergent estimates are yet to be discovered. Additionally, we observed in our numerical experiments, that the stabilized space-time isogeometric  formulation \eqref{stabwave} and the stabilized space-time FEM discrete formulation \eqref{stabwavefem} are stable respectively for splines of maximal and $C^0$ regularity in temporal direction. For the intermediate case, i.e. for splines of regularity $C^{k}, 1\leq k\leq p-2$, we observed that the formulations \eqref{stabwave} and \eqref{stabwavefem} show instability. Hence, the enhancement of the proposed stabilized formulation for this intermediate case is left as a future work. \\\\
 {\noindent{\bf Acknowledgment:}} The authors gratefully acknowledge Dr. Monica Montardini for her valuable discussions and insightful suggestions regarding Section \ref{sec:solver}. The first author would also like to thank the Department of Science and Technology (DST), New Delhi, India, for providing the INSPIRE Fellowship that supported this research. The second author acknowledges financial support from National Board for Higher Mathematics (NBHM) India, under grant no. 02011/43/2025/NBHM(R.P)/R$\&$D II/16568.\\\\
\FloatBarrier
	\bibliographystyle{unsrt}
\bibliography{first_draft_ref.bib}

@article{thinplates,
	title = {Analytical solutions for {E}uler–{B}ernoulli beam on visco‐elastic foundation subjected to moving load},
	volume = {37},
	ISSN = {1096-9853},
	url = {http://dx.doi.org/10.1002/nag.1135},
	DOI = {10.1002/nag.1135},
	number = {8},
	journal = {International Journal for Numerical and Analytical Methods in Geomechanics},
	publisher = {Wiley},
	author = {Basu,  D. and Kameswara Rao,  N. S. V.},
	year = {2012},
	pages = {945–960}
}

@article{vibrationbeam,
	author = {Li, B. and Fairweather, G. and Bialecki, B.},
	title = {Discrete-Time Orthogonal Spline Collocation Methods for Vibration Problems},
	journal = {SIAM Journal on Numerical Analysis},
	volume = {39},
	number = {6},
	pages = {2045-2065},
	year = {2002},
	doi = {10.1137/S0036142900348729},
}

@article{waterwaves,
	 title={Resonant interactions between waves. {T}he case of discrete oscillations},
	  volume={20},
	   DOI={10.1017/S0022112064001355},
	    number={3}, 
	    journal={Journal of Fluid Mechanics}, author={Bretherton, F. P.},
	     year={1964}, pages={457–479}
}

@article{LAZER1987243,
	title = {Large scale oscillatory behaviour in loaded asymmetric systems},
	journal = {Annales de l'Institut Henri Poincaré C, Analyse non linéaire},
	volume = {4},
	number = {3},
	pages = {243-274},
	year = {1987},
	issn = {0294-1449},
	doi = {https://doi.org/10.1016/S0294-1449(16)30368-7},
	url = {https://www.sciencedirect.com/science/article/pii/S0294144916303687},
	author = {A.C. Lazer and P.J. Mckenna},
}

@article{morley,
	title = {Morley {F}{E}{M} for the fourth-order nonlinear reaction-diffusion problems},
	journal = {Computers \& Mathematics with Applications},
	volume = {99},
	pages = {229-245},
	year = {2021},
	issn = {0898-1221},
	doi = {https://doi.org/10.1016/j.camwa.2021.08.010},
	url = {https://www.sciencedirect.com/science/article/pii/S0898122121002947},
	author = {P. Danumjaya and A. K. Pany and A. K. Pani},
}

@article {dGFEM,
	AUTHOR = {Georgoulis, E. H. and Virtanen, J. M.},
	TITLE = {Adaptive discontinuous {G}alerkin approximations to fourth
	order parabolic problems},
	FJOURNAL = {Mathematics of Computation},
	VOLUME = {84},
	YEAR = {2015},
	NUMBER = {295},
	PAGES = {2163--2190},
	ISSN = {0025-5718,1088-6842},
	MRCLASS = {65M60 (65M15 65M50)},
	MRNUMBER = {3356023},
	MRREVIEWER = {Manfred\ Dobrowolski},
	DOI = {10.1090/mcom/2936},
	URL = {https://doi.org/10.1090/mcom/2936},
}

@article {C0interior,
	AUTHOR = {Gudi, T. and Gupta, H. S.},
	TITLE = {A fully discrete {$C^0$} interior penalty {G}alerkin
	approximation of the extended {F}isher-{K}olmogorov equation},
	
	FJOURNAL = {Journal of Computational and Applied Mathematics},
	VOLUME = {247},
	YEAR = {2013},
	PAGES = {1--16},
	ISSN = {0377-0427,1879-1778},
	MRCLASS = {65M60 (65M12 65M15)},
	MRNUMBER = {3023298},
	MRREVIEWER = {Saulo\ Pomponet\ Oliveira},
	DOI = {10.1016/j.cam.2012.12.019},
	URL = {https://doi.org/10.1016/j.cam.2012.12.019},
}

@article{DAS202452,
	title = {A unified mixed finite element method for fourth-order time-dependent problems using biorthogonal systems},
	journal = {Computers \& Mathematics with Applications},
	volume = {165},
	pages = {52-69},
	year = {2024},
	issn = {0898-1221},
	doi = {https://doi.org/10.1016/j.camwa.2024.04.013},
	url = {https://www.sciencedirect.com/science/article/pii/S0898122124001718},
	author = {A. Das and B. P. Lamichhane and N. Nataraj},
}

@article{C1bihar,
	title = {${C}^1$-conforming variational discretization of the biharmonic wave equation},
	journal = {Computers \& Mathematics with Applications},
	volume = {119},
	pages = {208-219},
	year = {2022},
	issn = {0898-1221},
	doi = {https://doi.org/10.1016/j.camwa.2022.06.005},
	url = {https://www.sciencedirect.com/science/article/pii/S0898122122002498},
	author = {M. Bause and M. Lymbery and K. Osthues},
}

@article{HE20131,
	title = {Analysis of mixed finite element methods for fourth-order wave equations},
	journal = {Computers \& Mathematics with Applications},
	volume = {65},
	number = {1},
	pages = {1-16},
	year = {2013},
	issn = {0898-1221},
	doi = {https://doi.org/10.1016/j.camwa.2012.10.002},
	url = {https://www.sciencedirect.com/science/article/pii/S0898122112006037},
	author = {S. He and H. Li and Y. Liu},
}

@article{nataraj,
url = {https://doi.org/10.1515/cmam-2025-0002},
title = {Semi- and Fully-Discrete Analysis of Lowest-Order Nonstandard Finite Element Methods for the Biharmonic Wave Problem},
title = {},
author = {N. Nataraj and R. Ruiz-Baier and A. Yousuf},
pages = {921--948},
volume = {25},
number = {4},
journal = {Computational Methods in Applied Mathematics},
doi = {doi:10.1515/cmam-2025-0002},
year = {2025},
lastchecked = {2026-03-07}
}

@article{HE2023333,
	title = {An energy-conserving finite element method for nonlinear fourth-order wave equations},
	journal = {Applied Numerical Mathematics},
	volume = {183},
	pages = {333-354},
	year = {2023},
	issn = {0168-9274},
	doi = {https://doi.org/10.1016/j.apnum.2022.09.011},
	url = {https://www.sciencedirect.com/science/article/pii/S0168927422002513},
	author = {M. He and J. Tian and P. Sun and Z. Zhang},
	
}

@article{invineq2,
author = {Bazilevs, Y. and Beir\~{a}o da Veiga, L. and Cottrell, J. A. and Hughes, T. J. R. and Sangalli, G.},
title = {ISOGEOMETRIC ANALYSIS: APPROXIMATION, STABILITY AND ERROR ESTIMATES FOR h-REFINED MESHES},
journal = {Mathematical Models and Methods in Applied Sciences},
volume = {16},
number = {07},
pages = {1031-1090},
year = {2006},
}

@article{TAO2021113230,
	title = {A discontinuous {G}alerkin method and its error estimate for nonlinear fourth-order wave equations},
	journal = {Journal of Computational and Applied Mathematics},
	volume = {386},
	pages = {113230},
	year = {2021},
	issn = {0377-0427},
	doi = {https://doi.org/10.1016/j.cam.2020.113230},
	url = {https://www.sciencedirect.com/science/article/pii/S0377042720305215},
	author = {Q. Tao and Y. Xu and C. Shu},
	
}

@article{IgAIntro,
	title = {Isogeometric analysis:{ C}{A}{D}, finite elements, {N}{U}{R}{B}{S}, exact geometry and mesh refinement},
	journal = {Computer Methods in Applied Mechanics and Engineering},
	volume = {194},
	number = {39},
	pages = {4135-4195},
	year = {2005},
	issn = {0045-7825},
	author = {T. J. R. Hughes and J. A. Cottrell and Y. Bazilevs},
}

@article{IgAhigh,
	title = {Isogeometric Analysis and error estimates for high order partial differential equations in fluid dynamics},
	journal = {Computers \& Fluids},
	volume = {102},
	pages = {277-303},
	year = {2014},
	issn = {0045-7930},
	doi = {https://doi.org/10.1016/j.compfluid.2014.07.002},
	url = {https://www.sciencedirect.com/science/article/pii/S0045793014002813},
	author = {A. Tagliabue and L. Dedè and A. Quarteroni},
}

@article{Sogn2023,
	title = {Multigrid solvers for isogeometric discretizations of the second biharmonic problem},
	volume = {33},
	ISSN = {1793-6314},
	url = {http://dx.doi.org/10.1142/S0218202523500422},
	DOI = {10.1142/s0218202523500422},
	number = {09},
	journal = {Mathematical Models and Methods in Applied Sciences},
	publisher = {World Scientific Pub Co Pte Ltd},
	author = {Sogn,  J. and Takacs,  S.},
	year = {2023},
	pages = {1803–1828}
}

@article{MANNI2023116314,
	title = {Outlier-free spline spaces for isogeometric discretizations of biharmonic and polyharmonic eigenvalue problems},
	journal = {Computer Methods in Applied Mechanics and Engineering},
	volume = {417},
	pages = {116314},
	year = {2023},
	note = {A Special Issue in Honor of the Lifetime Achievements of T. J. R. Hughes},
	issn = {0045-7825},
	doi = {https://doi.org/10.1016/j.cma.2023.116314},
	url = {https://www.sciencedirect.com/science/article/pii/S0045782523004383},
	author = {C. Manni and E. Sande and H. Speleers},
	
}

@article{Meng2024,
	title = {The Convergence Analysis of a Class of Stabilized Semi-Implicit Isogeometric Methods for the {C}ahn-{H}illiard Equation},
	volume = {102},
	ISSN = {1573-7691},
	url = {http://dx.doi.org/10.1007/s10915-024-02753-5},
	DOI = {10.1007/s10915-024-02753-5},
	pages={26},
	journal = {Journal of Scientific Computing},
	publisher = {Springer Science and Business Media LLC},
	author = {Meng,  X. and Qin,  Y. and Hu,  G.},
	year = {2025},
}

@incollection{Steinbachreview,
	author    = "O. Steinbach and H. Yang",
	booktitle   = " Space-Time Methods:Applications to Partial Differential Equations ",
	editors="O. Steinbach and U. Langer",
	title= "{S}pace-time finite element methods for parabolic evolution equations: discretization, a posteriori error estimation, adaptivity and solution",
	volume="25",
	publisher = "De Gruyter",
	address = "Berlin, Boston",
	pages = "207-248",
	year      = "2019"
}

@article{LANGER2016342,
	title = {Space–time isogeometric analysis of parabolic evolution problems},
	journal = {Computer Methods in Applied Mechanics and Engineering},
	volume = {306},
	pages = {342-363},
	year = {2016},
	issn = {0045-7825},
	author = {U. Langer and S. E. Moore and M. Neumüller},
}

@article{leastsquare,
	title={Space-time least-squares isogeometric method and efficient solver for parabolic problems},
	journal={Mathematics of Computation},
	volume={89},
	number={323},
	year={2020},
	pages={1193-2227},
	author={M. Montardini and M. Negri and G. Sangalli and M. Tani}
}

@article{stigasteinbach,
	title = {An efficient solver for space–time isogeometric {G}alerkin methods for parabolic problems},
	journal = {Computers \& Mathematics with Applications},
	volume = {80},
	number = {11},
	pages = {2586-2603},
	year = {2020},
	
	issn = {0898-1221},
	doi = {https://doi.org/10.1016/j.camwa.2020.09.014},
	url = {https://www.sciencedirect.com/science/article/pii/S0898122120303709},
	author = {G. Loli and M. Montardini and G. Sangalli and M. Tani},
}

@article{Henning2022,
	title = {An ultraweak space-time variational formulation for the wave equation: Analysis and efficient numerical solution},
	volume = {56},
	ISSN = {2804-7214},
	url = {http://dx.doi.org/10.1051/m2an/2022035},
	DOI = {10.1051/m2an/2022035},
	number = {4},
	journal = {ESAIM: Mathematical Modelling and Numerical Analysis},
	publisher = {EDP Sciences},
	author = {Henning,  J. and Palitta,  D. and Simoncini,  V. and Urban,  K.},
	year = {2022},
	pages = {1173–1198}
}

@article{coer,
	author = {Steinbach, O. and Zank, M.},
	year = {2020},
	pages = {154-194},
	title = {Coercive space-time finite element methods for initial boundary value problems},
	volume = {52},
	journal = {ETNA - Electronic Transactions on Numerical Analysis},
	doi = {10.1553/etna_vol52s154}
}

@article{sarathes,
	title={Stability of space-time isogeometric methods for wave propagation problems},
	author={Fraschini, S.},
	journal={arXiv preprint arXiv:2303.15460},
	note={Master’s thesis. Universit`a degli Studi di Pavia},
	year={2021}

}

@article {Ferrari2025,
    AUTHOR = {Ferrari, M. and Fraschini, S.},
     TITLE = {Stability of conforming space-time isogeometric methods for
              the wave equation},
  JOURNAL = {Mathematics of Computation},
    VOLUME = {95},
      YEAR = {2025},
    NUMBER = {358},
     PAGES = {683--719},
      ISSN = {0025-5718,1088-6842},
   MRCLASS = {65M60 (15A12 15B05 65L60)},
  MRNUMBER = {4999540},
       DOI = {10.1090/mcom/4062},
       URL = {https://doi.org/10.1090/mcom/4062},
}

@Inbook{Steinbach2019,
	author="Steinbach, O.
	and Zank, M.",
	editor="Apel, T.
	and Langer, U.
	and Meyer, A.
	and Steinbach, O.",
	title="A Stabilized Space--Time Finite Element Method for the Wave Equation",
	bookTitle="Advanced Finite Element Methods with Applications: Selected Papers from the 30th Chemnitz Finite Element Symposium 2017",
	year="2019",
	publisher="Springer International Publishing",
	address="Cham",
	pages="341--370",
	isbn="978-3-030-14244-5",
	doi="10.1007/978-3-030-14244-5_17",
	url="https://doi.org/10.1007/978-3-030-14244-5_17"
}

@article{Bignardi2025,
	title = {A space–time continuous and coercive formulation for the wave equation},
	volume = {157},
	ISSN = {0945-3245},
	url = {http://dx.doi.org/10.1007/s00211-025-01478-3},
	DOI = {10.1007/s00211-025-01478-3},
	number = {4},
	journal = {Numerische Mathematik},
	publisher = {Springer Science and Business Media LLC},
	author = {Bignardi,  P. and Moiola,  A.},
	year = {2025},
	pages = {1211–1258}
}

@incollection {arxivwave,
    AUTHOR = {L\"oscher, R. and Steinbach, O. and Zank, M.},
     TITLE = {Numerical results for an unconditionally stable space-time
              finite element method for the wave equation},
 BOOKTITLE = {Domain decomposition methods in science and engineering
              {XXVI}},
    SERIES = {Lecture Notes in Computational Science and Engineering},
    VOLUME = {145},
     PAGES = {625--632},
 PUBLISHER = {Springer, Cham},
      YEAR = {2022},
      ISBN = {978-3-030-95024-8; 978-3-030-95025-5},
   MRCLASS = {},
  MRNUMBER = {4703898},
       DOI = {10.1007/978-3-030-95025-5\_68},
       URL = {https://doi.org/10.1007/978-3-030-95025-5_68},
}

@article{arxivwave2,
      title={Inf-sup stable space-time discretization of the wave equation based on a first-order-in-time variational formulation}, 
      author={M. Ferrari and I. Perugia and E. Zampa},
      year={2025},
      journal={arXiv preprint arXiv:2506.05886},
      primaryClass={math.NA},
      url={https://arxiv.org/abs/2506.05886}, 
}

@article{Ferrarihamil,
	author = {{M. Ferrari} and {S. Fraschini} and {G. Loli} and {I. Perugia}},
	title = {Unconditionally stable space–time isogeometric discretization for the wave equation in {H}amiltonian formulation},
	DOI= "10.1051/m2an/2025056",
	url= "https://doi.org/10.1051/m2an/2025056",
	journal = {ESAIM: M2AN},
	year = 2025,
	volume = 59,
	number = 5,
	pages = "2447-2490",
}

@inproceedings{Zank2021,
title = "Higher-order space-time continuous Galerkin methods for the wave equation",
author = "M. Zank",
year = "2021",
doi = "10.23967/wccm-eccomas.2020.167",
language = "English",
volume = "700",
series = "WCCM-ECCOMAS Congress",
publisher = "Scipedia",
pages = "1--10",
booktitle = "14th WCCM-ECCOMAS Congress 2020",
}

@article{Fraschini2024,
	title = {An unconditionally stable space–time isogeometric method for the acoustic wave equation},
	volume = {169},
	ISSN = {0898-1221},
	url = {http://dx.doi.org/10.1016/j.camwa.2024.06.009},
	DOI = {10.1016/j.camwa.2024.06.009},
	journal = {Computers \& Mathematics with Applications},
	publisher = {Elsevier BV},
	author = {Fraschini,  S. and Loli,  G. and Moiola,  A. and Sangalli,  G.},
	year = {2024},
	pages = {205–222}
}

@article{Loli2025,
	title = {Space–time isogeometric analysis: a review with application to wave propagation},
	volume = {82},
	ISSN = {2281-7875},
	url = {http://dx.doi.org/10.1007/s40324-025-00391-x},
	DOI = {10.1007/s40324-025-00391-x},
	number = {4},
	journal = {SeMA Journal},
	publisher = {Springer Science and Business Media LLC},
	author = {Loli,  G. and Sangalli,  G.},
	year = {2025},
	pages = {633–644}
}

@Book{nurbsbook,
	title = {Spline Functions: Basic Theory},
	ISBN = {9780511618994},
	url = {http://dx.doi.org/10.1017/CBO9780511618994},
	DOI = {10.1017/cbo9780511618994},
	publisher = {Cambridge University Press},
	author = {Schumaker,  L.},
	year = {2007}, 
}

@article{invineq,
	author = {Takacs, S. and Takacs, T.},
	title = {Approximation error estimates and inverse inequalities for {B}-splines of maximum smoothness},
	journal = {Mathematical Models and Methods in Applied Sciences},
	volume = {26},
	number = {07},
	pages = {1411-1445},
	year = {2016},
	doi = {10.1142/S0218202516500342},
}

@article{geopde,
	title = {A new design for the implementation of isogeometric analysis in Octave and Matlab: Geo{P}{D}{E}s 3.0},
	journal = {Computers \& Mathematics with Applications},
	volume = {72},
	number = {3},
	pages = {523-554},
	year = {2016},
	issn = {0898-1221},
	doi = {https://doi.org/10.1016/j.camwa.2016.05.010},
	url = {https://www.sciencedirect.com/science/article/pii/S0898122116302681},
	author = {R. Vázquez},
}

@misc{github,
	author = {S. Fraschini and G. Loli and A. Moiola and G. Sangali},
	title = {XTIgA-Waves},
	year =2023,
	 howpublished = {\url{https://github.com/XTIgA-Waves/XTIgA-Waves.git}},

}

@incollection{Spence2015,
  author    = {Spence, E. A.},
  title     = {``{W}hen All Else Fails, Integrate by Parts'': An Overview of New and Old Variational Formulations for Linear Elliptic {PDE}s},
  booktitle = {Unified Transform for Boundary Value Problems: Applications and Advances},
  publisher = {Society for Industrial and Applied Mathematics},
  year      = {2014},
  editor    = {A. S. Fokas and B. Pelloni},
  series    = {Other Titles in Applied Mathematics},
  chapter   = {6},
  pages     = {93-159},
  doi       = {10.1137/1.9781611973822.ch6},
  url       = {https://epubs.siam.org/doi/10.1137/1.9781611973822.ch6}
}

@unpublished{Moiola2021Scattering,
  author = {Moiola, A.},
  title = {Scattering of time-harmonic acoustic waves: {H}elmholtz equation, boundary integral equations and {B}{E}{M}},
  year = {2021},
  note = {Classnotes},
  howpublished = {Lecture notes}
}

@article{Chaudhary2026,
  author = {Chaudhary, S. and Chauhan, S. and Montardini, M.},
  title = {Space-time isogeometric method for a nonlocal parabolic problem},
  journal = {Advances in Computational Mathematics},
  year = {2026},
  volume = {52},
  pages = {12},
  doi = {10.1007/s10444-026-10285-9},
  url = {https://link.springer.com/article/10.1007/s10444-026-10285-9}
}

\end{document}